%% file: Extragradient_Armijo_SIAM_v2.tex
\documentclass[article]{siamart1116}

\input{Extragradient_Armijo_SIAM_shared}

\ifpdf
\hypersetup{
  pdftitle={\TheTitle},
  pdfauthor={\TheAuthors}
}
\fi

\begin{document}

\maketitle

% REQUIRED
\begin{abstract}
We propose dynamic sampled stochastic approximated (DS-SA) extragradient methods for stochastic variational inequalities (SVI) that are \emph{robust} with respect to an unknown Lipschitz constant $L$. We propose, to the best of our knowledge, the first provably convergent \emph{robust} SA \emph{method with variance reduction}, either for SVIs or stochastic optimization, assuming just an unbiased stochastic oracle and a large sample regime. This widens the applicability and improves, up to constants, the desired efficient acceleration of previous variance reduction methods, all of which still assume knowledge of $L$ (and, hence, are not robust against its estimate). Precisely, compared to the iteration and oracle complexities of $\mathcal{O}(\epsilon^{-2})$ of previous robust methods with a small stepsize policy, our robust method uses a DS-SA line search scheme obtaining the faster iteration complexity of $\mathcal{O}(\epsilon^{-1})$ with oracle complexity of $(\ln L)\mathcal{O}(d\epsilon^{-2})$ (up to log factors on $\epsilon^{-1}$) for a $d$-dimensional space. This matches, up to constants, the sample complexity of the sample average approximation estimator which does not assume additional problem information (such as $L$). Differently from previous robust methods for ill-conditioned problems, we allow an unbounded feasible set and an oracle with \emph{multiplicative noise} (MN) whose variance is not necessarily uniformly bounded. These properties are appreciated in our complexity estimates which depend only on $L$ and \emph{local} variances or forth moments at solutions $x^*$. The robustness and variance reduction properties of our DS-SA line search scheme come at the expense of \emph{nonmartingale-like dependencies} (NMD) due to the needed inner statistical estimation of a lower bound for $L$. In order to handle a NMD and a MN, our proofs rely on a novel localization argument based on empirical process theory. Additionally, we propose a second provable convergent method for SVIs over the wider class of H\"older continuous operators without any knowledge of its endogenous parameters.
\end{abstract}

% REQUIRED

% REQUIRED

\section{Introduction}

We consider methods for variational inequality problems where only a \emph{random perturbation} of the operator is available. In such problems, we have a closed convex set $X\subset\re^d$, a distribution $\probn$ over a sample space $\Xi$ and a measurable random operator $F:\Xi\times X\rightarrow\re^d$. We then define the expected operator
\begin{equation}\label{equation:expected:valued:objective}
T(x):=\probn F(\cdot,x):=\int_\Xi F(\xi,x)\dist\probn(\xi),\quad (x\in X),
\end{equation}
assuming it is well defined over $X$. It will be convenient to consider a common probability space $\Omega$ on which a probability measure $\prob$ and the correspondent expectation $\esp$ are defined. Precisely, from now on we set a random variable $\xi:\Omega\rightarrow \Xi$ with distribution $\probn$ so that $\probn(A)=\prob(\xi\in A)$ and $\probn g=\esp[g(\xi)]$ for any $A\in\Omega$ and integrable random variable $g:\Xi\rightarrow\re$.\footnote{We will sometimes use $\xi\in\Xi$ to denote a point in the sample space if no confusion arises.} Assuming \eqref{equation:expected:valued:objective}, the \emph{stochastic variational inequality} problem (SVI), denoted as VI$(T,X),$ is the problem of finding a $x^*\in X$ such that 
\begin{eqnarray}\label{problem:SVI:intro}
\langle T(x^*),x-x^*\rangle\ge0,\quad\forall x\in X.
\end{eqnarray}
The solution set of \eqref{equation:expected:valued:objective}-\eqref{problem:SVI:intro} will be denoted by $X^*$.

An important property of SVIs is that it generalizes \emph{stochastic optimization} (SP) in the sense that it includes many stochatic variational problems for which $T$ is \emph{not} integrable. Indeed, if $T=\nabla f$ for some smooth function $f:X\rightarrow\re$ satisfying $f=\esp[G(\xi,\cdot)]$ for some measurable function $G:\Xi\times X\rightarrow\re$, then VI$(T,X)$ is the first order necessary condition of the SP problem $\min_X f$. Additionally, both problems are equivalent if $f$ is convex. Notable examples of SVIs which are not related to SPs are the \emph{stochastic saddle-point problem} and the \emph{stochastic Nash equilibrium} problem. From another perspective, SVIs also generalize \emph{stochastic system of equations} in the sense that it includes geometric \emph{constraints} related to optimality conditions. Indeed, if $X=\re^d$ then the problem $T(x)=0$ is equivalent to VI$(T,\re^d)$. See e.g. \cite{facchinei:pang2003,jud:nem:tauvel2011}.

The challenge aspect of SVIs, when compared to deterministic variational inequalities, is that the expectation \eqref{equation:expected:valued:objective} cannot be evaluated.\footnote{Typical reasons are: a sample space with high dimension requiring Monte Carlo evaluation, no knowledge of the distribution $\probn$ or, even worse, no knowledge of a closed form for $F$.} However, a practical assumption is that the decision maker have access to samples drawn from the distribution $\probn$. Under this assumption, a popular methodology to solve \eqref{equation:expected:valued:objective}-\eqref{problem:SVI:intro} is the \emph{Stochastic Approximation} (SA) method. In this approach, the samples are accessed in an interior and online fashion: a deterministic version of an algorithm is chosen and a fresh independent identically distributed (i.i.d.) sample is used whenever the algorithm requires operator estimation at the current or previous iterates \cite{nem:jud:lan:shapiro2009}. In this setting, the mechanism to access $F$ via samples of $\probn$ is usually named a \emph{stochastic oracle} (SO). Precisely, given an input $x\in X$ and an i.i.d. sample $\{\xi_j\}$ drawn from $\probn$ (also independent of $x$), the SO outputs an unbiased sequence $\{F(\xi_j,x)\}$, that is, satisfying $\esp[F(\xi_j,x)]=T(x)$ for all $j$. A different methodology is the \emph{Sample Average Approximation} (SAA) method where an external and offline sample is acquired to approximate the SVI \cite{shapiro:dent:rus2009}. The approximated problem is then solved by a deterministic algorithm of preferred choice.

The SA methodology was first proposed by Robbins and Monro in the seminal paper \cite{robbins:monro1951} for the problem $\min_{x\in\re^d}\{f(x):=\esp[G(\xi,x)]\}$ for a random smooth convex function $G:\Xi\times\re^d\rightarrow\re$, that is, \eqref{equation:expected:valued:objective}-\eqref{problem:SVI:intro} with $F(\xi,\cdot):=\nabla G(\xi,\cdot)$. Their method takes the form 
\begin{equation}\label{equation:stochastic:gradient}
x^{k+1}:=x^k-\alpha_k\nabla G(\xi^k,x^k),
\end{equation}
given an i.i.d. sample sequence $\{\xi^k\}$ and positive stepsize sequence $\{\alpha_k\}$. This was the first instance of the now popular \emph{stochastic gradient method}. This methodology was then extensively explored in numerous works spanning the communities of statistics and stochastic approximation, stochastic optimization and machine learning (see e.g. \cite{bach:moulines2011, bottou:curtis:nocedal2016} and references therein). See also \cite{kushner:yin2003} for other problems where the SA procedure is relevant (such as online optimization, repeated games, queueing theory, signal processing, and control theory). More recently, the SA methodology was also analyzed for SVIs e.g. in \cite{jiang:xu2008,jud:nem:tauvel2011,koshal:nedic:shanbhag2013, yousefian:nedic:shanbhag2014,kannan:shanbhag2014,chen:lan:ouyang2017,
wang:bertsekas2015,iusem:jofre:thompson2015,iusem:jofre:oliveira:thompson2017,
balamurugan:bach2016,bianchi2016}. We refer also to \cite{hiriart-urruty1976}.

The estimation in SA methods is measured by the \emph{oracle error}. This is the map $\epsilon:\Xi\times X\rightarrow\re^d$ defined by
\begin{equation}
\epsilon(\xi,x):=F(\xi,x)-T(x),\quad\quad(\xi\in\Xi,x\in X).\label{equation:oracle:error}
\end{equation}
For $p\ge2$, the \emph{oracle error's $p$-moment} function is defined by
\begin{equation}
\sigma_p(x):=\sqrt[p]{\esp\left[\Vert\epsilon(\xi,x)\Vert^p\right]}\quad\quad(x\in X).\label{equation:oracle:error:variance}
\end{equation}
In the deterministic case, assumptions on the operator $T$ provide local surrogate models to establish the convergence of methods which solve VI$(T,X)$. In order to define and analyze SA methods, assumptions on the variance $\sigma(\cdot)^2:=\sigma_2(\cdot)^2$ (or even higher order moments) are as important as assumptions on $T$. This is because local surrogate models also need the estimation of $T$ from the SO. In that respect, we will consider Lemma \ref{lemma:holder:continuity:mean:std:dev} which is a consequence of the following assumption.
\begin{assumption}[Heavy-tailed H\"older continuous operators]\label{assumption:holder:continuity}
Consider definition \eqref{equation:expected:valued:objective}. There exist $\delta\in(0,1]$ and nonnegative random variable $\mathsf{L}:\Xi\rightarrow\re_+$ such that, for almost every $\xi\in\Xi$, $\mathsf{L}(\xi)\ge1$ and, for all $x,y\in X$,
$$
\Vert F(\xi,x)-F(\xi,y)\Vert\le\mathsf{L}(\xi)\Vert x-y\Vert^\delta.
$$
Define $\mathsf{a}:=1$ if $X$ is compact and $\mathsf{a}:=2$ for a general $X$. We assume there exist $x_*\in X$ and $p\ge2$ such that $\probn\left[\Vert F(\cdot,x_*)\Vert^{\mathsf{a}p}\right]<\infty$ and $\probn\left[\mathsf{L}(\cdot)^{\mathsf{a}p}\right]<\infty$. We define $L:=\probn\mathsf{L}(\cdot)$ and $L_q:=\sqrt[q]{\probn[\mathsf{L}(\cdot)^{q}]}+L$ for any $q>0$.
\end{assumption}
\begin{lemma}[H\"older continuity of the mean and the standard deviation]\label{lemma:holder:continuity:mean:std:dev}
Consider definitions \eqref{equation:oracle:error}-\eqref{equation:oracle:error:variance}, suppose Assumption \ref{assumption:holder:continuity} holds and take $q\in[p,2p]$ such that the integrability conditions of Assumption \ref{assumption:holder:continuity} are satisfied. Then $T$ is $(L,\delta)$-H\"older continuous on $X$ and $\sigma_q(\cdot)$ is $(L_q,\delta)$-H\"older continuous\footnote{We say $T$ is $(L,\delta)$-H\"older continuous if $\Vert T(x)-T(y)\Vert\le L\Vert x-y\Vert^\delta$ for all $x,y\in X$.} on $X$ with respect to the norm $\Vert\cdot\Vert$.
\end{lemma}

See the Appendix for a simple proof. In this work we shall only consider the Euclidean norm $\Vert\cdot\Vert$. Assumption \ref{assumption:holder:continuity} is standard in stochastic optimization \cite{shapiro:dent:rus2009}. It is much less standard in the literature of SA methods where, typically, it is assumed an \emph{uniform bound on the oracle's variance}, i.e., the existence of some $\sigma>0$, such that 
$
\sup_{x\in X}\sigma(x)^2\le\sigma^2.
$
Unless the stochastic error in \eqref{equation:oracle:error} is \emph{independent} of\footnote{This is the case of the additive noise model which is a reasonable assumption in many problem instances.} $x\in X$, such \emph{global} uniform bound implicitly assumes $X$ is compact. Moreover, even if such bound holds, it does not provide sharp complexity estimates since typically $\sigma(x^*)^2\ll\sigma^2$ for $x^*\in X^*$ (see Example 3.9 in \cite{iusem:jofre:oliveira:thompson2017}). Assumption \ref{assumption:holder:continuity} does not require compactness of $X$ (including unconstrained quadratic SPs and affine SVIs with a \emph{random} matrix). Moreover, we will show that our convergence bounds depend only on the \emph{local} variances $\sigma(x^*)^2$, or the correspondent forth moments, at solutions $x^*\in X^*$ (see Theorem \ref{thm:rate:convergence} and Section \ref{section:conclusion}). See e.g. \cite{bach2014} where adaptive methods are proposed to exploit \emph{local} strong convexity modulus in stochastic optimization.

From a practical point of view, our statistical analysis will be built upon the standard assumption of an \emph{unbiased oracle with i.i.d. sampling} (UO). In the rest of the paper, it will be convenient to define the following quantities associated to an i.i.d. sample $\xi^N:=\{\xi_j\}_{j=1}^N$ drawn from $\probn$. Recall definitions \eqref{equation:expected:valued:objective} and \eqref{equation:oracle:error}. We define the \emph{empirical mean operator} and the \emph{oracle's empirical mean error} associated to $\xi^N$, respectively, by
\begin{eqnarray}
\widehat F(\xi^N,x):=\frac{1}{N}\sum_{j=1}^NF(\xi_j,x),\quad\quad\widehat\epsilon(\xi^N,x):=\frac{1}{N}\sum_{j=1}^N\epsilon(\xi_j,x),\quad\quad (x\in X).\label{equation:empirical:mean:operator:&:error}
\end{eqnarray}

\subsection{Related work, proposed methods and contributions}\label{section:related:proposed:work}
The performance of first-order methods for optimization and variational inequalities strongly depend on the \emph{stepsize sequence}. As an example, given a smooth convex function $f:\re^d\rightarrow\re$, a classical method to solve $\min_{\re^d}f$ is the gradient method $x^{k+1}:=x^k-\alpha_k\nabla f(x^k)$, where $\{\alpha_k\}$ is a positive stepsize sequence. One choice of stepsizes that guarantees its convergence is the \emph{small stepsize policy} (SSP): any stepsize sequence satisfying $\sum_k\alpha_k=\infty$ and $\sum_k\alpha_k^2<\infty$, a typical choice being $\alpha_k=\mathcal{O}(k^{-1})$. If $L$ is the Lipschitz constant of $\nabla f(\cdot)$, the \emph{constant stepsize policy} (CSP) $\alpha_k=\mathcal{O}(\frac{1}{L})$ has a provable accelerated convergence rate in comparison to the SSP since its stepsize sequence do not vanishes. However, the later has the advantage of \emph{not requiring an estimate of $L$} and, in this sense, it is a more \emph{robust} and practical policy since $L$ is rarely known. A significant improvement is the use of \emph{line search schemes} which build endogenous \emph{adaptive} stepsizes bounded away from zero at the expence of a few more gradient evaluations. As an example, given iterate $x^k$, Armijo's line search \cite{armijo1966} defines the stepsize $\alpha_k$ as the maximum $\alpha\in\{\theta^\ell\hat\alpha:\ell\in\{0\}\cup\mathbb{N}\}$ such that
\begin{equation}\label{equation:armijo:rule}
f(x^k(\alpha))-f(x^k)\le\lambda\langle\nabla f(x^k),x^k(\alpha)-x^k\rangle,
\end{equation}
where $\hat\alpha\in(0,1]$, $\theta,\lambda\in(0,1)$ are exougenous parameters and, for all $\alpha>0$, $x^k(\alpha):=x^k-\alpha\nabla f(x^k)$. The next iterate is then defined by $x^{k+1}:=x^k(\alpha_k)$. This policy enjoys the accelerated convergence of the CSP and the robustness of the SSP, i.e., it does not require knowledge of $L$. Many variants of line search schemes were developed and extended to include other variational problems. For variational inequalities, two notable ones are the line search schemes of Khobotov \cite{khobotov1987} and of Iusem-Svaiter \cite{iusem:svaiter1997}.

As expected, the stepsize policy is also determinant in the performance of SA methods. In the seminal work \cite{robbins:monro1951} and in later developments, it is shown that the SSP is sufficient for the convergence of the SG method \eqref{equation:stochastic:gradient}. The nontrivial aspect here is that the SSP has to deal not only with the convergence of the sequence but also to progressively \emph{reduce the variance} of the oracle's error trajectory $\{\nabla G(\xi^k,x^k)-\nabla f(x^k)\}_{k\in\mathbb{N}}$. Such additional challenge in SA methods is still an active research subject and it has received a burst of interest in the last decade motivated by large scale statistical machine learning applications.\footnote{In this challenge setting, theoretical and practical experience has shown that first order methods are competitive and, sometimes, the best known methods.}

The performance of a SA method can be measured by its \emph{iteration} and \emph{oracle complexities} given a tolerance $\epsilon>0$ with respect to a suitable metric. The first is the total number of iterations, a measure for the optimization error, while the second is the total number of samples and oracle calls, a measure for the estimation error. As an example, statistical lower bounds \cite{agarwal:barlett:ravikumar:wainwright2012} show that the class of smooth convex functions has an optimal oracle complexity of $\mathcal{O}(\epsilon^{-2})$ in terms of the optimality gap. A fundamental improvement with respect to estimation error was Polyak-Ruppert's \emph{iterate averaging} scheme \cite{polyak1991,polyak:juditsky1992,ruppert1988,nem:rudin1978}. This scheme replaces the SSP by \emph{longer stepsizes} $\alpha_k=\mathcal{O}(k^{-\frac{1}{2}})$ with a subsequent \emph{final average} of the iterates using the stepsizes as weights (this is sometimes called \emph{ergodic average}). \emph{If one oracle call per iteration is postulated}, such scheme obtains a convergence rate of $\mathcal{O}(k^{-\frac{1}{2}})$ with optimal iteration and oracle complexities of $\mathcal{O}(\epsilon^{-2})$ on the class of smooth convex functions. This is also the size of the final ergodic average, a measure of the additional \emph{averaging effort} implicitly required in iterate averaging schemes. Such methods, hence, are efficient in terms of oracle complexity. Iterate averaging was then extensively explored (see e.g.  \cite{juditsky:nazin:tsybakov:vayatis2005,juditsky:rigollet:tsybakov2008,nesterov:vial2008,nesterov2009,%
nem:jud:lan:shapiro2009,xiao2010,jud:nem:tauvel2011}). The important work \cite{nem:jud:lan:shapiro2009} studies the \emph{robustness} of iterate averaging in SA methods and shows that such schemes can outperform the SAA approach on relevant convex problems. On the strongly convex class, \cite{bach:moulines2011} gives a detailed non-asymptotic \emph{robust} analysis of Polyak-Ruppert averaging scheme. It theoretically and numerically justifies the importance of iterate averaging in handling the oracle's error variance.

Although iterate averaging methods obtain optimal oracle complexity, a remaining question is if \emph{improved iteration complexity} with (near) \emph{optimal oracle complexity} can be achieved. In this sense, a rapidly and recent line of research proposes SA methods with \emph{variance reduction}  using \emph{more than one oracle call per iteration} to alleviate the role of the stepsize in reducing variance. Two representative examples include \emph{gradient aggregation methods} and \emph{dynamic sampling methods} (see \cite{bottou:curtis:nocedal2016}, Section 5). These methods can use a \emph{constant stepsize policy} and thus obtain an accelerated rate of convergence when compared to the iterate averaging scheme. Designed for finitely supported distributions with bounded data, gradient aggregation methods reduce the variance by combining in a specific manner eventual exact computation (or storage) of gradients and eventual iterate averaging (or randomization schemes) with frequent gradient sampling. See e.g. \cite{bottou:curtis:nocedal2016} and references therein. Designed to solve problems with an arbitrary distribution and online data acquisition (as is the case in many stochastic and simulation optimization problems based on Monte Carlo methods), dynamic sampling methods reduce variance by estimating the gradient via an \emph{empirical average} associated to a sample whose size (\emph{mini-batch}) is increased at every iteration. See e.g. \cite{deng:ferris2009,byrd:chiny:nocedal:wu2012,%
friedlander:schmidt2013,shanbhag:blanchet2015,
ghadimi:lan2016,iusem:jofre:oliveira:thompson2017} and references therein. However, an essential point is if such increased effort in computation per iteration is worth. A nice fact is that current gradient aggregation and dynamic sampling methods achieve, up to constants, the order of the deterministic optimal iteration complexity with the \emph{same} (near) optimal oracle complexity and averaging effort of standard iterate averaging schemes. \emph{In this sense}, gradient aggregation and dynamic sampling methods can be a more efficient option than iterate averaging.

We now comment on the main purpose of this work. All variance reduction SA methods mentioned above still use a constant stepsize policy $\alpha_k=\mathcal{O}(\frac{1}{L})$ \emph{assuming knowledge of the Lipschitz constant}. Hence, although they improve the convergence of SA methods, iterate averaging with $\alpha_k=\mathcal{O}(k^{-\frac{1}{2}})$ is still a more robust policy when $L$ or other needed parameters are unknown or poorly estimated \cite{nem:jud:lan:shapiro2009, bach:moulines2011}. In this setting, current variance reduction methods may be impractical. An important question is: can \emph{faster rates of convergence} with (near) \emph{optimal oracle complexity} be accomplished by \emph{robust variance reduction} methods? By robust variance reduction we mean the use of adaptive schemes that avoid exogenous estimation of $L$ and produce a stepsize sequence bounded away from zero. Motivated by line search schemes in deterministic methods, our aim is to propose line search schemes for a class of dynamic sampled SA methods (DS-SA). In this work we focus on SVIs and pursue an improved complexity analysis of stochastic optimization problems in future research. To the best of our knowledge, line search schemes for SVIs are currently nonexistent. Even for stochastic optimization, considering that Robbins-Monro's seminal work was published in 1951, it seems that only very few existing works treat adaptive stepsize search schemes for SA methods with \emph{stepsizes bounded away from zero} \cite{maclaurin:duvenaud:adams2015, mahsereci:hennig2017, schaul:zhang:lecun2013, masse:ollivier2015, tan:ma:dai:qian2016, wardi1990, krejic:jerinkic2015, krejic:luzanin:nikolovski:stojkovska2015, krejic:luzanin:ovcin:stojkovska2015}. Still, some of them only suggest a scheme without a provable convergence theory \cite{maclaurin:duvenaud:adams2015,mahsereci:hennig2017}. For those which do guarantee convergence, some still require knowledge of the Lipschitz constant or use a small stepsize policy \cite{schaul:zhang:lecun2013, masse:ollivier2015,tan:ma:dai:qian2016} and, hence, are not robust variance reduction methods. Finally, the analysis in \cite{wardi1990, krejic:jerinkic2015, krejic:luzanin:nikolovski:stojkovska2015, krejic:luzanin:ovcin:stojkovska2015} are too restrictive since they require much more than the standard assumption of an UO used in stochastic approximation and do not give complexity estimates. In all mentioned works, uniformly bounded assumptions are made (either on the set or on the oracle's variance) and no convergence rates or oracle complexity are given (hence their efficiency cannot be compared to the SSP). Differently, our oracle assumptions are standard and, in this sense, our proposals are also novel for stochastic optimization problems viewed as particular cases of SVIs. Moreover, we provide rate of convergence and oracle complexity and do not assume uniform boundedness. 

Finally, before presenting our methods and results, it will be very instructive to briefly discuss why the analysis of line search schemes in SA methods are considerably \emph{different} and intrinsically \emph{more difficult} than in the deterministic case. This may explain the absence of a satisfying
convergence theory of SA methods with line search schemes which: (1) do not use knowledge of the Lipschitz constant, (2) obtain stepsizes bounded away from zero and (3) only assume an UO. Since \cite{robbins:monro1951, robbins:siegmund1971}, it is well known that the analysis of SA methods strongly relies on \emph{martigale processes}. From a generic perspective, such martingale-like property is obtained by: 
\begin{itemize}
\item[(i)] \textsf{Optimization process}: the deterministic iterative algorithm satisfies a fixed-point contraction or Lyapunov principle.\footnote{This is usually obtained by properties like convexity of the objective, smoothness of a nonconvex objective and monotonicity or nonexpansion of an operator.}
\item[(ii)] \textsf{Estimation process}: standard stochastic approximation consists in using a \emph{fresh i.i.d. sample} update at every iteration.
\item[(iii)] \textsf{Exogenous stepsize policies}: for instance, the SSP $\alpha_k=\mathcal{O}(k^{-1})$, longer stepsizes $\alpha_k=\mathcal{O}(k^{-\frac{1}{2}})$ with iterate averaging and the CSP $\alpha_k=\mathcal{O}(\frac{1}{L})$. We also include adaptive vanishing stepsizes which achieve better tunned constants but still require exogenous parameters (see e.g. \cite{yousefian:nedic:shanbhag2016}). 
\end{itemize}
As an example, consider the stochastic gradient method \eqref{equation:stochastic:gradient} with stepsizes satisfying $0<\sup_k\alpha_k<\frac{1}{2L}$. Given a solution $x^*$, it is possible to show that
$$
\Vert x^{k+1}-x^*\Vert^2\le\Vert x^{k}-x^*\Vert^2-\left(\frac{1}{2}-L\alpha_k\right)\frac{\alpha_k^2}{2}\Vert\nabla f(x^k)\Vert^2+2\alpha_k\langle\epsilon^k,x^*-x^k\rangle+\alpha_k^2\Vert\epsilon^k\Vert^2,
$$
where $\epsilon^k:=\nabla G(\xi^k,x^k)-\nabla f(x^k)$ is the oracle error at the $k$-th iterate. If $\alg_k:=\sigma(\xi^0,\ldots,\xi^{k-1})$ denotes the $\sigma$-algebra encoding the information up to the iteration $k$, then the above relation and the fact that $\{\xi^i\}_{i=0}^\infty$ is an i.i.d. sequence imply that the iterates' error sequence $\{\Vert x^k-x^*\Vert^2\}$ defines a ``perturbed'' \emph{supermartingale} sequence adapted to $\{\alg_k\}$ (see Section \ref{section:preliminaries}, Theorem \ref{thm:rob}).\footnote{Robbins and Monro \cite{robbins:monro1951} called this an ``almost'' supermartingale sequence. In the classical terminology from the deterministic optimization community, this would correspond to a stochastically adapted version of quasi-F\'ejer sequences.} This sequence is defined over the \emph{iteration time-scale} and accounts for the optimization error. On the other hand, the oracle's error sequence $\{\epsilon^k\}$ defines an \emph{exact martingale difference} adapted to $\{\alg_k\}$, i.e., $\esp[\epsilon^k|\alg_k]=0$ for all $k$. This sequence is defined over the \emph{estimation time-scale} and accounts for the gradient estimation error.

If one considers \emph{adaptive endogenous stepsizes} and use variance reduction, a natural choice would be a SA version of Armijo's rule \eqref{equation:armijo:rule}: chose $\alpha_k$ as the maximum $\alpha\in\{\theta^\ell\hat\alpha:\ell\in\{0\}\cup\mathbb{N}\}$ such that
\begin{equation}\label{equation:armijo:DS-SA}
\widehat G\left(\xi^k,x^k(\alpha)\right)-\widehat G(\xi^k,x^k)\le\lambda\left\langle\nabla \widehat G(\xi^k,x^k),x^k(\alpha)-x^k\right\rangle,
\end{equation}
where $\hat\alpha\in(0,1]$, $\theta,\lambda\in(0,1)$, $\xi^k:=\{\xi^k_j\}_{j=1}^{N_k}$ is an i.i.d. sample from $\probn$ such that $N_k\rightarrow\infty$ and, for all $\alpha>0$, $x^k(\alpha):=x^k-\alpha\nabla\widehat G(\xi^k,x^k)$. In above, $\widehat G(\xi^k,x^k)$ and $\nabla \widehat G(\xi^k,x^k)$ denote, respectively, the empirical averages of $G(\cdot,x^k)$ and $\nabla G(\cdot,x^k)$ with respect to the sample $\xi^k$. The challenging aspect of the above scheme is highlighted:
\begin{quote}
\textsf{(A):} DS-SA \emph{line search schemes intrinsically introduce nonmartingale-like dependencies even when using i.i.d. sampling.
}
\end{quote}
To see this, first note that the \emph{backtracking} scheme \eqref{equation:armijo:DS-SA} examines the variation of $\widehat G(\xi^k,\cdot)$ along a discrete path $\alpha\mapsto x^k(\alpha)$ so that the chosen stepsize $\alpha_k$ and accepted iterate $x^{k+1}:=x^{k}(\alpha_k)$ are both measurable functions of $(\xi^k,x^k)$. Second, by using the contraction principle produced by \eqref{equation:armijo:DS-SA}, we are forced to estimate the oracle error $\widehat\epsilon(\xi^k,x^{k+1})=\widehat G(\xi^k,x^{k+1})- f(x^{k+1})$ which \emph{is not a martingale difference}: it is a measurable function of the \emph{coupled} variables $\xi^k$ and $x^{k+1}$ due to backtracking. Even when $\xi^k$ is an i.i.d. sample of $\probn$, this coupling is inevitable and, hence, the desired convergence
\begin{equation}\label{equation:postulated:SLLN}
\lim_{k\rightarrow\infty}\widehat\epsilon(\xi^k,x^{k+1})=\lim_{k\rightarrow\infty}\sum_{j=1}^{N_k}\frac{G(\xi^k_j,x^{k+1})-f(x^{k+1})}{N_k}=0,
\end{equation}
either in almost sure sense or in distribution, \emph{does not} follow from the standard Strong Law of Large Numbers or the Central Limit Theorem: the above sum \emph{is not a sum of independent random variables}. The nontrivial aspect here is that a SA method with line search has \emph{two} statistical estimation processes: the gradient estimation of item (ii) above \emph{and} the Lipschitz constant estimation replacing (iii). In this sense, DS-SA methods with line search schemes are statistically different than standard SA methods. We finally remark that in all the works  \cite{wardi1990,krejic:luzanin:ovcin:stojkovska2015,krejic:jerinkic2015,
krejic:luzanin:nikolovski:stojkovska2015} the convergence \eqref{equation:postulated:SLLN} is \emph{postulated}, putting aside the challenging aspect in \textsf{(A)}. Thus, their assumptions are far beyond the usual assumption of an UO. Errors of the type $\widehat{\epsilon}(\xi^k,x^{k+1})$ will be referred as \emph{correlated errors}.

In this work, we propose \textsf{Algorithm \ref{algorithm:DSSA:extragradient}} for Lipschitz continuous operators and \textsf{Algorithm \ref{algorithm:DSSA:hyperplane}} for general H\"older continuous operators to solve SVIs via the SA methodology. These methods use dynamic sampling and line search schemes to cope with the absence of the Lipschitz constant or the parameters of H\"older continuity. Our contributions are resumed in the following.

\begin{algorithm}
\caption{DS-SA-extragradient method with a DS-SA line search scheme}\label{algorithm:DSSA:extragradient}
\begin{algorithmic}[1]
  \scriptsize
  \STATE INITIALIZATION: Choose the initial iterate $x^0\in\mathbb{R}^d$, parameters 
$\hat\alpha,\theta\in(0,1]$ and $\lambda\in\left(0,\frac{1}{\sqrt{6}}\right)$ and the sample rate $\{N_k\}$.
  \STATE ITERATIVE STEP: Given iterate $x^k$, generate sample $\xi^k:=\{\xi^k_{j}\}_{j\in [N_k]}$ from $\probn$ and compute
\begin{equation}\label{equation:empirical:average:DSSA:extragradient}
\widehat F(\xi^k,x^k):=N_k^{-1}\sum_{j=1}^{N_k}F(\xi_j^k,x^k).
\end{equation}
If $x^k=\Pi\left[x^k-\hat\alpha\widehat F(\xi^k,x^k)\right]$ stop. Otherwise,

\textsf{LINE SEARCH RULE}: define $\alpha_k$ as the 
maximum $\alpha\in\{\theta^\ell\hat\alpha:\ell\in\{0\}\cup\mathbb{N}\}$ such that 
\begin{equation}\label{algo:armijo:rule}
\alpha\left\Vert\widehat F\left(\xi^k,z^k(\alpha)\right)-\widehat F\left(\xi^k,x^k\right)\right\Vert
\le\lambda\Vert z^k(\alpha)-x^k\Vert, 
\end{equation}
where, for all $\alpha>0$, compute
$
z^k(\alpha):=\Pi\left[x^k-\alpha\widehat F(\xi^k,x^k)\right]
$
and 
$
\widehat F\left(\xi^k,z^k(\alpha)\right):=N_k^{-1}\sum_{j=1}^{N_k}F(\xi_j^k,z^k(\alpha)).
$

Generate sample $\eta^k:=\{\eta^k_{j}\}_{j\in [N_k]}$ from $\probn$ and set 
\begin{eqnarray}
z^k&=&\Pi\left[x^k-\alpha_k\widehat F(\xi^k,x^k)\right],\label{algo:extragradient:armijo1}\\
x^{k+1}&=&\Pi\left[x^k-\alpha_k\widehat F(\eta^k,z^k)\right].\label{algo:extragradient:armijo2}
\end{eqnarray}
\end{algorithmic}
\end{algorithm}

\begin{algorithm}
\caption{DS-SA-hyperplane method}\label{algorithm:DSSA:hyperplane}
\begin{algorithmic}[1]
  \scriptsize
  \STATE INITIALIZATION: Choose the initial iterate $x^0\in\mathbb{R}^d$, parameters 
$\tilde\beta\ge\hat\beta>0$, $\hat\alpha\in(0,1]$ and $\lambda,\theta\in(0,1)$, the 
step sequence $\{\beta_k\}\subset[\hat\beta,\tilde\beta]$ and the sample rate $\{N_k\}$.
  \STATE ITERATIVE STEP: Given iterate $x^k$, generate 
sample $\xi^k:=\{\xi^k_{j}\}_{j=1}^{N_{k}}$ from $\probn$ and compute
$
\widehat F(\xi^k,x^k):=N_k^{-1}\sum_{j=1}^{N_k}F(\xi_j^k,x^k).
$  
If $x^k=\Pi\left[x^k-\beta_k\widehat F(\xi^k,x^k)\right]$ stop. Otherwise,

\textsf{LINE SEARCH RULE:} define $\alpha_k$ as the maximum $\alpha\in\{\theta^\ell\hat\alpha:\ell\in\{0\}\cup\mathbb{N}\}$ such that 
\begin{equation}\label{algo:armijo:rule2}
\left\langle\widehat F\left(\xi^k,\bar z^k(\alpha)\right),x^k-\Pi(g^k)\right\rangle
\ge\frac{\lambda}{\beta_k}\Vert x^k-\Pi(g^k)\Vert^2,
\end{equation}
where $g^k:=x^k-\beta_k\widehat F(\xi^k,x^k)$ and for all $\alpha>0$, compute 
$
\overline z^k(\alpha):=\alpha\Pi(g^k)+(1-\alpha)x^k
$
and 
$
\widehat F\left(\xi^k,\overline z^k(\alpha)\right):=N_k^{-1}\sum_{j=1}^{N_k}F(\xi_j^k,\overline z^k(\alpha)).
$

Set 
\begin{eqnarray}
z^k&:=&\alpha_k\Pi\left[x^k-\beta_k\widehat F(\xi^k,x^k)\right]+(1-\alpha_k)x^k,\label{algo:hyperplane1}\\
x^{k+1}&:=&\Pi\left[x^k-\gamma_k\widehat F(\xi_k,z^k)\right],\label{algo:hyperplane2}
\end{eqnarray}
with 
$
\gamma_k:=\left\langle\widehat F(\xi^k,z^k),x^k-z^k\right\rangle\cdot\Vert\widehat F(\xi^k,z^k)\Vert^{-2}.
$
\end{algorithmic}
\end{algorithm}

(i) \emph{Robust variance reduction with efficient oracle complexity and multiplicative noise}: To the best of our knowledge, \textsf{Algorithm \ref{algorithm:DSSA:extragradient}} is the first provable \emph{robust} variance reduced SA method, either for SVIs or SPs, with improved iteration complexity and near optimal oracle complexity. This means that we obtain, up to constants, an optimal iteration complexity of $\mathcal{O}(\epsilon^{-1})$ and near optimal oracle complexity of $\mathcal{O}(\epsilon^{-2})$ (up to log factors on $\epsilon$ and $L$) in the large sample setting for SVIs with Lipschitz continuous operators without the a priori knowledge of the Lipschitz constant $L$. Previous nonrobust variance reduction methods use the policy $\alpha_k=\mathcal{O}(\frac{1}{L})$ and obtain, up to constants, the same complexities \cite{iusem:jofre:oliveira:thompson2017} but require an exogenous estimate of $L$. Such estimate is often nonexistent in practice. Moreover, even in possession of such an estimate, the convergence can be slow if it is of a poor quality. On the other hand, previous robust methods use vanishing stepsizes with the poorer iteration complexity of $\mathcal{O}(\epsilon^{-2})$ in the case of ill-conditioned problems \cite{nem:jud:lan:shapiro2009}. Concerning line search schemes, they are nonexistent for SVIs but it seems our results are also new for SPs (seen as a particular SVI): all current methods either still use the knowledge of $L$ and other parameters, use the small stepsize policy (and, hence, have a slower iteration complexity) or postulate \eqref{equation:postulated:SLLN} without giving complexity estimates \cite{wardi1990,krejic:luzanin:ovcin:stojkovska2015,krejic:jerinkic2015,
krejic:luzanin:nikolovski:stojkovska2015}. Condition \eqref{equation:postulated:SLLN} is much stronger than the standard assumption of an UO, a sufficient assumption for our analysis. Differently than previous robust methods for bounded ill-contidioned problems \cite{nem:jud:lan:shapiro2009}, we ask only Assumption \ref{assumption:holder:continuity} (an oracle with multiplicative noise). In this aggressive but practical setting, the oracle's variance is not uniformly upper bounded if $X$ is unbounded. Our bounds are local in the sense that they depend on variance at solutions, the Lipschitz constant and initial iterates (but not on the diamater of $X$ nor on a global variance upper bound). Compared to nonrobust variance reduced methods \cite{byrd:chiny:nocedal:wu2012,ghadimi:lan2016,iusem:jofre:oliveira:thompson2017}, a price to pay in our estimates for not having an exogenous estimate of $L$ is that the oracle complexity of \textsf{Algorithm \ref{algorithm:DSSA:extragradient}} has an additional factor of $\ln (L)\mathcal{O}(d)$. We note however, that such upper bound is tight in comparison to the sample complexity of the general SAA estimator, an estimator which does not assume extra information on the problem (see e.g. Theorem 5.18 in \cite{shapiro:dent:rus2009}). We refer to Theorem \ref{thm:rate:convergence}, Corollary \ref{cor:oracle:complexity} and Section \ref{section:conclusion}.\footnote{Our complexities hold for the quadratic natural residual or the D-gap function (see Section \ref{section:preliminaries}). If $X$ is compact, our method achieves the same complexities, up to constants, in terms of the dual-gap function (see e.g. \cite{nem:jud:lan:shapiro2009, chen:lan:ouyang2017}).}

(ii) \emph{Complexity estimates of SA methods via a local empirical process theory}: as mentioned before, DS-SA line schemes intrinsically introduce nonmartigale-like processes. Going beyond standard martingale techniques used in SA methods with exogenous stepsize policies, we use a novel analysis based on advanced techniques from \emph{Empirical Process Theory}  \cite{boucheron:lugosi:massart2013,panchenko2003} to analyze correlated errors introduced in stochastically approximated line search schemes. Very importantly, we \emph{do not} postulate significantly narrower oracle assumptions such as \eqref{equation:postulated:SLLN} used in \cite{wardi1990,krejic:luzanin:ovcin:stojkovska2015,krejic:jerinkic2015,
krejic:luzanin:nikolovski:stojkovska2015}. We refer the reader to Section \ref{section:empirical:process:theory:DSSA} for a detailed description. This is the most sensible part of our work and the cornerstone tool. Our analysis also sets the ground for potential generalizations to other robust algorithms based on the SA methodology\footnote{Possibly requiring nontrivial adaptations.}. In a nutshell, our proposition is to \emph{locally} decouple the dependency in the correlated error up to the control of an empirical process over a suitable ball centered at the current iterate. The intuition here is that the iterate generated after the line search scheme, although highly dependent on the fresh i.i.d. sample, lies at a ball whose radius is dependent on previous information and on a martingale difference error. We refer to Section \ref{section:empirical:process:theory:DSSA} for futher details. The statistical preliminaries are carefully presented. 

Besides items (i)-(ii) above, another contribution is the proof of convergence of \textsf{Algorithm \ref{algorithm:DSSA:hyperplane}}. Our main interest in this algorithm is that, differently than \textsf{Algorithm \ref{algorithm:DSSA:extragradient}} whose convergence holds for Lipschitz continuous operators, \textsf{Algorithm \ref{algorithm:DSSA:hyperplane}} converges for arbitrary H\"older continuous operators without any knowledge of the exponent $\delta$  and the H\"older modulus.

In Section \ref{section:preliminaries} we give some preliminaries. Section \ref{section:empirical:process:theory:DSSA} develops a general empirical process theory which is later applied in the convergence theory of \textsf{Algorithms \ref{algorithm:DSSA:extragradient}} and \textsf{\ref{algorithm:DSSA:hyperplane}}. The convergence theory of \textsf{Algorithm \ref{algorithm:DSSA:extragradient}} is presented in Section \ref{section:algorithm:extragradient:DSSA} while the convergence theory of \textsf{Algorithm \ref{algorithm:DSSA:hyperplane}} is presented in Section \ref{section:algorithm:hyperplane:DSSA}. Section \ref{section:conclusion} concludes with some discussions concerning \textsf{Algorithm \ref{algorithm:DSSA:extragradient}}. Some lemmas are proved in the Appendix.

\section{Preliminaries and notation}\label{section:preliminaries}
For $x,y\in\re^d$, we denote by $\langle x,y\rangle$ the standard inner product, and by $\Vert x\Vert=\sqrt{\langle x,x\rangle}$ the correspondent Euclidean norm. Given $C\subset\re^d$ and $x\in\re^d$, we use the notation $\dist(x,C):=\inf\{\Vert x-y\Vert:y\in C\}$ and $\diam(C):=\sup\{\Vert x-y\Vert:x,y\in C\}$. For a closed and convex set $C\subset\mathbb{R}^d$, we use the notation $\Pi_{C}(x):=\argmin_{y\in C}\Vert y-x\Vert^2$ for $x\in\re^d$. Given $H:\re^d\rightarrow\re^d$, S$(H,C)$ denotes the solution set of VI$(H,C)$. The following properties of the projection operator are well known (see e.g. \cite{facchinei:pang2003, iusem:svaiter1997} and \cite{chen:lan:ouyang2017}, Proposition 4.1). 
\begin{lemma}\label{lemma:proj}
Take a closed and convex set $C\subset\mathbb{R}^d$.
\begin{itemize}
\item[i)] Let $v\in\re^d$ and $x\in C$ with $z:=\Pi_C[x-v]$. Then, for all $u\in C$, 
$
2\langle v,z-u\rangle\le\Vert x-u\Vert^2-\Vert z-u\Vert^2-\Vert z-x\Vert^2.
$
\item[ii)] For all $x\in\mathbb{R}^d, y\in C$,
$
\Vert \Pi_{C}(x)-y\Vert^2+\Vert \Pi_{C}(x)-x\Vert^2\le\Vert x-y\Vert^2.
$	
\item[iii)]For all $x,y\in\mathbb{R}^d$,
$
\Vert \Pi_{C}(x)-\Pi_{C}(y)\Vert\le\Vert x-y\Vert.
$	
\item[iv)]Given $H:\re^d\rightarrow\re^d$, $\mbox{\emph{S}}(H,C)=\{x\in\re^d:x=\Pi_{C}[x-H(x)]\}$.
\item[v)] For all $x\in C,y\in\mathbb{R}^d$,
$
\langle x-y,x-\Pi_C(y)\rangle\ge\Vert x-\Pi_C(y)\Vert^2.
$
\end{itemize}
\end{lemma}
For $X$ as in \eqref{problem:SVI:intro}, we use the notation $\Pi:=\Pi_X$. Given an operator $H:\re^d\rightarrow\re^d$, for any $x\in\re^{n}$ and $\alpha>0$, the \emph{natural residual function} associated to VI$(H,X)$ is
$$
r_\alpha(H;x):=\left\Vert x-\Pi\left[x-\alpha H(x)\right]\right\Vert,\quad\quad (x\in X).
$$
It is a equivalent metric to the D-gap function (see \cite{facchinei:pang2003}, Theorems 10.2.3 and 10.3.3 and Proposition 10.3.7). For $T$ as in \eqref{problem:SVI:intro}, we use the notation $r_\alpha:=r_\alpha(T,\cdot)$. For $\alpha=1$, we define $r(H;\cdot):=r_1(H;\cdot)$ and $r:=r_1$. We shall need the following lemma (see \cite{facchinei:pang2003}, Proposition 10.3.6).

\begin{lemma}\label{lemma:residual:decrease}
Given $x\in\re^d$, the function $(0,\infty)\ni\alpha\mapsto \frac{r_\alpha(H,x)}{\alpha}$ is non-increasing.
\end{lemma}

Given sequences $\{x^k\}$ and $\{y^k\}$, we use the notation $x^k=\mathcal{O}(y^k)$ or $\Vert x^k\Vert\lesssim\Vert y^k\Vert$ 
to mean that there exists a constant $C>0$ such that $\Vert x^k\Vert\le C\Vert y^k\Vert$ for all $k$. The notation $\Vert x^k\Vert\sim\Vert y^k\Vert$ means that $\Vert x^k\Vert\lesssim\Vert y^k\Vert$ and $\Vert y^k\Vert\lesssim\Vert x^k\Vert$. Given a $\sigma$-algebra $\alg$ and a  random variable $\xi$, we denote by $\esp[\xi]$, $\esp[\xi|\alg]$, and $\var[\xi]$, the expectation, conditional expectation and 
variance, respectively. Given $p\ge1$, $\Lpnorm{\xi}$ is the $\mathcal{L}^p$-norm of $\xi$ and 
$
\Lpnorm{\xi\,|\alg}:=\sqrt[p]{\esp\left[|\xi|^p\,|\alg\right]}
$
is the $\mathcal{L}^p$-norm of $\xi$ conditional to $\alg$. We denote by $\sigma(\xi_1,\ldots,\xi_k)$ the $\sigma$-algebra generated by the random variables $\{\xi_i\}_{i=1}^k$ and $\esp[\cdot|\xi_1,\ldots,\xi_k]:=\esp[\cdot|\sigma(\xi_1,\ldots,\xi_k)]$. We write $\xi\in\alg$ for ``$\xi$ is $\alg$-measurable'', $\xi\perp\perp\alg$ for ``$\xi$ is independent of $\alg$'' and $\unit_A$ for the characteristic function of a set $A\in\alg$. Given  $x,y\in\re$, $\lceil x\rceil$ denotes the smallest integer greater than $x$, $x\vee y:=\max\{x,y\}$ and $x\wedge y:=\min\{x,y\}$. $\mathbb{N}_0:=\mathbb{N}\cup\{0\}$ and, for $m\in\mathbb{N}$, we use the notation $[m]=\{1,\ldots,m\}$. $|\mathcal{V}|$ denotes the cardinality of a set $\mathcal{V}$, $\mathbb{B}$ denotes the Euclidean unit ball and $\mathbb{B}[x,r]$ denotes the Euclidean ball with center $x$ and radius $r>0$.

As in other stochastic approximation methods, a fundamental tool to be used is the following Convergence Theorem of Robbins and Siegmund \cite{robbins:siegmund1971} for perturbed nonnegative supermartingales.
\begin{theorem}\label{thm:rob}
Let $\{y_k\},\{u_k\}, \{a_k\}, \{b_k\}$ be sequences of non-negative random variables, adapted to the filtration $\{\alg_k\}$, such that almost surely (a.s.) $\sum a_k<\infty$, $\sum b_k<\infty$ and for all $k\in\mathbb{N}$,
$
\esp\big[y_{k+1}\big| \alg_k\big]\le(1+a_k)y_k-u_k+b_k.
$
Then a.s. $\{y_k\}$ converges and $\sum u_k<\infty$.
\end{theorem}

\section{An empirical process theory for DS-SA line search schemes}\label{section:empirical:process:theory:DSSA}

As mentioned in Section \ref{section:related:proposed:work}, if $L$ in Assumption \ref{assumption:holder:continuity} is known then the analysis of SA methods with the CSP can exploit the fact that the oracle error's define a \emph{martingale difference}. This type of errors can be controlled in a relatively straightforward way (see Lemma \ref{lemma:decay:empirical:error} in Section \ref{section:proof:theorem}). The main objective of this section is to prove the following theorem. This is will the most sensitive part of our analysis and it is the cornerstone tool to handle \emph{nonmartingale-like} oracle errors obtained when stepsize DS-SA line search schemes are used to estimate an unknown $L$ (see \textsf{(A)} in Section \ref{section:related:proposed:work} and comments following it). 

\begin{theorem}[Local bound for the $\mathcal{L}^p$-norm of the correlated error in DS-SA line search schemes]\label{thm:variance:error:with:line:search}
Consider the \emph{SVI} given by \eqref{equation:expected:valued:objective}-\eqref{problem:SVI:intro} with solution set $X^*$. Let $\xi^N:=\{\xi_j\}_{j=1}^N$ be an i.i.d sample drawn from $\probn$ and let $\alpha_N:\Xi\rightarrow[0,\hat\alpha]$ be a random variable for some $0<\hat\alpha\le1$. Suppose that Assumption \ref{assumption:holder:continuity} holds, recall definitions \eqref{equation:oracle:error}-\eqref{equation:empirical:mean:operator:&:error} and define $\delta_1:=0$ if $\delta=1$ and $\delta_1:=1$ if $\delta\in(0,1)$.  

Given $(\alpha,x)\in[0,\hat\alpha]\times X$, we define
$$
z\left(\xi^N;\alpha,x\right):=\Pi\left[x-\alpha\widehat F\left(\xi^N,x\right)\right],
$$
and $\overline z_\beta(\xi^N;\alpha,x):=\alpha z(\xi^N;\beta,x)+(1-\alpha)x$, given $\beta>0$. Then the following holds:
\begin{itemize}
\item[(i)] There exist positive constants $\{\mathsf{c}_i\}_{i=1}^4$ (depending on $d$, $\delta$, $p$ and $L_{2p}\hat\alpha$) such that, for any $x\in X$ and $x^*\in X^*$,
\begin{eqnarray*}
\Lpnorm{\left\Vert\widehat\epsilon\left(\xi^N, z(\xi^N;\alpha_N,x)\right)\right\Vert}&\le &\frac{\mathsf{c}_1\sigma_{2p}(x^*)+\overline{L}_{2p}\left[\delta_1\vee\Vert x-x^*\Vert^\delta\right]}{\sqrt{N}},
\end{eqnarray*}
where $\overline{L}_{2p}:=\mathsf{c}_2L_2+\mathsf{c}_3L_p+\mathsf{c}_4L_{2p}$.
\item[(ii)] If $X$ is compact, there exist positive constants $\mathsf{d}_2$ and $C_p$ (depending on $d$, $\delta$ and $p$) such that, for any $x\in X$ and $x^*\in X^*$,
\begin{eqnarray*}
\Lpnorm{\left\Vert\widehat\epsilon\left(\xi^N, z(\xi^N;\alpha_N,x)\right)\right\Vert}&\le &\frac{C_p\sigma_{p}(x^*)+L_{p}^*\diam(X)^\delta}{\sqrt{N}},
\end{eqnarray*}
where ${L}_{p}^*:=\mathsf{d}_2L_2+pL_p$.
\end{itemize} 
Up to universal constants, the same bounds above holds for 
$\Lpnorm{\left\Vert\widehat\epsilon\left(\xi^N, \overline z_\beta(\xi^N;\alpha_N,x)\right)\right\Vert}$.
\end{theorem}

For further detail on the constants of Theorem \ref{thm:variance:error:with:line:search}, see Remark \ref{rem:constants:thm:correlated:error} in Section \ref{section:proof:theorem}. To prove Theorem \ref{thm:variance:error:with:line:search}, we will crucially require intermediate results which rely on a branch of statistics called \emph{Empirical Process Theory}. Let $\{X_j\}_{j=1}^N$ be a sequence of \emph{independent} stochastic processes $X_j:=(X_{j,t})_{t\in\mathcal{T}}$ indexed by a countable set $\mathcal{T}$ with real-valued random components $X_{j,t}$. The associated \emph{empirical process} (EP) is the stochastic process $\mathcal{T}\in t\mapsto Z_t:=\sum_{j=1}^NX_{j,t}$. An essential quantity in this theory is $Z:=\sup_{t\in\mathcal{T}}Z_t$. If $\mathcal{T}=\{t\}$, then $Z$ is simply a sum of independent random variables. Otherwise, $Z$ is a much more complicated object. To understand $Z$, it is important to bound its expectation and variance. EPs arise in many different settings in mathematical statistics \cite{boucheron:lugosi:massart2013}.

We apply EP theory as a novel way to successfully analyze stochastic approximated line search schemes. Referring to \textsf{Algorithm \ref{algorithm:DSSA:extragradient}} and Theorem \ref{thm:variance:error:with:line:search}, we have $z^k=z(\xi^k;\alpha_k,x^k)$ and must control the correlated error $\widehat\epsilon(\xi^k,z(\xi^k;\alpha_k,x^k))$. Our strategy is to construct an EP that \emph{locally decouples} the dependence in $\widehat\epsilon(\xi^k,z^k)$ between $\xi^k$ and $z^k$ at the $k$-th iteration.\footnote{Recall that such dependence is produced by the need to evaluate $\widehat F(\xi^k,\cdot)$ along the path $\alpha\mapsto z^k(\alpha)$ in order to choose the stepsize $\alpha_k$. Analogous observations hold for \eqref{algo:armijo:rule2}: $z^k=\overline z_{\beta_k}(\xi^k;\alpha_k,x^k)$.} The intuition behind our decoupling technique is that, although $z^k$ is a function of $(\xi^k,x^k)$, $z^k$ \emph{lies at a ball $\mathbb{B}_k$ centered at any given $x^*\in X^*$ with radius of $\mathcal{O}(\Vert x^k-x^*\Vert+\Vert\widehat\epsilon(\xi^k,x^k)\Vert)$}. Based on this fact and that, by i.i.d. sampling, $\xi^k\perp\perp x^k$, we can decouple $\xi^k$ and $z^k$ using the following guidelines:
\begin{itemize}
\item[(i)] we \emph{condition} on the past information $\alg_k$, noting that $x^k\in\alg_k$ and $\xi^k\perp\perp\alg_k$,
\item[(ii)] we then \emph{control an \emph{EP}} indexed by the ball $\mathbb{B}_k$, 
\item[(iii)] we further note that in item (ii) we must also control $\widehat\epsilon(\xi^k,x^k)$ which affects the radius of the ball $\mathbb{B}_k$. Nevertheless, since $x^k\in\alg_k$ and $\xi^k\perp\perp\alg_k$, $\widehat\epsilon(\xi^k,x^k)$ is a \emph{martingale difference} and, hence, easier to estimate.
\end{itemize}

The developed theory is presented in consecutive sections. The statistical preliminaries used outside the proofs are carefully introduced so to make the presentation as self contained as possible. We refer to the excelent book \cite{boucheron:lugosi:massart2013} by S. Boucheron, G. Lugosi and P. Massart, a standard reference in the area. A global outline is as follows. Typically, if $Z:=\sup_{t\in\mathcal{T}}Z_t$ for a stochastic process $(Z_t)_{t\in\mathcal{T}}$, an upper bound on $\esp[Z]$ is derived under a suitable tail property on the increments of $(Z_t)_{t\in\mathcal{T}}$ and chaining arguments \cite{dudley1967}. In Section \ref{section:L2:norm}, we derive instead an upper bound on $\Lnorm{Z}\ge\esp[Z]$ in Lemma \ref{lemma:lnorm:process}. The main reason to do so is that we assume \emph{heavy-tailed} random operators satisfying Assumption \ref{assumption:holder:continuity}. As a consequence, we will work with the \emph{square} of \emph{sub-Gaussian} random variables (see Definition \ref{def:subGaussian}). In Section \ref{section:Lp:norm}, we apply Lemma \ref{lemma:lnorm:process} derived in Section \ref{section:L2:norm} to obtain the general Lemma \ref{lemma:error:decay:empirical:process}. This lemma provides an \emph{uniform bound over a ball} on the $\mathcal{L}^p$-norm of \emph{empirical error increments of heavy-tailed H\"older continuous operators}, the main stochastic object in this work. \emph{Self-normalization} (see \cite{panchenko2003} and Theorem \ref{thm:panchenko}), variance bounds (Theorem \ref{thm:moment:emp:lugosi}) and a simple decoupling argument based on H\"older's inequality are also needed for that purpose. Finally, the proof of Theorem \ref{thm:variance:error:with:line:search} is given in Section \ref{section:proof:theorem}. It relies on Lemma \ref{lemma:error:decay:empirical:process}, the Burkholder-Davis-Gundy's moment inequality for martingales in Hilbert spaces \cite{burkholder:davis:gundy1972,marinelli:rockner2016} and the ideas of items (i)-(iii) above. 

\subsection{The $\mathcal{L}^2$-norm of suprema of sub-Gaussian processes}\label{section:L2:norm}

In order to bound the expectation or the $\mathcal{L}^2$-norm of $\sup_{t\in\mathcal{T}}Z_t$ for a stochastic process $(Z_t)_{t\in\mathcal{T}}$, it is important to understand the tail behavior of its increments $(Z_t-Z_{t'})_{(t,t')\in\mathcal{T}\times\mathcal{T}}$. We will thus need the definitions of \emph{sub-Gaussian} and \emph{sub-Gamma} random variables. 
\begin{definition}[sub-Gaussian and sub-Gamma random variables]\label{def:subGaussian}
A random variable $Y\in\re$ is called \emph{sub-Gaussian with variance factor $\sigma^2>0$} if, for all $s\in\re$,
$
\ln\esp\left[e^{sY}\right]\le\frac{\sigma^2s^2}{2}.
$
A random variable $Y\in\re$ is called \emph{sub-Gamma on the right tail with variance factor $\sigma^2>0$ and scale parameter $c>0$} if, for all $0<s<\frac{1}{c}$,
$
\ln\esp\left[e^{sY}\right]\le\frac{\sigma^2s^2}{2(1-cs)}.
$
\end{definition}

Hence, a random variable $Y$ is sub-Gaussian if $Y$ and $-Y$ are sub-Gamma on the right tail with scale parameter $c=0$. In order to compute $\mathcal{L}^2$-norms under heavier tails, we will need also the following result which establishes that the centered \emph{square} of a sub-Gaussian random variable is sub-Gamma on the right tail. It follows, e.g., as a corollary of Theorem 2.1 and Remark 2.3 in \cite{hsu:kakade:zhang2012} in the one dimensional setting.
\begin{theorem}[Square of sub-Gaussian random variables]\label{thm:quad:form}
Suppose that $Y\in\re$ is a sub-Gaussian random variable with variance factor $\sigma^2$. Then, 
for all $0\le s<\frac{1}{2\sigma^2}$,
$
\ln\esp\left[e^{sY^2}\right]\le\sigma^2s+\frac{\sigma^4s^2}{1-2\sigma^2s}.
$
\end{theorem}

One celebrated technique to understand $\sup_{t\in\mathcal{T}}Z_t$ for a stochastic process $(Z_t)_{t\in\mathcal{T}}$ is the so called \emph{chaining method} (see e.g. \cite{dudley1967}). This consists in approximating $\mathcal{T}$ by a increasing chain of finer discrete subsets. In this quest, the ``complexity'' of the index set $\mathcal{T}$ plays an important role. This is formalized in the next definition.
\begin{definition}[Metric entropy]\label{definition:metric:entropy}
Let $(\mathcal{T},d)$ be a totally bounded metric space. Given $\theta>0$, a $\theta$\emph{-net} for $\mathcal{T}$ is a finite set $\mathcal{T}_\theta\subset\mathcal{T}$ of maximal cardinality $N(\theta,\mathcal{T})$ such that for all $s,t\in\mathcal{T}_\theta$ with $s\neq t$, one has $\dist(s,t)>\theta$. The $\theta$\emph{-entropy number} is $H(\theta,\mathcal{T}):=\ln N(\theta,\mathcal{T})$. The function $H(\cdot,\mathcal{T})$ is called the \emph{metric entropy} of $\mathcal{T}$.
\end{definition}
In particular, for all $t\in\mathcal{T}$, there is $s\in\mathcal{T}_\theta$ such that $\dist(s,t)\le\theta$. Note that the metric entropy is a nonincreasing real-valued function. The next lemma establishes the metric entropy of the Euclidean unit ball $\mathbb{B}$ of $\re^d$ (see Lemma 13.11 of \cite{boucheron:lugosi:massart2013}).

\begin{lemma}[Metric entropy of Euclidean balls]\label{lemma:entropy}
Let $\mathbb{B}$ be the Euclidean unit ball of $\re^d$. For all $\theta\in(0,1]$, 
$
H(\theta,\mathbb{B})\le d\ln\left(1+\frac{1}{\theta}\right).
$
\end{lemma}
Hence, the ``complexity'' of $\mathbb{B}$ is proportional to $d$, an effect perceived in high-dimensional problems. However, note that $H(\theta,\mathbb{B})$ grows slowly when the discretization precision $\theta$ diminishes. This is a key property in order for the chaining method to work.

Before proving the main Lemma \ref{lemma:lnorm:process} in this section, we state one more needed preliminary result. It bounds the expectation of the maximum of a \emph{finite} number of sub-Gamma random variables (see, e.g., Corollary 2.6 of \cite{boucheron:lugosi:massart2013}). It is an essential lemma while using discretization arguments.
\begin{lemma}[Expectation of maxima of sub-Gamma random variables]\label{lemma:maximal:inequality}
Let $\{Y_i\}_{i=1}^N$ be real-valued sub-Gamma random variables on the right tail with variance factor $\sigma^2>0$ and scale parameter $c>0$. Then 
$$
\esp\left[\max_{i=1,\ldots,N}Y_i\right]\le \sqrt{2\sigma^2\ln N}+c\ln N.
$$
\end{lemma}

\begin{lemma}[$\mathcal{L}^2$-norm of suprema of sub-Gaussian processes]\label{lemma:lnorm:process} 
Let $(\mathcal{T},d)$ be a totally bounded metric space and $\theta:=\sup_{t\in\mathcal{T}}\dist(t,t_0)$ for some $t_0\in\mathcal{T}$. Suppose $(Z_t)_{t\in\mathcal{T}}$ is a continuous stochastic process for which there exist $a,v>0$ and $\delta\in(0,1]$ such that, for all $t,t'\in\mathcal{T}$ and all $\lambda>0$,
\begin{equation}\label{lemma:lnorm:process:eq0}
\ln\esp[\exp\{\lambda(Z_t-Z_{t'})\}]\le a\dist(t,t')^\delta\lambda+\frac{v\dist(t,t')^{2\delta}\lambda^{2}}{2}.
\end{equation}
Then
$$
\Lnorm{\sup_{t\in\mathcal{T}}Z_t-Z_{t_0}}\le
(3\theta)^\delta\sqrt{2(a^2+v)}\left[\frac{1}{2^\delta-1}+\sum_{i=1}^\infty\frac{\sqrt[4]{8H\left(\theta 2^{-i},\mathcal{T}\right)}+2\sqrt{H\left(\theta 2^{-i},\mathcal{T}\right)}}{2^{i\delta}}\right].
$$
\end{lemma}
\begin{proof}
We first note that the continuity of $t\mapsto Z_t$ and separability of $\mathcal{T}$ imply that, for any continuous function $f$, $\sup_{t\in\mathcal{T}}f(Z_t)$ is measurable since it equals $\sup_{t\in\mathcal{T'}}f(Z_t)$ for a countable dense subset $\mathcal{T}'$ of $\mathcal{T}$. 

Set $\mathcal{T}_0:=\{t_0\}$. Given $i\in\mathbb{N}$, we set $\theta_i:=\theta2^{-i}$ and denote by $\mathcal{T}_i$ a $\theta_i$-net for $\mathcal{T}$ with maximal cardinality $N(\theta_i,\mathcal{T})$. We also denote by $\Pi_i:\mathcal{T}\rightarrow\mathcal{T}_i$ the metric projection associated to $\dist$, that is, for any $t\in\mathcal{T}$, $\Pi_i(t)\in\argmin_{t'\in\mathcal{T}_i}\dist(t,t')$. By  the definition of a net, we have that, for all $t\in\mathcal{T}$ and $i\in\mathbb{N}$,
$
\dist(t,\Pi_i(t))\le\theta_i.
$
By the triangular inequality, this implies that for all $t\in\mathcal{T}$ and $i\in\mathbb{N}$,
\begin{equation}\label{lemma:lnorm:process:eq4}
\dist(\Pi_i(t),\Pi_{i+1}(t))\le\theta_i+\theta_{i+1}=3\theta_{i+1}.
\end{equation}

For any $t\in\mathcal{T}$, $\lim_{i\rightarrow\infty}\Pi_i(t)=t$ and $\Pi_0(t)=t_0$ imply that 
\begin{equation*}
Z_t=Z_{t_0}+\sum_{j=0}^\infty(Z_{\Pi_{i+1}(t)}-Z_{\Pi_i(t)}).
\end{equation*}
In the following, we denote $\Delta_i(t):=Z_{\Pi_{i+1}(t)}-Z_{\Pi_i(t)}$ for all $i\in\mathbb{N}$ and $t\in\mathcal{T}$. The above equality implies that
$
(Z_t-Z_{t_0})^2=\sum_{i=0}^\infty\sum_{k=0}^\infty\Delta_i(t)\Delta_k(t).
$
Hence,
\begin{eqnarray}
\esp\left[\sup_{t\in\mathcal{T}}(Z_t-Z_{t_0})^2\right]&\le &
\sum_{i=0}^\infty\sum_{k=0}^\infty\esp\left[\sup_{t\in\mathcal{T}}\left\{\Delta_i(t)\Delta_k(t)\right\}\right]\nonumber\\
&\le &\sum_{i=0}^\infty\sum_{k=0}^\infty\Lnorm{\sup_{t\in\mathcal{T}}|\Delta_i(t)|}\cdot\Lnorm{\sup_{t\in\mathcal{T}}|\Delta_k(t)|}\nonumber\\
&=&\left[\sum_{i=0}^\infty\Lnorm{\sup_{t\in\mathcal{T}}|\Delta_i(t)|}\right]^2,\label{lemma:lnorm:process:eq2}
\end{eqnarray}
using H\"older's inequality in the second inequality. 

Fix $i\in\mathbb{N}$. Since $N(\theta_{i},\mathcal{T})\le N(\theta_{i+1},\mathcal{T})$, we have that
\begin{equation}\label{lemma:lnorm:process:eq3}
|\{(\Pi_i(t),\Pi_{i+1}(t)):t\in\mathcal{T}\}|\le N(\theta_{i+1},\mathcal{T})^2
=e^{2H(\theta_{i+1})}.
\end{equation}
Relations \eqref{lemma:lnorm:process:eq0} and \eqref{lemma:lnorm:process:eq4} imply that, for all $t\in\mathcal{T}$,
\begin{eqnarray*}
\ln\esp\left[e^{\lambda\Delta_i(t)}\right]\le a\dist\left(\Pi_i(t),\Pi_{i+1}(t)\right)^{\delta}\lambda+\frac{v\dist\left(\Pi_i(t),\Pi_{i+1}(t)\right)^{2\delta}\lambda^2}{2}
\le a_i\lambda+\frac{v_i\lambda^2}{2},
\end{eqnarray*}
where we have defined $a_i:=a(3\theta_{i+1})^\delta$ and $v_i:=v(3\theta_{i+1})^{2\delta}$. The above relation implies that, for all $t\in\mathcal{T}$, $\Delta_i(t)-a_i$ is sub-Gaussian with variance factor $v_i$. This, Theorem \ref{thm:quad:form}, the bound $\Delta_i(t)^2\le2[\Delta_i(t)-a_i]^2+2a_i^2$ and the change of variables $\lambda\mapsto2\lambda$ imply that, for all $t\in\mathcal{T}$ and $0<\lambda<\frac{1}{4v_i}$,
\begin{equation}\label{lemma:lnorm:process:eq5}
\ln\esp\left[e^{\lambda\Delta_i(t)^2}\right]\le 2(a_i^2+v_i)\lambda+\frac{4v_i^2\lambda^2}{(1-4v_i\lambda)}, 
\end{equation}
that is, for all $t\in\mathcal{T}$, $\Delta_i(t)^2-2(a_i^2+v_i)$ is sub-Gamma on the right tail with variance factor $8v_i^2$ and scale parameter $4v_i$. Relations \eqref{lemma:lnorm:process:eq3}-\eqref{lemma:lnorm:process:eq5} 
and Lemma \ref{lemma:maximal:inequality} imply further that
\begin{eqnarray*}
\esp\left[\sup_{t\in\mathcal{T}}\Delta_i(t)^2\right]&\le &
2(a_i^2+v_i)+\sqrt{2\cdot8v_i^2\cdot2H(\theta_{i+1},\mathcal{T})}
+4v_i\cdot2H(\theta_{i+1},\mathcal{T})\\
&\le &2\cdot9^\delta(a^2+v)\left[\theta_{i+1}^{2\delta}+\theta_{i+1}^{2\delta}\sqrt{8H(\theta_{i+1},\mathcal{T})}+4\theta_{i+1}^{2\delta} H(\theta_{i+1},\mathcal{T})\right].
\end{eqnarray*}
Taking the square root in the above relation we get
\begin{equation}\label{lemma:lnorm:process:eq6}
\Lnorm{\sup_{t\in\mathcal{T}}|\Delta_i(t)|}\le
3^\delta\sqrt{2(a^2+v)}\left[\theta_{i+1}^{\delta}+\theta_{i+1}^{\delta}\sqrt[4]{8H(\theta_{i+1},\mathcal{T})}+2\theta_{i+1}^{\delta}\sqrt{H(\theta_{i+1},\mathcal{T})}\right].
\end{equation}

We now take the square root in \eqref{lemma:lnorm:process:eq2} and use \eqref{lemma:lnorm:process:eq6}, valid for any $i\in\mathbb{N}$, obtaining
\begin{eqnarray*}
\Lnorm{\sup_{t\in\mathcal{T}}Z_t-Z_{t_0}}&\le &
3^\delta\sqrt{2(a^2+v)}\left[\sum_{i=1}^\infty\theta_{i}^{\delta}+\sum_{i=1}^\infty\theta_{i}^{\delta}\sqrt[4]{8H(\theta_{i},\mathcal{T})}+2\sum_{i=1}^\infty\theta_{i}^{\delta}\sqrt{H(\theta_{i},\mathcal{T})}\right].
\end{eqnarray*}
To finish the proof, we use $\theta_i=\theta 2^{-i}$ and $\sum_{i=1}^\infty\theta_{i}^{\delta}=\frac{\theta^\delta}{2^\delta-1}$ in the above inequality.
\end{proof}

\subsection{Heavy-tailed H\"older continuous operators: self-normalization and $\mathcal{L}^q$-norms of suprema of EPs}\label{section:Lp:norm}

We will now focus on bounds of EPs associated to sums of the form $x\mapsto\sum_{j=1}^N\frac{F(\xi_j,x)-T(x)}{N}$, where $\{\xi_j\}_{j=1}^N$ is an i.i.d. sample of $\probn$ and $F:\Xi\times X\rightarrow\re^d$ satisfies Assumption \ref{assumption:holder:continuity}. The main result proved in this section is Lemma \ref{lemma:error:decay:empirical:process}. Its proof will need Lemma \ref{lemma:lnorm:process} and the following theorem (see Theorem 15.14 in \cite{boucheron:lugosi:massart2013}).

\begin{theorem}[$\mathcal{L}^q$-norm for suprema of EPs]\label{thm:moment:emp:lugosi}
Let $\{X_j\}_{j=1}^N$ be an independent sequence of stochastic processes $X_j:=(X_{j,t})_{t\in\mathcal{T}}$ indexed by a countable set $\mathcal{T}$ with real-valued random components $X_{j,t}$ such that $\esp[X_{j,t}]=0$ and $\esp[X_{j,t}^2]<\infty$ for all $t\in\mathcal{T}$ and $j\in[N]$. Define $Z:=\sup_{t\in\mathcal{T}}\left|\sum_{j=1}^NX_{j,t}\right|$ and  
\begin{eqnarray*}
M:=\max_{j\in[N]}\sup_{t\in\mathcal{T}}|X_{j,t}|,\quad\quad
\widehat\sigma^2 :=\sup_{t\in\mathcal{T}}\sum_{j=1}^N\esp\left[X_{j,t}^2\right].
\end{eqnarray*}
Set $\kappa:=\frac{\sqrt{e}}{2(\sqrt{e}-1)}<1.271$. Then, for all $q\ge2$,
$$
\Lqnorm{Z}\le2\esp[Z]+2\sqrt{2\kappa q}\widehat\sigma+4\sqrt{\kappa q}\Lnorm{M}+20\kappa q\Lqnorm{M}. 
$$
\end{theorem}

In order to cope with a heavy-tailed $\mathsf{L}(\xi)$ in Assumption \ref{assumption:holder:continuity}, we will need Theorem \ref{thm:panchenko}, a result due to Panchenko (see Theorem 1 in \cite{panchenko2003} or Theorem 12.3 in \cite{boucheron:lugosi:massart2013}). It establishes a sub-Gaussian tail for the deviation of an EP around its mean after a proper \emph{normalization} with respect to a \emph{random} quantity $V$. In our set-up, the standard H\"older continuous assumption turns out to be sufficient to estimate this quantity. 
\begin{theorem}[Panchenko's inequality for self-normalized EPs]\label{thm:panchenko}
Consider a countable family $\mathcal{G}$ of measurable functions $f:\Xi\rightarrow\re$ such that $\probn f(\cdot)^2<\infty$. Let $\{\xi_j\}_{j=1}^N$ and $\{\eta_j\}_{j=1}^N$ be i.i.d. samples of $\probn$ independent of each other. Set 
$$
Y:=\sup_{f\in\mathcal{G}}\sum_{j=1}^Nf(\xi_j),\quad\mbox{ and }\quad V:=\esp\left\{\sup_{f\in\mathcal{G}}\sum_{j=1}^N\left[f(\xi_j)-f(\eta_j)\right]^2\Bigg|\xi_1,\ldots,\xi_N\right\}.
$$
Then there exists an universal constant $\mathsf{c}>0$ such that, for all $t>0$,
$$
\prob\left\{Y-\esp[Y]\ge\mathsf{c}\sqrt{V(1+t)}\right\}\bigvee\prob\left\{Y-\esp[Y]\le -\mathsf{c}\sqrt{V(1+t)}\right\}\le e^{-t}.
$$
\end{theorem}

Finally, before proving Lemma \ref{lemma:error:decay:empirical:process}, we will need Theorem \ref{thm:characterization:subgaussian} which is a standard tail characterization of sub-Gaussian random variables. Theorem 2.1 in \cite{boucheron:lugosi:massart2013} gives a proof for the case $\esp[\tilde Y]=0$. The adaptation for the general case is immediate using the facts that $\esp[e^{-t\tilde Y}]\ge e^{-t\esp[\tilde Y]}$ by Jensen's inequality, the integral formula $\esp[\tilde Y]\le\esp[|\tilde Y|]=\int_0^\infty\prob(|\tilde Y|>t)\dist t$ and $\int_0^\infty e^{-\frac{t^2}{2}}\dist t=\sqrt{\frac{\pi}{2}}$.
\begin{theorem}[Tail characterization of sub-Gaussian random variables]\label{thm:characterization:subgaussian}
If $\tilde Y\in\re$ is a random variable such that, for some $v>0$ and for all $t>0$,
$$
\prob\left\{\tilde Y\ge\sqrt{2vt}\right\}\bigvee\prob\left\{\tilde Y\le -\sqrt{2vt}\right\}\le e^{-t},
$$ 
then, for all $t>0$, we have 
$
\ln\esp\left[e^{t\tilde Y}\right]\le e^{\sqrt{\frac{v\pi}{2}}t+8vt^2}.
$
\end{theorem}

We now prove the main lemma of this section. It uses Lemma \ref{lemma:lnorm:process} and Theorems \ref{thm:moment:emp:lugosi}-\ref{thm:characterization:subgaussian}.
\begin{lemma}[Local uniform bound for the $\mathcal{L}^p$-norm of empirical error increments]\label{lemma:error:decay:empirical:process}
Consider definition \eqref{equation:expected:valued:objective} and let $\xi^N:=\{\xi_j\}_{j=1}^N$ be an i.i.d. sample from $\probn$. Suppose that Assumption \ref{assumption:holder:continuity} holds and recall definitions \eqref{equation:oracle:error}-\eqref{equation:empirical:mean:operator:&:error}. Given $x_*\in X$ and $R>0$, we define
\begin{equation}
Z:=\sup_{x\in \mathbb{B}[x_*,R]\cap X}\left\Vert\widehat\epsilon(\xi^N,x)-\widehat\epsilon(\xi^N,x_*)\right\Vert.\label{def:empirical:process:statement}
\end{equation}
Then
$$
\Lpnorm{Z}\lesssim\left[\frac{3^\delta\sqrt{d}L_2}{\sqrt{\delta}\left(\sqrt{2}^\delta-1\right)}+\sqrt{p}L_2+pL_p\right]\frac{R^\delta}{\sqrt{N}}.
$$
\end{lemma}
\begin{proof}
A first step is to rewrite $Z$ as  the supremum of a suitable EP and  use Theorem \ref{thm:moment:emp:lugosi}. In the following, we define the set $\mathbb{B}_X:=\{u\in\mathbb{B}:x_*+Ru\in X\}$ for $x_*\in X$ and $R>0$ as stated in the theorem. Note that
\begin{eqnarray}
Z&=&\sup_{u\in \mathbb{B}_X}\frac{1}{N}\left\Vert\sum_{j=1}^N\epsilon(\xi_j,x_*+Ru)-\epsilon(\xi_j,x_*)\right\Vert\nonumber\\
&=&\sup_{u\in \mathbb{B}_X}\frac{1}{N}\sup_{y\in \mathbb{B}}\left\langle\sum_{j=1}^N\epsilon(\xi_j,x_*+Ru)
-\epsilon(\xi_j,x_*),y\right\rangle\label{def:empirical:process}\\
&=&\sup_{(u,y)\in \mathbb{B}_X\times\mathbb{B}}\frac{1}{N}\sum_{j=1}^N\left\langle\epsilon(\xi_j,x_*+Ru)-\epsilon(\xi_j,x_*),y\right\rangle,\nonumber
\end{eqnarray} 
where the second equality uses the fact that $\Vert\cdot\Vert=\sup_{y\in \mathbb{B}}\langle y,\cdot\rangle$. Next, we define the index set $\mathcal{T}:=\mathbb{B}_X\times\mathbb{B}$ and, for every $j\in[N]$ and $t:=(u,y)\in\mathbb{B}_X\times\mathbb{B}$, we define the random variables 
\begin{eqnarray}
X_{j,t}&:=&\frac{1}{N}\left\langle\epsilon(\xi_j,x_*+Ru)-
\epsilon(\xi_j,x_*),y\right\rangle,\label{def:empirical:process:xjt}\\
\tilde Z_t&:=&\sum_{j=1}^NX_{j,t}.\label{def:empirical:process:zt}
\end{eqnarray}
From Assumption \ref{assumption:holder:continuity}, it is not difficult to show that, for every $j\in[N]$, the process $\mathcal{T}\ni t\mapsto X_{j,t}$ is H\"older continuous with respect to the metric 
\begin{equation}\label{equation:metric:process}
\dist(t,t'):=\Vert u-u'\Vert+\Vert y-y'\Vert.
\end{equation}
This fact, the separability of $\mathcal{T}$ and \eqref{def:empirical:process}, imply that $(\tilde Z_t)_{t\in\mathcal{T}}$ is a continuous process and $Z=\sup_{t\in\mathcal{T}_0}\tilde Z_t=\sup_{t\in\mathcal{T}_0}\left|\tilde Z_t\right|$ is measurable, where $\mathcal{T}_0$ is a dense countable subset of $\mathcal{T}$. Hence, we may assume next that $\mathcal{T}$ is countable without loss on generality. Our next objective is to use Theorem \ref{thm:moment:emp:lugosi}, bounding $\Lpnorm{Z}$ in terms of $\esp[Z]$, $M$ and $\widehat\sigma^2$.

\textsf{\textbf{PART 1} (An upper bound on $\esp[Z]$)}: To bound $\esp[Z]$ we will need Lemma \ref{lemma:lnorm:process} and Theorems \ref{thm:panchenko}-\ref{thm:characterization:subgaussian}. At this point, let's fix $t=(u,y)\in\mathcal{T}$ and $t'=(u',y')\in\mathcal{T}$ and define the measurable function
$$
f(\cdot):=\frac{1}{N}\left\langle\epsilon(\cdot,x_*+Ru)-
\epsilon(\cdot,x_*),y\right\rangle-\frac{1}{N}\left\langle\epsilon(\cdot,x_*+Ru')-
\epsilon(\cdot,x_*),y'\right\rangle.
$$
We have that $\probn f(\cdot)^2<\infty$ since $\Lnorm{\Vert F(\xi,\cdot)\Vert}<\infty$ on $X$ (Assumption \ref{assumption:holder:continuity}). By construction and \eqref{def:empirical:process:xjt}-\eqref{def:empirical:process:zt}, we have $f(\xi_j)=X_{j,t}-X_{j,t'}$ for all $j\in[N]$ and $\tilde Z_t-\tilde Z_{t'}=\sum_{j=1}^Nf(\xi_j)$. Note also that $\esp\left[\sum_{j=1}^Nf(\xi_j)\right]=0$, using \eqref{equation:expected:valued:objective}, \eqref{equation:oracle:error} and that $\{\xi_j\}_{j\in[N]}$ is an i.i.d. sample of $\probn$. 

The previous observations allow us to claim Theorem \ref{thm:panchenko} with $\mathcal{G}:=\{f\}$ and $Y:=\sum_{j=1}^Nf(\xi_j)$. Precisely, if $\{\eta_j\}_{j=1}^N$ is an i.i.d. sample from $\probn$ which is independent of $\{\xi_j\}_{j=1}^N$, then Theorem \ref{thm:panchenko} and $\esp\left[\sum_{j=1}^Nf(\xi_j)\right]=0$ imply that, for all $\lambda>0$,
\begin{equation}\label{thm:moment:inequality:ez:eq1}
\prob\left\{\sum_{j=1}^Nf(\xi_j)\ge \mathsf{c}\sqrt{V(1+\lambda)}\right\}\bigvee\prob\left\{\sum_{j=1}^Nf(\xi_j)\le-\mathsf{c}\sqrt{V(1+\lambda)}\right\}\le e^{-\lambda},
\end{equation}
for some universal constant $\mathsf{c}>0$ and
$$
V:=\esp\left[\sum_{j=1}^N\left[f(\xi_j)-f(\eta_j)\right]^2\Bigg|\xi_1,\ldots,\xi_N\right].
$$

We will now give an upper bound on $V$. Given $\xi\in\Xi$, \eqref{equation:expected:valued:objective},  \eqref{equation:oracle:error} and H\"older continuity of $F(\xi,\cdot)$ and $T$ (Assumption \ref{assumption:holder:continuity} and Lemma \ref{lemma:holder:continuity:mean:std:dev}) imply that $\epsilon(\xi,\cdot)$ is $(\mathsf{L}(\xi)+L, \delta)$-H\"older continuous on $X$. This, definition of $f$ and the facts that $y,y,u,u'\in\mathbb{B}$ and $x_*+Ru,x_*+Ru'\in X$ imply that, for all $j\in[N]$ and $\Delta f_j:=N\left|[f(\xi_j)-f(\eta_j)]\right|$,
\begin{eqnarray*}
\Delta f_j&\le &\left|\left\langle\epsilon(\xi_j,x_*+Ru)-\epsilon(\xi_j,x_*)-\epsilon(\eta_j,x_*+Ru)+\epsilon(\eta_j,x_*),y-y'\right\rangle\right|\\
&+&\left|\left\langle\epsilon(\xi_j,x_*+Ru)-\epsilon(\xi_j,x_*+Ru')-\epsilon(\eta_j,x_*+Ru)+\epsilon(\eta_j,x_*+Ru'),y'\right\rangle\right|\\
&\le &\left[\mathsf{L}(\xi_j)+\mathsf{L}(\eta_j)+2L\right]R^\delta\left[\Vert y-y'\Vert+\Vert u-u'\Vert^\delta\right]\\
&\le &\left[\mathsf{L}(\xi_j)+\mathsf{L}(\eta_j)+2L\right]R^\delta2^{1-\delta}\left[\Vert y-y'\Vert^{\frac{1}{\delta}}+\Vert u-u'\Vert\right]^\delta\\
&\le &\left[\mathsf{L}(\xi_j)+\mathsf{L}(\eta_j)+2L\right]R^\delta2^{(1-\delta)}\left[\Vert y-y'\Vert+\Vert u-u'\Vert\right]^\delta,
\end{eqnarray*}
where we used concavity of $\re_+\ni x\mapsto x^\delta$ in third inequality and the fact that $\Vert y-y'\Vert^{\frac{1}{\delta}}\le2^{\frac{(1-\delta)}{\delta}}\Vert y-y'\Vert$ for $y,y'\in\mathbb{B}$ in last inequality. We take squares in the above inequality, use relation $(\sum_{i=1}^3a_i)^2\le3\sum_{i=1}^3a_i^2$ and definitions of $V$ and \eqref{equation:metric:process}. We thus obtain 
\begin{eqnarray}
V&\le &\frac{3\cdot4^{1-\delta}R^{2\delta}\dist(t,t')^{2\delta}}{N}\left\{\sum_{j=1}^N\frac{\mathsf{L}(\xi_j)^2}{N}+\sum_{j=1}^N\frac{\esp\left[\mathsf{L}(\eta_j)^2|\xi_1,\ldots,\xi_N\right]}{N}+4L^2\right\}\nonumber\\
&=&\frac{3\cdot4^{1-\delta}R^{2\delta}\dist(t,t')^{2\delta}W_N^2}{N},
\label{thm:moment:inequality:ez:eq2}
\end{eqnarray}
where we have defined 
\begin{equation}\label{thm:moment:inequality:ez:WN}
W_N:=\sqrt{\frac{1}{N}\sum_{j=1}^N\mathsf{L}(\xi_j)^2+\Lnorm{\mathsf{L}(\xi)}^2+4L^2},
\end{equation}
and used that $\{\eta_j\}_{j\in[N]}$ is an i.i.d. sample of $\probn$ independent of $\{\xi_j\}_{j\in[N]}$.

Set $\tilde Y:=\frac{\tilde Z_t-\tilde Z_{t'}}{W_N}-\frac{\sqrt{3}\mathsf{c}2^{1-\delta}R^{\delta}\dist(t,t')^{\delta}}{\sqrt{N}}$. Relations \eqref{thm:moment:inequality:ez:eq1}-\eqref{thm:moment:inequality:ez:eq2} and $\sum_{j=1}^Nf(\xi_j)=\tilde Z_t-\tilde Z_{t'}$, together with $\sqrt{1+\lambda}\le1+\sqrt{\lambda}$ for $\lambda>0$, imply that 
\begin{equation*}
\prob\left\{\tilde Y\ge\frac{\sqrt{3}\mathsf{c}2^{1-\delta}R^{\delta}\dist(t,t')^{\delta}}{\sqrt{N}}\sqrt{\lambda}\right\}\bigvee\prob\left\{\tilde Y\le-\frac{\sqrt{3}\mathsf{c}2^{1-\delta}R^{\delta}\dist(t,t')^{\delta}}{\sqrt{N}}\sqrt{\lambda}\right\}\le e^{-\lambda}.
\end{equation*}
The above relation and Theorem \ref{thm:characterization:subgaussian} imply that for some universal constants $C_1,C_2>0$ and for all $\lambda>0$,
\begin{equation}\label{thm:moment:inequality:ez:eq4}
\ln\esp\left[\exp\left\{\frac{(\tilde Z_t-\tilde Z_{t'})}{W_N}\lambda\right\}\right]\le\frac{C_12^{1-\delta}R^{\delta}\dist(t,t')^{\delta}}{\sqrt{N}}\lambda+\frac{C_2^24^{1-\delta}R^{2\delta}\dist(t,t')^{2\delta}}{2N}\lambda^2.
\end{equation}

We now observe that \eqref{thm:moment:inequality:ez:eq4} holds for any $t,t'\in\mathcal{T}$. Inequality \eqref{thm:moment:inequality:ez:eq4} and Lemma \ref{lemma:lnorm:process} with $(\mathcal{T},\dist)$ as defined in \eqref{equation:metric:process}, the continuous process $\mathcal{T}\ni t\mapsto Z_t:=\frac{\tilde Z_t}{W_N}$, $t_0:=(0,0)$, $\theta:=\sup_{t\in\mathcal{T}}\dist(t,0)\le2$, $a:=\frac{C_1 2^{1-\delta}R^{\delta}}{\sqrt{N}}$ and $v:=\frac{C_2^2 4^{1-\delta}R^{2\delta}}{N}$ imply that
\begin{equation}\label{thm:moment:inequality:ez:eq4'}
\Lnorm{\sup_{t\in\mathcal{T}}Z_t}\le
\frac{\sqrt{2}C 2^{1-\delta}(6R)^\delta}{\sqrt{N}}\left[\frac{1}{2^\delta-1}+\sum_{i=1}^\infty\frac{\sqrt[4]{8H\left(2^{-i+1},\mathcal{T}\right)}+2\sqrt{H\left(2^{-i+1},\mathcal{T}\right)}}{2^{i\delta}}\right],
\end{equation}
where we defined $C=\sqrt{C_1^2+C_2^2}$ and used the fact that $Z_{t_0}=\frac{\tilde Z_{t_0}}{W_N}=0$. From Lemma \ref{lemma:entropy} and the fact that, for any $\theta>0$, $H(\theta,\mathbb{B}_X\times\mathbb{B})\le H(\theta,\mathbb{B}_X)+H(\theta,\mathbb{B})\le2H(\theta,\mathbb{B})$, we also have that 
\begin{eqnarray}
\sum_{i=1}^\infty\frac{\sqrt[4]{8H\left(2^{-i+1},\mathcal{T}\right)}+2\sqrt{H\left(2^{-i+1},\mathcal{T}\right)}}{2^{i\delta}}&\lesssim &\sqrt{d}\sum_{i=1}^\infty\frac{\sqrt{\ln(1+2^{i+1})}}{2^{i\delta}}\nonumber\\
&\lesssim &\sqrt{d}\sum_{i=1}^\infty\frac{\sqrt{i+1}}{2^{i\delta}}\lesssim
\frac{\sqrt{d/\delta}}{2^{\frac{\delta}{2}}-1},
\end{eqnarray}
where we used the facts that $\ln(1+x)\le x$, $\sqrt{i+1}\le\frac{2^{\frac{i\delta}{2}}}{\sqrt{\delta}\ln 2}$ and\footnote{The previous fact can be derived from the inequality $2^x\ge1+(\ln2)x$.} $\sum_{i=1}^\infty2^{-\frac{i\delta}{2}}=\frac{1}{2^{\frac{\delta}{2}}-1}$.

H\"older's inequality implies that 
\begin{equation}\label{thm:moment:inequality:ez:eq5}
\esp[Z]=\esp\left[\sup_{t\in\mathcal{T}}|\tilde Z_t|\right]
=\esp\left[\sup_{t\in\mathcal{T}}\left|Z_t\right|\cdot W_N\right]
\le\Lnorm{\sup_{t\in\mathcal{T}}\left|Z_t\right|}\cdot\Lnorm{W_N}.
\end{equation}
Since $\{\xi_j\}_{j\in[N]}$ is an i.i.d. sample from $\probn$, we also obtain from \eqref{thm:moment:inequality:ez:WN} that $\Lnorm{W_N}\lesssim \Lnorm{\mathsf{L}(\xi)}+L=L_2$. Finally, this, relations \eqref{thm:moment:inequality:ez:eq4'}-\eqref{thm:moment:inequality:ez:eq5} and the facts that $2^{1-\delta} 6^\delta=2\cdot3^\delta$ and $2^{\delta}-1\ge2^{\frac{\delta}{2}}-1$ imply that
\begin{eqnarray}
\esp[Z]\lesssim\frac{\sqrt{d}(3R)^\delta L_2}{\left(2^{\frac{\delta}{2}}-1\right)\sqrt{\delta N}}.\label{lemma:error:decay:empirical:process:EZ}
\end{eqnarray}
	
\textsf{\textbf{PART 2} (An upper bound on $M$ and $\widehat\sigma^2$)}: From the definition of $\widehat\sigma^2$ in Theorem \ref{thm:moment:emp:lugosi} and \eqref{def:empirical:process:xjt}, we get
\begin{eqnarray}\label{lemma:error:decay:empirical:process:sigma}
\widehat\sigma &=&\sqrt{\sup_{(u,y)\in\mathcal{T}}\frac{1}{N^2}\sum_{j=1}^N\esp\left[\left\langle\epsilon(\xi_j,x_*+Ru)-\epsilon(\xi_j,x_*),y\right\rangle^2\right]}\nonumber\\
&\le &\sqrt{\frac{1}{N}\sup_{(u,y)\in\mathcal{T}}\esp\left[\sum_{j=1}^N\frac{(\mathsf{L}(\xi_j)+L)^2}{N}R^{2\delta}\Vert u\Vert^{2\delta}\Vert y\Vert^2\right]}\nonumber\\
&\le &\frac{R^{\delta}(\Lnorm{\mathsf{L}(\xi)}+L)}{\sqrt{N}},
\end{eqnarray}
where we used the fact that $\Vert\epsilon(\xi_j,x_*+Ru)-\epsilon(\xi_j,x_*)\Vert\le[\mathsf{L}(\xi_j)+L]R^\delta$ for $u\in\mathbb{B}_X$ (Assumption \ref{assumption:holder:continuity} and Lemma \ref{lemma:holder:continuity:mean:std:dev}) in first inequality and the fact that $\{\xi_j\}_{j\in[N]}$ is an i.i.d. sample of $\probn$ in the last inequality.

From the definition of $M$ in Theorem \ref{thm:moment:emp:lugosi} and \eqref{def:empirical:process:xjt}, we get
\begin{eqnarray*}
\Lpnorm{M}^p&=&\esp\left[\left(\max_{j\in[N]}\sup_{t\in\mathcal{T}}|X_{j,t}|\right)^p\right]
=\esp\left[\max_{j\in[N]}\sup_{t\in\mathcal{T}}|X_{j,t}|^p\right]\nonumber\\
&\le &\frac{1}{N^p}\sum_{j=1}^N\esp\left[\sup_{t\in\mathcal{T}}\left|\left\langle\epsilon(\xi_j,x_*+Ru)
-\epsilon(\xi_j,x_*),y\right\rangle\right|^p\right]\nonumber\\
&\le &\frac{1}{N^{p-1}}\sup_{(u,y)\in\mathcal{T}}\esp\left[\sum_{j=1}^N\frac{(\mathsf{L}(\xi_j)+L)^p}{N}R^{p\delta}\Vert u\Vert^{p\delta}\Vert y\Vert^p\right]\nonumber\\
&\le &\frac{R^{p\delta}\Lpnorm{\mathsf{L}(\xi)+L}^p}{N^{p-1}},
\end{eqnarray*}
where, again, we used the fact that $\Vert\epsilon(\xi_j,x_*+Ru)-\epsilon(\xi_j,x_*)\Vert\le[\mathsf{L}(\xi_j)+L]R^\delta$ for $u\in\mathbb{B}_X$ in second inequality and the fact that $\{\xi_j\}_{j\in[N]}$ is an i.i.d. sample of $\probn$ in the last inequality. We take the $p$-th root in the above inequality and note that for $p\ge2$ we have $N^{\frac{p-1}{p}}\ge \sqrt{N}$, obtaining
\begin{equation}\label{lemma:error:decay:empirical:process:m}
\Lpnorm{M}\le\frac{(\Lpnorm{\mathsf{L}(\xi)}+L) R^{\delta}}{\sqrt{N}}.
\end{equation}

From \eqref{lemma:error:decay:empirical:process:EZ}-\eqref{lemma:error:decay:empirical:process:m} and definitions of $L_2$ and $L_p$ in Assumption \ref{assumption:holder:continuity}, we obtain the required claim.
\end{proof}

\subsection{The proof of Theorem \ref{thm:variance:error:with:line:search}}\label{section:proof:theorem}

With the theory developed in Sections \ref{section:L2:norm}-\ref{section:Lp:norm}, we are now ready to prove Theorem \ref{thm:variance:error:with:line:search}. We shall use Lemma \ref{lemma:error:decay:empirical:process} and follow the ideas of items (i)-(iii) presented in the introduction of Section \ref{section:empirical:process:theory:DSSA}. We will also need the next Lemma \ref{lemma:decay:empirical:error} which controls the oracle's empirical error. Its control is easier than the oracle's correlated error, since it defines a martingale difference. Its proof uses Assumption \ref{assumption:holder:continuity} and a version of Burkholder-Davis-Gundy's inequality in Hilbert spaces (see \cite{burkholder:davis:gundy1972,marinelli:rockner2016}). 

\begin{theorem}[Burkholder-Davis-Gundy inequality in $\re^d$]\label{thm:BDG}
Let $\Vert\cdot\Vert$ be the Euclidean norm in $\re^d$. Then, for all $q\ge2$, there exists $C_q>0$ such that for any vector-valued martingale $\{y_j\}_{j=0}^N$ adapted to the filtration $\{\mathcal{G}_j\}_{j=1}^N$ with $y_0=0$, it holds that 
$$
\Lqnorm{\sup_{j\leq N}\Vert y_j\Vert}\leq C_q\,\Lqnorm{\sqrt{\sum_{j=1}^N\Vert y_j-y_{j-1}\Vert^2}}\le C_q\,\sqrt{\sum_{j=1}^N\,\Lqnorm{\Vert y_j-y_{j-1}\Vert}^2}.
$$
\end{theorem}

\begin{lemma}[Local bound for the $\mathcal{L}^q$-norm of the empirical error]\label{lemma:decay:empirical:error}
Consider definition \eqref{equation:expected:valued:objective} and let $\xi^N:=\{\xi_j\}_{j=1}^N$ be an i.i.d. sample from $\probn$. Suppose that Assumption \ref{assumption:holder:continuity} holds and take $q\in[p,2p]$ such that the integrability conditions of Assumption \ref{assumption:holder:continuity} are satisfied. Recall definitions in \eqref{equation:oracle:error}-\eqref{equation:empirical:mean:operator:&:error} and definition of $C_q$ in Theorem \ref{thm:BDG}. Set $C_2:=1$ if $q=p=2$. Then, for any $x,x_*\in X$, 
$$
\Lqnorm{\left\Vert\widehat\epsilon(\xi^N,x)\right\Vert}\le C_q\frac{\sigma_q(x^*)
+L_q\Vert x-x^*\Vert^\delta}{\sqrt{N}}.
$$
\end{lemma}
\begin{proof}
We define the $\re^d$-valued process $\{y_t\}_{t=0}^N$ by $y_0=0$ and $y_t:=\sum_{j=1}^t\frac{\epsilon(\xi_j,x)}{N}$ for $t\in[N]$ and the filtration $\mathcal{G}_t:=\sigma(y_0,\ldots,y_t)$ for $t\in\{0\}\cup[N]$. Since $\{\xi_j\}_{j=1}^N$ is an i.i.d. sample of $\probn$, $\{y_t,\mathcal{G}_t\}_{t=0}^N$ is a $\re^d$-valued martingale whose increments satisfy
\begin{equation*}
\Lqnorm{\Vert y_t-y_{t-1}\Vert} = \Lqnorm{\frac{\Vert\epsilon(\xi,x)\Vert}{N}}\le
\frac{\Lqnorm{\Vert\epsilon(\xi,x_*)\Vert}+L_q\Vert x-y\Vert^\delta}{N},
\end{equation*}
using that $\Lqnorm{\Vert\epsilon(\xi,\cdot)\Vert}$ is H\"older continuous with modulus $L_q=\Lqnorm{\mathsf{L}(\xi)}+L$ and exponent $\delta$ (Lemma \ref{lemma:holder:continuity:mean:std:dev}) in the inequality. The required claim follows from the above relation, Theorem \ref{thm:BDG} and $\widehat\epsilon(\xi^N,x)=y_N$. We note that if $q=2$, then the linearity of the expectation, the Pythagorean identity (valid for the Euclidean norm) and independence imply the sharper equality $\Lnorm{\left\Vert\widehat\epsilon(\xi^N,x)\right\Vert}=\frac{\Lnorm{\Vert\epsilon(\xi,x)\Vert}}{\sqrt{N}}$. This fact and Lemma \ref{lemma:holder:continuity:mean:std:dev} imply the claim of the lemma with $C_2=1$.
\end{proof}

\begin{proof}[Proof of Theorem \ref{thm:variance:error:with:line:search}]
We fix $x\in X$ and $x^*\in X^*$ as stated in the theorem and set $z^N:=z(\xi^N;\alpha_N,x)$ and $\overline z^N:=\overline z(\xi^N;\alpha_N,x)$. In the following, we only give a proof for $\widehat\epsilon(\xi^N,z^N)$. The proof for $\widehat\epsilon(\xi^N,\overline z^N)$ requires only minor changes. For reasons to be shown in the following, it will be convenient to define 
$\Delta(x,x^*):=\Vert x-x^*\Vert\vee\Vert x-x^*\Vert^\delta$ and, for any $s>0$, $\mathsf{R}(s):=(1+L\hat\alpha)\Delta(x,x^*)+\hat\alpha s$ and the ball $\mathbb{B}(s):=\mathbb{B}[x^*,\mathsf{R}(s)]$.

Example 14.29 of \cite{rockafellar:wets1998} and Assumption \ref{assumption:holder:continuity} imply that the map $\Xi\times X\ni(\omega,x)\mapsto\Vert\widehat\epsilon(\xi^N(\omega),x)\Vert$ is a \emph{normal integrand}, that is, 
$$
\omega\mapsto\epi\left\Vert\widehat\epsilon(\xi^N(\omega),\cdot)\right\Vert:=\{(x,y)\in X\times\re:\left\Vert\widehat\epsilon(\xi^N(\omega),x)\right\Vert\le y\}
$$
is a set-valued measurable function. This fact and Theorem 14.37 in \cite{rockafellar:wets1998} imply further that, for any measurable function $\epsilon:\Omega\rightarrow[0,\infty)$ and $R>0$,
\begin{eqnarray}
\omega\mapsto\sup_{x'\in \mathbb{B}(\epsilon(\omega))\cap X}\left\Vert\widehat\epsilon(\xi^N(\omega),x')\right\Vert \quad\mbox{ and }\quad
\omega\mapsto\sup_{x'\in \mathbb{B}[x^*,R]\cap X}\left\Vert\widehat\epsilon(\xi^N(\omega),x')\right\Vert\label{equation:measurability:issue}
\end{eqnarray}
are measurable functions.

We first prove item (ii) for the easier case when $X$ is compact. We set $R:=\diam(X)$ and note that $z^N\in\mathbb{B}[x^*,R]\cap X$. This and \eqref{equation:measurability:issue} imply that 
\begin{eqnarray*}
\Lpnorm{\Vert\widehat\epsilon(\xi^N,z^N)\Vert}&\le &\Lpnorm{\sup_{x'\in\mathbb{B}[x^*,R]\cap X}\Vert\widehat\epsilon(\xi^N,x')\Vert}\nonumber\\
&\le &\Lpnorm{\sup_{x'\in\mathbb{B}[x^*,R]\cap X}\Vert\widehat\epsilon(\xi^N,x')-\widehat\epsilon(\xi^N,x^*)\Vert}+\Lpnorm{\Vert\widehat\epsilon(\xi^N,x^*)\Vert}\nonumber\\
&\le &c\left[\frac{3^\delta\sqrt{d}L_2}{\sqrt{\delta}\left(\sqrt{2}^\delta-1\right)}+\sqrt{p}L_2+pL_p\right]\frac{\diam(X)^\delta}{\sqrt{N}}+\frac{C_{p}\Lpnorm{\Vert\epsilon(\xi,x^*)\Vert}}{\sqrt{N}},
\end{eqnarray*}
for some universal constant $c>0$, where we used Lemmas \ref{lemma:error:decay:empirical:process} and \ref{lemma:decay:empirical:error} with $q=p$ in the last inequality. The above inequality and definition \eqref{equation:oracle:error:variance} prove item (ii).

We now prove item (i) in the case $X$ may be unbounded. Given $\alpha\in[0,\hat\alpha]$, Lemma \ref{lemma:proj}(iv) implies that $x^*=\Pi[x^*-\alpha T(x^*)]$. Taking into account this fact, Lemma \ref{lemma:proj}(iii) and definitions of $z(\xi^N;\alpha,x)$, \eqref{equation:oracle:error} and \eqref{equation:empirical:mean:operator:&:error}, we get that, for any $\alpha\in[0,\hat\alpha]$,
\begin{eqnarray}
\left\|x^*-z(\xi^N;\alpha,x)\right\|&=&\,\left\Vert\Pi\left[x^*-\alpha T(x^*)\right]- \Pi\left[x-\alpha\left(T(x)+\widehat\epsilon(\xi^N,x)\right)\right]\right\Vert\nonumber\\
&\leq & \,\|x^*-x\|+\alpha\| T(x)-T(x^*)\| + \alpha\left\Vert\widehat\epsilon(\xi^N,x)\right\Vert\nonumber\\
&\leq & (1+L\hat\alpha)\left[\Vert x-x^*\Vert\vee\Vert x-x^*\Vert^\delta\right]+\hat\alpha\,\left\|\widehat\epsilon(\xi^N,x)\right\|,\label{lemma:error:decay:emp:eq1}
\end{eqnarray}
where, in last inequality, we used H\"older continuity of $T$ (Lemma \ref{lemma:holder:continuity:mean:std:dev}). 

In the sequel we define the quantities 
\begin{equation}\label{lemma:error:decay:emp:s*:eN}
s_*:=L_{2p}\Delta(x,x^*)\quad\mbox{ and }\quad\epsilon_N:=\left\Vert\widehat\epsilon(\xi^N,x)\right\Vert.
\end{equation}
Setting $\alpha:=\alpha_N$ in \eqref{lemma:error:decay:emp:eq1}, we have that\footnote{Note that from $\alpha_N\in[0,1]$ and convexity of $X$ and $\mathbb{B}(\epsilon_N)$, we also have that $\overline z^N\in \mathbb{B}(\epsilon_N)\cap X$.} $z^N\in \mathbb{B}(\epsilon_N)\cap X$. We now make the following decomposition
\begin{eqnarray}
\Lpnorm{\left\Vert\widehat\epsilon(\xi^N,z^N)\right\Vert}=I_1+I_2,\label{lemma:error:decay:emp:eq2}
\end{eqnarray}
using the definitions
\begin{eqnarray*}
I_1:=\Lpnorm{\left\Vert\widehat\epsilon(\xi^N,z^N)\right\Vert\unit_{\{\epsilon_N\le s_*\}}}\quad\mbox{ and }\quad
I_2:=\Lpnorm{\left\Vert\widehat\epsilon(\xi^N,z^N)\right\Vert\unit_{\{\epsilon_N>s_*\}}}.
\end{eqnarray*}

\textsf{\textbf{PART 1} (Upper bound on $I_1$):}  from the fact that $z^N\in\mathbb{B}(\epsilon_N)\cap X$ and \eqref{equation:measurability:issue}, we may bound $I_1$ by
\begin{eqnarray*}
I_1&=&\Lpnorm{\Vert\widehat\epsilon(\xi^N,z^N)\Vert\unit_{\{\epsilon_N\le s_*\}}}\nonumber\\
&\le &\Lpnorm{\sup_{x'\in \mathbb{B}(s_*)\cap X}\Vert\widehat\epsilon(\xi^N,x')\Vert}\nonumber\\
&\le &\Lpnorm{\sup_{x'\in \mathbb{B}(s_*)\cap X}\Vert\widehat\epsilon(\xi^N,x')-\widehat\epsilon(\xi^N,x^*)\Vert}+\Lpnorm{\Vert\widehat\epsilon(\xi^N,x^*)\Vert}\nonumber\\
&\le &c\left[\frac{3^\delta\sqrt{d}L_2}{\sqrt{\delta}\left(\sqrt{2}^\delta-1\right)}+\sqrt{p}L_2+pL_p\right]\frac{\mathsf{R}(s_*)^\delta}{\sqrt{N}}+\frac{C_{p}\Lpnorm{\Vert\epsilon(\xi,x^*)\Vert}}{\sqrt{N}},
\end{eqnarray*}
where we used Lemmas \ref{lemma:error:decay:empirical:process} and \ref{lemma:decay:empirical:error} with $q=p$ in the last inequality. Using the fact that $\mathsf{R}(s_*)=\left(1+L\hat\alpha+L_{2p}\hat\alpha\right)\Delta(x,x^*)$ and setting $c_\delta:=\frac{c3^\delta}{\sqrt{\delta}(\sqrt{2}^\delta-1)}$, we get from the above chain of inequalities that
\begin{equation}\label{lemma:error:decay:emp:i1:eq3}
I_1\le\left[\left(c_\delta\sqrt{d}+c\sqrt{p}\right)L_2+cpL_p\right]C_{\mathsf{L}\hat\alpha,p}^\delta\frac{\Delta(x,x^*)^\delta}{\sqrt{N}}+\frac{C_{p}\Lpnorm{\Vert\epsilon(\xi,x^*)\Vert}}{\sqrt{N}},
\end{equation}
with $C_{\mathsf{L}\hat\alpha,p}:=1+L\hat\alpha+L_{2p}\hat\alpha$.

\textsf{\textbf{PART 2} (Upper bound on $I_2$):} Defining $\widehat L_N:=N^{-1}\sum_{j=1}^N\mathsf{L}(\xi_j)$, we note that
\begin{eqnarray*}
\left\Vert\widehat\epsilon(\xi^N,z^N)\right\Vert &\le &\left\Vert\widehat\epsilon(\xi^N,z^N)-\widehat\epsilon(\xi^N,x^*)\right\Vert+\left\Vert\widehat\epsilon(\xi^N,x^*)\right\Vert\nonumber\\
&\le &\left\Vert\frac{1}{N}\sum_{j=1}^N\left[F(\xi_j,z^N)-F(\xi_j,x^*)\right]\right\Vert+
\left\Vert T(z^N)-T(x^*)\right\Vert+\left\Vert\widehat\epsilon(\xi^N,x^*)\right\Vert\nonumber\\
&\le &\left(\widehat L_N+L\right)\left\Vert z^N-x^*\right\Vert^\delta+\left\Vert\widehat\epsilon(\xi^N,x^*)\right\Vert\nonumber\\
&\le &\left(\widehat L_N+L\right)(1+L\hat\alpha)\Delta(x,x^*)+\hat\alpha\left(\widehat L_N+L\right)\epsilon_N+\epsilon_N^*,
\end{eqnarray*}
using Assumption \ref{assumption:holder:continuity} and Lemma \ref{lemma:holder:continuity:mean:std:dev} in the third inequality and 
\eqref{lemma:error:decay:emp:eq1} with $\alpha:=\alpha_N$, \eqref{lemma:error:decay:emp:s*:eN} and the definition $\epsilon_N^*:=\left\Vert\widehat\epsilon(\xi^N,x^*)\right\Vert$ in the last inequality. The inequality above and definition of $I_2$ imply that
\begin{eqnarray}
I_2&=&\Lpnorm{\left\Vert\widehat\epsilon(\xi^N,z^N)\right\Vert\unit_{\{\epsilon_N>s_*\}}}\nonumber\\
&\le &(1+L\hat\alpha)\Delta(x,x^*)\Lpnorm{\left(\widehat L_N+L\right)\unit_{\{\epsilon_N>s_*\}}}%
+\hat\alpha\Lpnorm{\left(\widehat L_N+L\right)\epsilon_N}
+\Lpnorm{\epsilon_N^*}\nonumber\\
&\le & (1+L\hat\alpha)\Delta(x,x^*)\Ldpnorm{\widehat L_N+L}\Ldpnorm{\unit_{\{\epsilon_N>s_*\}}}
+\hat\alpha\Ldpnorm{\widehat L_N+L}\Ldpnorm{\epsilon_N}
+\Lpnorm{\epsilon_N^*},\label{lemma:error:decay:emp:eq4}
\end{eqnarray}
where we used H\"older's inequality. 

With respect to the last term in the rightmost expression of \eqref{lemma:error:decay:emp:eq4}, 
we have, in view of Lemma \ref{lemma:decay:empirical:error} with $q=p$,
\begin{eqnarray}
\Lpnorm{\epsilon_N^*}=\Lpnorm{\Vert\widehat\epsilon(\xi^N,x^*)\Vert}\le \frac{C_{p}\Lpnorm{\Vert\epsilon(\xi,x^*)\Vert}}{\sqrt{N}}.\label{lemma:error:decay:emp:eq5}
\end{eqnarray}

Concerning the second term in the rightmost expression of \eqref{lemma:error:decay:emp:eq4}, 
Lemma \ref{lemma:decay:empirical:error} with $q=2p$ implies that
\begin{equation}
\Ldpnorm{\epsilon_N}=\Ldpnorm{\left\Vert\widehat\epsilon(\xi^N,x)\right\Vert}\le C_{2p}\frac{\Ldpnorm{\Vert\epsilon(\xi,x^*)\Vert}+L_{2p}\Vert x-x^*\Vert^\delta}{\sqrt{N}}. 
\label{lemma:error:decay:emp:eq6}
\end{equation}

From Markov's inequality and \eqref{lemma:error:decay:emp:eq6} we obtain 
\begin{eqnarray}
\Ldpnorm{\unit_{\{\epsilon_N>s_*\}}}&=&\sqrt[2p]{\esp\left[\unit_{\left\{\epsilon_N>s_*\right\}}\right]}=\sqrt[2p]{\prob\left(\left\Vert\widehat\epsilon(\xi^N,x)\right\Vert>s_*\right)}\nonumber\\
&\le &\sqrt[2p]{\frac{\esp\left[\left\Vert\widehat\epsilon(\xi^N,x)\right\Vert^{2p}\right]}{s_*^{2p}}}\nonumber\\
&=&\frac{\Ldpnorm{\left\Vert\widehat\epsilon(\xi^N,x)\right\Vert}}{s_*}\nonumber\\
&\le &C_{2p}\frac{\Ldpnorm{\Vert\epsilon(\xi,x^*)\Vert}+L_{2p}\Vert x-x^*\Vert^\delta}{s_*\sqrt{N}}.\label{lemma:error:decay:emp:eq7}
\end{eqnarray}

The convexity of $t\mapsto t^{2p}$ and the fact that $\{\xi_j\}_{j\in[N]}$ is an i.i.d. sample of $\probn$ imply that $\Ldpnorm{\widehat L_N+L}\le\Ldpnorm{\mathsf{L}(\xi)}+L=L_{2p}$. Using this fact and putting together relations \eqref{lemma:error:decay:emp:eq4}-\eqref{lemma:error:decay:emp:eq7} we get
\begin{eqnarray}
I_2&\le & (1+L\hat\alpha)\frac{\Delta(x,x^*)L_{2p}}{s_*}
C_{2p}\frac{\Ldpnorm{\Vert\epsilon(\xi,x^*)\Vert}+L_{2p}\Vert x-x^*\Vert^\delta}{\sqrt{N}}
\nonumber\\
&&+L_{2p}\hat\alpha C_{2p}\frac{\Ldpnorm{\Vert\epsilon(\xi,x^*)\Vert}+L_{2p}\Vert x-x^*\Vert^\delta}{\sqrt{N}}
+\frac{C_{p}\Lpnorm{\Vert\epsilon(\xi,x^*)\Vert}}{\sqrt{N}}\nonumber\\
&=&C_{2p}\left(1+L\hat\alpha+L_{2p}\hat\alpha\right)\frac{\Ldpnorm{\Vert\epsilon(\xi,x^*)\Vert}}{\sqrt{N}}+\frac{C_{p}\Lpnorm{\Vert\epsilon(\xi,x^*)\Vert}}{\sqrt{N}}\nonumber\\
&&+C_{2p}\left(1+L\hat\alpha+L_{2p}\hat\alpha\right)\frac{L_{2p}\Vert x-x^*\Vert^\delta}{\sqrt{N}},\label{lemma:error:decay:emp:eq8}
\end{eqnarray}
where we used the fact that\footnote{Note that $[\Delta(x,x^*)\Vert x-x^*\Vert^\delta]^2\lesssim\Vert x-x^*\Vert^{4\delta}$ with $4\delta>2$ in the Lipschitz continuous case. The geometry of projection methods implies the derivation of a recursion in terms of $\{\Vert x^k-x^*\Vert^2\}$. It is then crucial for the convergence analysis that follows that we can choose a $s_*$ that balances the bounds $\mathsf{R}(s_*)^\delta\lesssim\Vert x-x^*\Vert^{\beta_1}$ in $I_1$ and $\frac{\Delta(x,x^*)}{s_*}\Vert x-x^*\Vert^\delta\lesssim\Vert x-x^*\Vert^{\beta_2}$ in $I_2$ with $\beta_1,\beta_2\in(0,1]$.} $s_*=L_{2p}\Delta(x,x^*)$. 

Relations \eqref{lemma:error:decay:emp:eq2}-\eqref{lemma:error:decay:emp:i1:eq3} and \eqref{lemma:error:decay:emp:eq8}, definition \eqref{equation:oracle:error:variance} and the facts that $\Delta(x,x^*)^\delta\le\delta_1\vee\Vert x-x^*\Vert^\delta$ and $\Lpnorm{\Vert\epsilon(\xi,x^*)\Vert}\le\Ldpnorm{\Vert\epsilon(\xi,x^*)\Vert}$ prove item (i).
\end{proof}

\begin{remark}[Constants]\label{rem:constants:thm:correlated:error}
In Theorem \ref{thm:variance:error:with:line:search}, the constants satisfy
\begin{eqnarray*}
&&\mathsf{c}_1:=2C_p+C_{2p}C_{\mathsf{L}\hat\alpha,p},\quad\quad \mathsf{c}_3\lesssim pC_{\mathsf{L}\hat\alpha,p}^\delta,\quad\quad\mathsf{c}_4:= C_{2p} C_{\mathsf{L}\hat\alpha,p},\\
&&\mathsf{c}_2\lesssim\left[\frac{3^\delta\sqrt{d}}{\sqrt{\delta}\left(\sqrt{2}^\delta-1\right)}+\sqrt{p}\right]C_{\mathsf{L}\hat\alpha,p}^\delta,\quad\quad
\mathsf{d}_2\lesssim \left[\frac{3^\delta\sqrt{d}}{\sqrt{\delta}\left(\sqrt{2}^\delta-1\right)}+\sqrt{p}\right],
\end{eqnarray*}
where $C_{\mathsf{L}\hat\alpha,p}:=1+2L\hat\alpha+\Ldpnorm{\mathsf{L}(\xi)}\hat\alpha$ and $C_p$ and $C_{2p}$ are defined in Lemma \ref{lemma:decay:empirical:error}.
\end{remark}

\section{Analysis of \textsf{Algorithm \ref{algorithm:DSSA:extragradient}} for Lipschitz continuous operators}\label{section:algorithm:extragradient:DSSA}

We state  next additional assumptions needed for the convergence analysis of our algorithms. In this section we always assume that in Assumption \ref{assumption:holder:continuity} we have $\delta=1$. For brevity, we will not mention it any further. 
\begin{assumption}[Consistency]\label{assump:existence}
The solution set $X^*$ of VI$(T,X)$ is non-empty.
\end{assumption}

\begin{assumption}[Pseudo-monotonicity]\label{assump:monotonicity}
We assume that $T:X\rightarrow\re^d$ as defined in \eqref{equation:expected:valued:objective} is pseudo-monotone\footnote{Pseudo-monotonicity generalizes monotonicity: $\langle T(z),z-x\rangle\ge\langle T(x),z-x\rangle$ for all $x,z\in X$. Recall that the gradient of a smooth convex function is monotone and the gradient of a quotient of a positive smooth convex function with a positive smooth concave function is pseudo-monotone.}: for all $z,x\in X$,
$
\langle T(x),z-x\rangle\ge 0\Longrightarrow\langle T(z),z-x\rangle\ge 0.
$
\end{assumption}

\begin{assumption}[I.I.D. sampling]\label{assump:iid:sampling}
In \textsf{Algorithm \ref{algorithm:DSSA:extragradient}}, the sequences $\{\xi^k_j:k\in\mathbb{N}_0,j\in[N_k]\}$ and $\{\eta^k_j:k\in\mathbb{N}_0,j\in[N_k]\}$ are i.i.d. samples drawn from $\probn$ independent of each other. Moreover, $\sum_{k=0}^\infty N_k^{-1}<\infty$.
\end{assumption}

Concerning \textsf{Algorithm \ref{algorithm:DSSA:extragradient}}, we shall study the stochastic process $\{x^k\}$ with respect to the filtrations
$$
\alg_k=\sigma(x^0,\xi^0,\ldots,\xi^{k-1},
\eta^0,\ldots,\eta^{k-1}),\quad
\widehat\alg_k=\sigma(x^0,\xi^0,\ldots,\xi^k,
\eta^0,\ldots,\eta^{k-1}).
$$

Recalling \eqref{equation:oracle:error}, \eqref{equation:empirical:mean:operator:&:error} and \textsf{Algorithm \ref{algorithm:DSSA:extragradient}}, we will define the following oracle errors: 
\begin{eqnarray}
\epsilon^k_1:=\widehat\epsilon(\xi^k,x^k),\quad\quad\quad\epsilon^k_2:=\widehat\epsilon(\eta^k,z^k),\quad\quad\quad\epsilon^k_3:=\widehat\epsilon(\xi^k,z^k). \label{equation:oracle:errors:DSSA:extragradient}
\end{eqnarray}
Their relations to $\alg_k$ and $\widehat\alg_k$ will be essential in the convergence analysis. By definition of \textsf{Algorithm \ref{algorithm:DSSA:extragradient}} and Assumption \ref{assump:iid:sampling}, we have that $z^k\in\widehat{\alg}_k$ and $\eta^k\perp\perp\widehat{\alg}_k$. These facts imply that, with respect to the \emph{sampling} time scale, the process $[N_k]\ni t\mapsto N_k^{-1}\sum_{j=1}^{t}\epsilon(\eta^k_j,z^k)$ defines a martingale difference adapted to the filtration $\sigma(\widehat{\alg}_k,\eta^k_1,\ldots,\eta^k_t)$ with final element $\epsilon^k_2$. With respect to the \emph{iteration} time scale, the same facts imply that the process $k\mapsto\epsilon^k_2$ defines a martingale difference adapted to the filtration $\widehat{\alg}_{k+1}$. Similar observations hold for the processes $[N_k]\ni t\mapsto N_k^{-1}\sum_{j=1}^{t}\epsilon(\xi^k_j,x^k)$ and $k\mapsto\epsilon^k_1$ using the facts that $x^k\in\alg_k$ and $\xi^k\perp\perp\alg_k$. The line search scheme \eqref{algo:armijo:rule} introduces the error $\epsilon^k_3$. It does not have the previous martingale-like properties above due to the coupling between $\xi^k$ and $z^k$. This will be resolved by applying Theorem \ref{thm:variance:error:with:line:search} with $\xi^N:=\xi^k$, $x:=x^k$ and $\alpha_N:=\alpha_k$, noting that $z^k=z(\xi^k;\alpha_k,x^k)$ and $x^k\in\alg_k$. It is also important to note that the stepsize $\alpha_k$ is a random variable satisfying $\alpha_k\notin\alg_k$ and $\alpha_k\in\widehat{\alg}_k$.

\subsection{Convergence analysis}
We first show that the line search \eqref{algo:armijo:rule} in \textsf{Algorithm \ref{algorithm:DSSA:extragradient}} is well defined.
\begin{lemma}[Good definition of the line search]
Consider Assumption \ref{assumption:holder:continuity}. Then the line search \eqref{algo:armijo:rule} in \textsf{\emph{Algorithm \ref{algorithm:DSSA:extragradient}}} terminates after a finite number $\ell_k$ of iterations.
\end{lemma}
\begin{proof}
Set $\gamma_\ell:=\theta^{-\ell}\hat\alpha$ and $H_k:=\widehat F(\xi^k,\cdot)$. Assuming by contradiction that the line search \eqref{algo:armijo:rule} does not terminate after a finite number of iterations, for every $\ell\in\mathbb{N}_0$,
\begin{equation*}
\left\Vert\widehat F\left(\xi^k,z^k(\gamma_\ell)\right)-\widehat F\left(\xi^k,x^k\right)\right\Vert>\lambda \frac{r_{\gamma_\ell}(H_k;x^k)}{\gamma_\ell}\ge\lambda\cdot r(H_k;x^k),
\end{equation*} 
using definition of $r_{\alpha}(H_k;\cdot)$ in Section \ref{section:preliminaries}, the fact that $\gamma_\ell\in(0,1]$ and Lemma \ref{lemma:residual:decrease} in the last inequality. The contradiction follows by letting $\ell\rightarrow\infty$ in the above inequality and invoking the continuity of  $\widehat F(\xi^k,\cdot)$, resulting from  Assumption \ref{assumption:holder:continuity}, the fact that $\lim_{\ell\rightarrow\infty}z^k(\gamma_\ell)=x^k$, which follows from the continuity of $\Pi$, and the fact that $r(H_k;x^k)>0$, which follows from the definition of \textsf{Algorithm \ref{algorithm:DSSA:extragradient}}.
\end{proof}

The next lemma shows that the DS-SA line search scheme \eqref{algo:armijo:rule} either chooses the initial stepsize $\hat\alpha$ or it is an \emph{unbiased} stochastic oracle for a \emph{lower bound} of the Lipschitz constant $L=\esp[\mathsf{L}(\xi)]$ (using the \emph{same samples} generated by the operator's stochastic oracle). Precisely, if $\hat\alpha$ is not chosen, then $\mathcal{O}(\alpha_k^{-1})$ (with an explicit constant) is a.s. a lower bound for 
$
\widehat L_k:=\frac{1}{N_k}\sum_{j=1}^{N_k}\mathsf{L}(\xi^k_j).
$
\begin{lemma}[Unbiased lower estimation of the Lipschitz constant]\label{lemma:step:lower}
Consider Assumptions \ref{assumption:holder:continuity} and \ref{assump:iid:sampling} and define $\widehat L_k:=\frac{1}{N_k}\sum_{j=1}^{N_k}\mathsf{L}(\xi^k_j)$. Then, if the \textsf{\emph{Algorithm \ref{algorithm:DSSA:extragradient}}} does not stop at iteration $k+1$, a.s.
$
\alpha_k\ge\left(\frac{\lambda\theta}{\widehat L_k}\right)\wedge\hat\alpha.
$
Moreover, 
$
\Lnorm{\alpha_k\big|\alg_k}\cdot\Lnorm{\mathsf{L}(\xi)}\ge(\lambda\theta)\wedge\hat\alpha.
$
\end{lemma}
\begin{proof}
If $\hat\alpha$ satisfies \eqref{algo:armijo:rule}, then $\alpha_k=\hat\alpha$. Otherwise, we have
\begin{equation}\label{lemma:step:lower:eq1}
\theta^{-1}\alpha_k\left\Vert\widehat F\left(\xi^k,z^k(\theta^{-1}\alpha_k)\right)-
\widehat F(\xi^k,x^k)\right\Vert>\lambda\left\Vert z^k\left(\theta^{-1}\alpha_k\right)-x^k\right\Vert.
\end{equation}
Assumption \ref{assumption:holder:continuity} and definition of $\widehat F(\xi^k,\cdot)$ in \eqref{equation:empirical:mean:operator:&:error} imply that
\begin{equation}\label{lemma:step:lower:eq2}
\left\Vert\widehat F\left(\xi^k,z^k(\theta^{-1}\alpha_k)\right)-\widehat F(\xi^k,x^k)\right\Vert\le
\widehat L_k\left\Vert z^k\left(\theta^{-1}\alpha_k\right)-x^k\right\Vert.
\end{equation}
The fact that
$z^k\left(\theta^{-1}\alpha_k\right)\neq x^k$ (since the method did not stopped at iteration $k+1$) and \eqref{lemma:step:lower:eq1}-\eqref{lemma:step:lower:eq2} imply that $\alpha_k\ge\frac{\lambda\theta}{\widehat L_k}$. We have thus proved the first statement.

Since a.s. $\mathsf{L}(\xi)\ge1$, we also have a.s. $\widehat{L}_k\alpha_k\ge(\lambda\theta)\wedge\hat\alpha$. The second statement follows from this fact and
\begin{eqnarray*}
(\lambda\theta)\wedge\hat\alpha &\le &\esp\left[\alpha_k\widehat L_k\Big|\alg_k\right]\\
&\le & \Lnorm{\alpha_k\big|\alg_k}\cdot\Lnorm{\widehat L_k\Big|\alg_k}\\
&=&\Lnorm{\alpha_k\big|\alg_k}\sqrt{\esp\left[\left(\frac{1}{N_k}\sum_{j=1}^{N_k}\mathsf{L}(\xi^k_j)\right)^2\Bigg|\alg_k\right]}\\
&\le &\Lnorm{\alpha_k\big|\alg_k}\sqrt{\frac{1}{N_k}\sum_{j=1}^{N_k}\esp\left[\mathsf{L}(\xi^k_j)^2\Big|\alg_k\right]}=\Lnorm{\alpha_k\big|\alg_k}\cdot\Lnorm{\mathsf{L}(\xi)},
\end{eqnarray*}
using H\"older's inequality 
in the second inequality, 
the convexity of $t\mapsto t^2$ 
in the third inequality 
and the fact
that $\xi^k$ is an i.i.d sample of $\probn$ with $\xi^k\perp\perp\alg_k$ in the last equality. 
\end{proof}

Recall \eqref{equation:oracle:errors:DSSA:extragradient}. We define, for $k\in\mathbb{N}_0$ and for $x^*\in X^*$,
\begin{eqnarray}
\Delta A_{k}&:=&(1-6\lambda^2)\hat\alpha^2\Vert\epsilon^k_1\Vert^2+
6\hat\alpha^2\Vert\epsilon^k_2\Vert^2+6\hat\alpha^2\Vert\epsilon^k_3\Vert^2,\label{def:A:armijo}\\
\Delta M_{k}(x^*)&:=&2\alpha_k\langle x^*-z^k,\epsilon^k_2\rangle.\label{def:M:armijo}
\end{eqnarray}
Lemma \ref{lemma:recursion:armijo} and Proposition \ref{prop:A:armijo} stated in the following are proved in the Appendix.

\begin{lemma}[A recursive error bound for \textsf{Algorithm \ref{algorithm:DSSA:extragradient}}]\label{lemma:recursion:armijo}
Consider Assumptions \ref{assumption:holder:continuity} and \ref{assump:existence}-\ref{assump:monotonicity}. If \textsf{\emph{Algorithm \ref{algorithm:DSSA:extragradient}}} does not stop at iteration $k+1$ then, for all $x^*\in X^*$,
\begin{eqnarray*}
\Vert x^{k+1}-x^*\Vert^2\le \Vert x^k-x^*\Vert^2-\frac{(1-6\lambda^2)\alpha_k^2}{2} r(x^k)^2+\Delta M_{k}(x^*)+\Delta A_k.
\end{eqnarray*}
\end{lemma}

\begin{proposition}[Bounds on the oracle's errors]\label{prop:A:armijo}
Consider Assumptions \ref{assumption:holder:continuity}, \ref{assump:existence} and \ref{assump:iid:sampling}. Recall definitions in \eqref{equation:oracle:error:variance}, \eqref{def:A:armijo}, Theorem \ref{thm:variance:error:with:line:search} and Lemma \ref{lemma:decay:empirical:error}. Then there exist positive constants $\mathsf{C}_{p}$ and $\mathsf{\overline C}_{p}$ (depending only on $d$, $p$, $\mathsf{L}({\xi})\hat\alpha$ and $\{N_k\}$) such that, if the method does not stop at iteration $k+1$ we have, for all $x^*\in X^*$, 	
\begin{eqnarray*}
\Lpddnorm{\Delta A_k|\alg_k}&\le &\frac{\mathsf{C}_{p}\left[\hat\alpha\sigma_{\mathsf{a}p}(x^*)\right]^2+\mathsf{\overline{C}}_{p}\left(\hat\alpha\widetilde L_p\right)^2D_k^2}{N_k}.
\end{eqnarray*}
In above, for $X$ compact, we have $\widetilde L_p:=(C_pL_p)\vee L_p^*$ and $D_k:\equiv\diam(X)$. For a general $X$, we have $\widetilde L_p:=\overline{L}_{2p}$ and $D_k:=\Vert x^k-x^*\Vert$.
\end{proposition}
See Remark \ref{rem:constants:A:armijo} in the Appendix for details on the constants above. In the following convergence analysis, we set $p=2$ (see Remark \ref{rem:lp:boundedness} for the interest in higher moments).

\begin{proposition}[Stochastic quasi-Fej\'er property]\label{prop:fejer:armijo}
Consider Assumptions \ref{assumption:holder:continuity} and \ref{assump:existence}-\ref{assump:iid:sampling} and definitions in Proposition \ref{prop:A:armijo} with $p=2$. Set 
$\nu:=\frac{(1-6\lambda^2)\left[(\lambda\theta)\wedge\hat\alpha\right]^2}{2\Lnorm{\mathsf{L}(\xi)}^2}$. If \textsf{\emph{Algorithm \ref{algorithm:DSSA:extragradient}}} does not stop at iteration $k+1$ then, for all $x^*\in X^*$,
\begin{eqnarray*}
\esp\left[\|x^{k+1}-x^*\|^2|\alg_k\right]&\le &\|x^{k}-x^*\|^2 -\nu\cdot r(x^k)^2
+\frac{\mathsf{C}_{2}\left[\hat\alpha\sigma_{2\mathsf{a}}(x^*)\right]^2}{N_k}+\frac{\mathsf{\overline{C}}_{2}\left(\hat\alpha\widetilde L_2\right)^2}{N_k}D_k^2.
\end{eqnarray*}
\end{proposition}
\begin{proof}
We first show that $\{\Delta M_k(x^*),\alg_k\}$ defines a martingale difference even if $\alpha_k\notin\alg_k$. Indeed, the facts that $z^k\in\widehat{\alg}_k$ and $\eta^k\perp\perp\widehat{\alg}_k$ imply that $\esp[\epsilon^k_2|\widehat{\alg}_k]=0$, where $\epsilon^k_2$ is defined in \eqref{equation:oracle:errors:DSSA:extragradient}. This fact, $z^k\in\widehat\alg_k$ and $\alpha_k\in\widehat{\alg}_k$ imply that $\esp[\Delta M_k(x^*)|\widehat{\alg}_k]=0$. Using this and the fact that $\esp[\esp[\cdot|\widehat{\alg}_k]|\alg_k]=\esp[\cdot|\alg_k]$, we finally conclude that $\esp[\Delta M_k(x^*)|\alg_k]=0$ as claimed. The recursion in the statement follows immediately from this fact, relation $\esp\left[\alpha_k^2\big|\alg_k\right]\ge\frac{[(\lambda\theta)\wedge\hat\alpha]^2}{\Lnorm{\mathsf{L}(\xi)}^2}$ in Lemma \ref{lemma:step:lower}, Lemma \ref{lemma:recursion:armijo} and Proposition \ref{prop:A:armijo} with $p=2$, after we take $\esp[\cdot|\alg_k]$ in Lemma \ref{lemma:recursion:armijo} and use the fact that $x^k\in\alg_k$.
\end{proof}

We now proceed to establish the asymptotic convergence of \textsf{Algorithm \ref{algorithm:DSSA:extragradient}}.
\begin{theorem}[Asymptotic convergence]\label{thm:convergence:armijo} 
Under Assumptions \ref{assumption:holder:continuity} and \ref{assump:existence}-\ref{assump:iid:sampling}, either \emph{\textsf{Algorithm \ref{algorithm:DSSA:extragradient}}} stops at iteration $k+1$, in which case $x^k$ is a solution of \emph{VI}$(T,X)$, or it generates an infinite sequence $\{x^k\}$ such that a.s. it is bounded, 
$
\lim_{k\rightarrow\infty}\dist(x^k,X^*)=0,
$
and
$
r(x^k)
$ 
converges to $0$ almost surely and in $\mathcal{L}^2$. In particular, a.s. every cluster point of $\{x^k\}$ belongs to $X^*$.
\end{theorem}
\begin{proof}
If \textsf{Algorithm \ref{algorithm:DSSA:extragradient}} stops at iteration $k$, then 
$
x^k=\Pi[x^k-\hat\alpha\widehat F(\xi^k,x^k)].
$
From this fact and Lemma \ref{lemma:proj}(iv) we get, for all $x\in X$,
\begin{equation}\label{thm:convergence:armijo:eq0}
\langle \widehat F(\xi^k,x^k),x-x^k\rangle\ge0.
\end{equation}
From the facts that $x^k\in\alg_k$ and $\xi^k\perp\perp\alg_k$, Assumption \ref{assump:iid:sampling} and \eqref{equation:expected:valued:objective}, we have that 
$\esp\left[\widehat F(\xi^k,x^k)|\alg_k\right]=T(x^k)$. Using this result and the fact that $x^k\in\alg_k$, we take $\esp[\cdot|\alg_k]$ in \eqref{thm:convergence:armijo:eq0} and obtain, for all $x\in X$, $\langle T(x^k),x-x^k\rangle\ge0$. Hence $x^k\in X^*$.

Suppose now that \textsf{Algorithm \ref{algorithm:DSSA:extragradient}} generates an infinite sequence. 
Take some $x^*\in X^*$. Taking into account $\sum_k N_k^{-1}<\infty$, Proposition \ref{prop:fejer:armijo} for a general $X$ ($\mathsf{a}:=2$) and the fact that $x^k\in\alg_k$, we apply Theorem \ref{thm:rob} with $y_k:=\Vert x^k-x^*\Vert^2$, $a_k:=\frac{\mathsf{\overline{C}}_{2}(\hat\alpha\overline{L}_{4})^2}{{N}_k}$, $b_k:=\frac{\mathsf{C}_{2}[\hat\alpha\sigma_4(x^*)]^2}{{N}_k}$ and $u_k:=\frac{(1-6\lambda^2)[(\lambda\theta)\wedge\hat\alpha]^2}{2\Lnorm{\mathsf{L}(\xi)}^2}r(x^k)^2$, in order to conclude that a.s. $\{\Vert x^k-x^*\Vert^2\}$ converges and $\sum_{k=0}^\infty r(x^k)^2<\infty$. In particular, a.s. $\{x^k\}$ is bounded and
\begin{equation}\label{thm:convergence:armijo:eq4}
0 = \lim_{k\rightarrow\infty}r(x^k)^2\\
=\lim_{k\rightarrow\infty}\left\Vert x^k-\Pi\left[x^k-T(x^k)\right]\right\Vert^2.
\end{equation}
The fact that $\lim_{k\rightarrow\infty}\esp[r(x^k)^2]=0$ is proved in a similar way, taking expectation in the recursion of Proposition \ref{prop:fejer:armijo}.

Relation \eqref{thm:convergence:armijo:eq4}
and the continuity of $T$ (Lemma \ref{lemma:holder:continuity:mean:std:dev}) and $\Pi$ (Lemma \ref{lemma:proj}(iii)) imply that a.s. every  cluster point $\bar x$ of $\{x^k\}$ satisfies 
$
0=\bar x-\Pi\left[\bar x-T(\bar x)\right].
$
From Lemma \ref{lemma:proj}(iv), we conclude that $\bar x\in X^*$. A.s. the boundedness 
of $\{x^k\}$ and the fact that every cluster point of $\{x^k\}$ 
belongs to $X^*$ imply that $\lim_{k\rightarrow\infty}\dist(x^k,X^*)=0$. 
\end{proof}

\subsection{Convergence rate and oracle complexity}
As mentioned in the Introduction, we allow $X$ to be unbounded and the SO may not have an uniformly bounded variance over $X$. In this setting, it is not possible to infer a priori the boundedness of the sequence $\left\{\Lnorm{\Vert x^k\Vert}\right\}$ (i.e., $\mathcal{L}^2$-boundedness of the iterates). In this section, we obtain such $\mathcal{L}^2$-boundedness when using DS-SA schemes. This will be essential to obtain complexity estimates.

\begin{proposition}[$\mathcal{L}^2$-boundedness of the iterates: unbounded case]\label{prop:l2:boundedness}
Let Assumptions \ref{assumption:holder:continuity} and \ref{assump:existence}-\ref{assump:iid:sampling} hold and recall definitions in \textsf{\emph{Algorithm \ref{algorithm:DSSA:extragradient}}}, \eqref{equation:oracle:error}-\eqref{equation:oracle:error:variance}, Theorem \ref{thm:variance:error:with:line:search} and Proposition \ref{prop:A:armijo} with $p=2$. Let $x^*\in X^*$ and choose $k_0:=k_0(\overline{\mathsf{C}}_2,\hat\alpha\overline{L}_4)\in\mathbb{N} $ and $\phi\in(0,1)$ such that 
\begin{equation}\label{prop:l2:boundedness:k0}
\sum_{i\ge k_0}^\infty\frac{1}{N_i}\le\frac{\phi}{\overline{\mathsf{C}}_2\left(\hat\alpha\overline{L}_4\right)^2}.
\end{equation}
Then
$
\sup_{k\ge k_0}\Lnorm{\Vert x^k-x^*\Vert}^2<\frac{\Lnorm{\Vert x^{k_0}-x^*\Vert}^2+\frac{\phi\mathsf{C}_2\sigma_4(x^*)^2}{\overline{\mathsf{C}}_2\overline{L}_4^2}}{1-\phi}.
$
\end{proposition}
\begin{proof}
In the following, we set $d_i:=\Vert x^i-x^*\Vert^2$ for $i\in\mathbb{N}_0$. Let $k>k_0$ in $\mathbb{N}_0$ with $k_0$ as stated in \eqref{prop:l2:boundedness:k0}. Note that such $k_0$ always exists since $\sum_kN_k^{-1}<\infty$ by Assumption \ref{assump:iid:sampling}. Consider the recursion of Proposition \ref{prop:fejer:armijo} for the case $X$ is unbounded ($\mathsf{a}:=2$). We take the expectation, use $\esp[\esp[\cdot|\alg_i]]=\esp[\cdot]$ and drop the negative term in the right hand side. We then sum recursively the obtained inequality from $i:=k_0$ to $i:=k-1$, obtaining
\begin{equation}\label{prop:l2:boundedness:eq1}
\Lnorm{d_k}^2\le\Lnorm{d_{k_0}}^2+\mathsf{\overline{C}}_2\left(\hat\alpha\overline{L}_4\right)^2\sum_{i=k_0}^{k-1}\frac{\Lnorm{d_i}^2}{N_i}+\mathsf{C}_2\left[\hat\alpha\sigma_4(x^*)\right]^2\sum_{i=k_0}^{k-1}\frac{1}{N_i}.
\end{equation} 

For any $a>0$, we define the stopping time 
$
\tau_a:=\{k\ge k_0:\Lnorm{d_k}>a\}.
$
From \eqref{prop:l2:boundedness:k0}-\eqref{prop:l2:boundedness:eq1} and definition of $\tau_a$, we have that, for any $a>0$ such that $\tau_a<\infty$,
\begin{eqnarray*}
a^2<\Lnorm{d_{\tau_a}}^2&\le &\Lnorm{d_{k_0}}^2+\mathsf{\overline{C}}_2\left(\hat\alpha\overline{L}_4\right)^2\sum_{i=k_0}^{\tau_a-1}\frac{\Lnorm{d_i}^2}{N_i}+\mathsf{C}_2\left[\hat\alpha\sigma_4(x^*)\right]^2\sum_{i=k_0}^{\tau_a-1}\frac{1}{N_i}\\
&<&\Lnorm{d_{k_0}}^2+\phi a^2+\frac{\phi\mathsf{C}_2\sigma_4(x^*)^2}{\overline{\mathsf{C}}_2\overline{L}_4^2},
\end{eqnarray*}
and hence,
$%\label{prop:l2:boundedness:eq2}
a^2<\frac{\Lnorm{d_{k_0}}^2+\frac{\phi\mathsf{C}_2\sigma_4(x^*)^2}{\overline{\mathsf{C}}_2\overline{L}_4^2}}{1-\phi}=:B,	 
$
where we used that $\phi\in(0,1)$. By definition of $\tau_a$ for any $a>0$, the argument above implies that any threshold $a^2$ which the sequence $\{\Lnorm{d_k}^2\}_{k\ge k_0}$ eventually exceeds is bounded above by $B$. Hence $\{\Lnorm{d_k}^2\}_{k\ge k_0}$ is bounded and it satisfies the statement of the proposition.  
\end{proof}

We now obtain a rate of convergence. 
\begin{theorem}[Rate of convergence]\label{thm:rate:convergence}
Consider Assumptions \ref{assumption:holder:continuity} and \ref{assump:existence}-\ref{assump:iid:sampling} and recall definitions in \textsf{\emph{Algorithm \ref{algorithm:DSSA:extragradient}}}, \eqref{equation:oracle:error}-\eqref{equation:oracle:error:variance} and Proposition \ref{prop:A:armijo} with $p=2$. Set
\begin{equation}\label{thm:rate:convergence:Nk}
N_k:=N\left\lceil(k+\mu)(\ln(k+\mu))^{1+b}\right\rceil,
\end{equation}
for any $N\in\mathbb{N}$, $b>0$ and $\mu>2$. Then Theorem \ref{thm:convergence:armijo} holds and the sequence $\{x^k\}$ generated by \textsf{\emph{Algorithm \ref{algorithm:DSSA:extragradient}}} is bounded in $\mathcal{L}^2$. Moreover, for any $x^*\in X^*$, if $\mathsf{J}>0$ is such that
$
\sup_{k\ge0}{\Lnorm{\Vert x^k-x^*\Vert}^2}\le\mathsf{J},
$
the following bound holds for all $k\in\mathbb{N}_0$:
\begin{eqnarray*}
\min_{i=0,\ldots,k}\esp\left[r(x^i)^2\right]
\le\frac{\left\{\frac{2\Lnorm{\mathsf{L}(\xi)}^2}{(1-6\lambda^2)[(\lambda\theta)\wedge\hat\alpha]^2}\right\}}{k+1}\left\{\Vert x^0-x^*\Vert^2+\frac{\mathsf{C}_2[\hat\alpha\sigma_{2\mathsf{a}}(x^*)]^2+\mathsf{\overline{C}}_2\left(\hat\alpha\widetilde L_2\right)^2\mathsf{J}}{Nb[\ln(\mu-1)]^b}\right\}. 
\end{eqnarray*}
\end{theorem}
\begin{proof}
Clearly, $\{N_k\}$ satisfies Assumption \ref{assump:iid:sampling} and, hence, Theorem \ref{thm:convergence:armijo} and Proposition \ref{prop:l2:boundedness} hold. In particular, $\{x^k\}$ is bounded in $\mathcal{L}^2$. Let $x^*\in X^*$ and $\mathsf{J}$ as stated in the theorem. Hence, $\sup_k\esp[D_k^2]\le\mathsf{J}$. In the recursion of Proposition \ref{prop:fejer:armijo}, we take the expectation, use $\esp[\esp[\cdot|\alg_i]]=\esp[\cdot]$ and sum recursively the obtained inequality from $i:=0$ to $i:=k$. We then obtain
$$
\frac{(1-6\lambda^2)[(\lambda\theta)\wedge\hat\alpha]^2}{2\Lnorm{\mathsf{L}(\xi)}^2}\sum_{i=0}^k\esp\left[r(x^i)^2\right]
\le\Vert x^0-x^*\Vert^2+\left\{\mathsf{C}_2[\hat\alpha\sigma_{2\mathsf{a}}(x^*)]^2+\mathsf{\overline{C}}_2\left(\hat\alpha\widetilde L_2\right)^2\mathsf{J}\right\}\mathsf{S}_k,
$$
where $\mathsf{S}_k:=\sum_{i=0}^kN_i^{-1}$. The proof of the statement follows from the above inequality, the bound
\begin{eqnarray*}
\mathsf{S}_k\le\sum_{i=0}^\infty\frac{1}{N_i}\le\int_{-1}^\infty\frac{\dist t}{N(t+\mu)[\ln(t+\mu)]^{1+b}}=\frac{1}{Nb[\ln(\mu-1)]^b},
\end{eqnarray*} 
and $\min_{i=0,\ldots,k}\esp\left[r(x^i)^2\right]\le\frac{1}{k+1}\sum_{i=0}^k\esp\left[r(x^i)^2\right]$.
\end{proof}

A near optimal oracle complexity is guaranteed in the next corollary of Theorem \ref{thm:rate:convergence}.
\begin{corollary}[Iteration and oracle complexities]\label{cor:oracle:complexity}
Let the assumptions of Theorem \ref{thm:rate:convergence} hold and set $N:=\mathcal{O}(d)$. Given $\epsilon>0$, \emph{\textsf{Algorithm \ref{algorithm:DSSA:extragradient}}} achieves the tolerance 
$$
\min_{i=0,\ldots,K}\esp[r(x^i)^2]\le\epsilon,
$$ 
after $K=b^{-1}\mathcal{O}(\epsilon^{-1})$ iterations and with a.s. an oracle complexity 
$\sum_{i=0}^K(1+\ell_i)N_i$ 
bounded above by
$$
b^{-2}\cdot\log_{\frac{1}{\theta}}\left(\frac{\hat\alpha\max_{i=0,\ldots,K}\widehat{L}_i}{(\lambda\theta)\wedge\hat\alpha}\right)\cdot\left[\ln\left(b^{-1}\epsilon^{-1}\right)\right]^{1+b}\cdot\mathcal{O}(d\epsilon^{-2}),
$$
where $\ell_k$ is the number of oracle calls used in the line search scheme \eqref{algo:armijo:rule} at  iteration $k$ and $\widehat{L}_k$ is defined in Lemma \ref{lemma:step:lower}. 
Moreover, the mean oracle complexity satisfies the same upper bound above with $\max_{i=0,\ldots,K}\widehat{L}_i$ replaced by $L$.
\end{corollary}
\begin{proof}
We recall the definitions in Assumption \ref{assumption:holder:continuity}, Theorem \ref{thm:variance:error:with:line:search}, Lemma \ref{lemma:decay:empirical:error}, Remark \ref{rem:constants:thm:correlated:error}, Proposition \ref{prop:A:armijo} and Remark \ref{rem:constants:A:armijo}. The definitions of $\widetilde L_2$, $\overline{L}_4$, $L_2^*$, $\mathsf{c}_2$ and $\mathsf{d}_2$ (which depend on $d$) and Theorem \ref{thm:rate:convergence} imply that, up to a constant $B>0$, for every $k\in\mathbb{N}$, $\min_{i=0,\ldots,k}\esp[r(x^i)^2]\le Bd(Nbk)^{-1}$. Given $\epsilon>0$, let $K$ be the least natural number such that $Bd(Nbk)^{-1}\le\epsilon$. Then $K=\mathcal{O}(dN^{-1}b^{-1}\epsilon^{-1})$, the total number of oracle calls is
\begin{eqnarray}
\sum_{i=0}^K(1+\ell_i)N_i&\lesssim &\sum_{i=1}^K N i(\ln i)^{1+b}\lesssim\left(\max_{i=0,\ldots,K}\ell_i\right)N K^2(\ln K)^{1+b}\nonumber\\
&\lesssim & \left(\max_{i=0,\ldots,K}\ell_i\right)N^{-1}d^2 b^{-2}\epsilon^{-2}\left[\ln\left(dN^{-1}b^{-1}\epsilon^{-1}\right)\right]^{1+b},\label{cor:oracle:complexity:eq1}
\end{eqnarray}
and $\min_{i=0,\ldots,K}\esp[r(x^i)^2]\le\epsilon$. Lemma \ref{lemma:step:lower} implies that $\ell_k\le\log_{\frac{1}{\theta}}\left(\frac{\hat\alpha\widehat{L}_k}{(\lambda\theta)\wedge\hat\alpha}\right)$. This fact, \eqref{cor:oracle:complexity:eq1} and $N=\mathcal{O}(d)$ imply the claimed bound on $\sum_{i=0}^K(1+\ell_i)N_i$. The concavity of $t\mapsto\log_{\frac{1}{\theta}}t$ and Jensen's inequality imply
\begin{eqnarray*}
\esp[\ell_k]\le\esp\left[\log_{\frac{1}{\theta}}\left(\frac{\hat\alpha\widehat{L}_k}{(\lambda\theta)\wedge\hat\alpha}\right)\right]\le \log_{\frac{1}{\theta}}\left(\frac{\hat\alpha L}{(\lambda\theta)\wedge\hat\alpha}\right),
\end{eqnarray*}
where we used that $\esp[\widehat{L}_k]=L$ by definitions of $\widehat{L}_k$ and $L$ and Assumption \ref{assump:iid:sampling}. The above relation, \eqref{cor:oracle:complexity:eq1} and $N:=\mathcal{O}(d)$ imply the claimed bound on the mean oracle complexity $\sum_{i=0}^K(1+\esp[\ell_i])N_i$. 
\end{proof}

\begin{remark}[Linear memory budget per operation]
Recall that $N:=\mathcal{O}(d)$ in Corollary \ref{cor:oracle:complexity}. This policy requires the computation of the sum \eqref{equation:empirical:average:DSSA:extragradient} of size $N_k\sim dk$ (up to logs) of $d$-dimensional vectors at iteration $k$. For large $d$, such computation is still cheap in terms of memory budget \emph{per operation}: the sum \eqref{equation:empirical:average:DSSA:extragradient} can be computed serially in $k$ steps, each one requiring the storage of just two $d$-dimensional vectors. Hence, it requires memory of $\mathcal{O}(d)$ per operation. It can also be easily parallelized.
\end{remark}

\begin{remark}[Radius estimate for unbounded $X$]
By Proposition \ref{prop:l2:boundedness}, the constant $\mathsf{J}$ in Theorem \ref{thm:rate:convergence} can be estimated by
\begin{equation}\label{equation:J}
\mathsf{J}\le\frac{\max_{k=0,\ldots,k_0}\Lnorm{\Vert x^k-x^*\Vert}^2+\frac{\phi\mathsf{C}_2\sigma_4(x^*)^2}{\mathsf{\overline{C}}_2\overline L_4^2}}{1-\phi}\lesssim\max_{k=0,\ldots,k_0}\Lnorm{\Vert x^k-x^*\Vert}^2+\frac{\sigma_4(x^*)^2}{\Lqrtnorm{\mathsf{L}(\xi)}^2},
\end{equation}
using the fact that $1-\phi\in(0,1)$ and the constant definitions in Assumption \ref{assumption:holder:continuity}, Theorem \ref{thm:variance:error:with:line:search}, Lemma \ref{lemma:decay:empirical:error} and Remarks \ref{rem:constants:thm:correlated:error} and \ref{rem:constants:A:armijo} with $p=2$. From \eqref{prop:l2:boundedness:k0} and \eqref{thm:rate:convergence:Nk}, $k_0$ in \eqref{equation:J} can be estimated by
\begin{equation}\label{thm:rate:convergence:k0}
k_0:=\left\lceil\exp\left\{\sqrt[b]{\frac{\mathsf{\overline{C}}_2\left(\hat\alpha\overline{L}_4\right)^2}{\phi b N}}\right\}-\mu+1\right\rceil.
\end{equation}
As discussed in Section \ref{section:conclusion}, the exponential dependence in \eqref{thm:rate:convergence:k0} is not a serious issue. Nevertheless, it can be improved to $k_0\lesssim\sqrt[a]{\mathsf{\overline{C}}_2\left(\hat\alpha\overline{L}_4\right)^2/(\phi b N)}-\mu$ if the sampling policy is taken as $N_k\sim N(k+\mu)^{1+a}(\ln(k+\mu))^{1+b}$ for some $a,b>0$. This come at the expense of an oracle complexity of $[\ln(\epsilon^{-1})]^{1+b}\mathcal{O}(\epsilon^{-(2+a)})$ which is polynomially near optimal (instead of logarithmically as in Corollary \ref{cor:oracle:complexity}).
\end{remark}

\begin{remark}[Boundedness in $\mathcal{L}^p$]\label{rem:lp:boundedness}
Adapting the proofs of Propositions \ref{prop:A:armijo} and \ref{prop:l2:boundedness}, it is possible to prove, in case $X$ is unbounded, that \emph{the sequence $\{x^k\}$ is $\mathcal{L}^p$-bounded} for any given $p\ge4$ satisfying Assumption \ref{assumption:holder:continuity}. This is a significant statistical stability property. The proof requires exploiting that $\Delta M_k(x^*)$ in \eqref{def:M:armijo} is a martingale difference.\footnote{The nonmartingale-like dependency is present only in $\Delta A_k$ via the error $\epsilon^k_3$ in \eqref{equation:oracle:errors:DSSA:extragradient}.}
\end{remark}

\section{Analysis of \textsf{Algorithm \ref{algorithm:DSSA:hyperplane}} for H\"older continuous operators}\label{section:algorithm:hyperplane:DSSA}

With respect to \textsf{Algorithm \ref{algorithm:DSSA:hyperplane}}, we will set $y^k:=x^k-\gamma_k\widehat F(\xi_k,z^k)$ and study the stochastic process $\{x^k\}$ with respect to the filtration
$$
\alg_k=\sigma(x^0,\xi^0,\ldots,\xi^{k-1}).
$$
We will replace Assumption \ref{assump:iid:sampling} by the following one.
\begin{assumption}[I.I.D. sampling]\label{assump:iid:sampling:hyperplane}
In \textsf{Algorithm \ref{algorithm:DSSA:hyperplane}}, the sequence $\{\xi^k_j:k\in\mathbb{N}_0,j\in[N_k]\}$ is an i.i.d. sample drawn from $\probn$ and $\sum_{k=0}^\infty N_k^{-\frac{1}{2}}<\infty$.
\end{assumption}

We also define the oracle errors:
\begin{eqnarray}
\bar\epsilon^k_{1} &:=&\widehat F(\xi^k,x^k)-T(x^k),\label{algo:noise:hyper:sample1}\\
\bar\epsilon^k_{2} &:=&\widehat F(\xi^k,z^k)-T(z^k),\label{algo:noise:hyper:sample2}\\
\bar\epsilon^k_{3} &:=&\widehat F(\xi^k,\widehat z^k)-T(z^k),\label{algo:noise:hyper:sample3}
\end{eqnarray}
where $\widehat z^k:=\bar z^k(\theta^{-1}\alpha_k)$ (see line search \eqref{algo:armijo:rule2} for the definition of $\bar z^k(\alpha)$). We remark that $\overline{\epsilon}^k_2$ and $\overline{\epsilon}^k_3$ are correlated errors in the sense that $z^k$ and $\widehat{z}^k$ are dependent on $\xi^k$. In the setting of Theorem \ref{thm:variance:error:with:line:search}, this means that $z^k=\overline z_{\beta_k}(\xi^k;\alpha_k,x^k)$ and $\widehat z^k=\overline z_{\beta_k}(\xi^k;\theta^{-1}\alpha_k,x^k)$. We start by showing the line search \eqref{algo:armijo:rule2} in \textsf{Algorithm \ref{algorithm:DSSA:hyperplane}} is well defined.
\begin{lemma}[Good definition of the line search]\label{lemma:armijo:hyper:def}
Consider Assumption \ref{assumption:holder:continuity}. Then
\begin{itemize}
\item[i)] The line search \eqref{algo:armijo:rule2} in \emph{\textsf{Algorithm \ref{algorithm:DSSA:hyperplane}}} terminates after a finite number of iterations.
\item[ii)] If \textsf{\emph{Algorithm \ref{algorithm:DSSA:hyperplane}}} does not stop at iteration $k+1$, then 
$
\left\langle\widehat F(\xi^k,z^k),x^k-z^k\right\rangle>0.
$
In particular, $\gamma_k>0$ in \eqref{algo:hyperplane2}. 
\end{itemize}
\end{lemma}
\begin{proof}
Item (ii) is a direct consequence of (i). 
We prove next item (i). Assume by contradiction that for every $\ell\in\mathbb{N}_0$,
$$
\left\langle \beta_k\widehat F\Big(\xi^k,z^k\left(\theta^{-\ell}\widehat\alpha\right)\Big),x^k-\Pi(g^k)\right\rangle<\lambda\Vert x^k-\Pi(g^k)\Vert^2.
$$
We let $\ell\rightarrow\infty$ above and by continuity of $\widehat F(\xi^k,\cdot)$, resulting from 
Assumption \ref{assumption:holder:continuity}, we obtain
$$
\lambda\Vert x^k-\Pi(g^k)\Vert^2\ge\langle x^k-g^k,x^k-\Pi(g^k)\rangle\ge\Vert x^k-\Pi(g^k)\Vert^2,
$$
using Lemma \ref{lemma:proj}(v) in the last inequality. Since we have $x^k\neq\Pi(g^k)$ by the  definition of the method, we obtain that $\lambda\ge1$, a contradiction.
\end{proof}
	
The following Lemma is also proved in the Appendix.
\begin{lemma}\label{lemma:recursion:hyper}
Consider Assumptions \ref{assump:existence}-\ref{assump:monotonicity} and \eqref{algo:noise:hyper:sample2}.
Suppose that \textsf{\emph{Algorithm \ref{algorithm:DSSA:hyperplane}}} does not stop at iteration $k+1$. Then, for all $x^*\in X^*$, 
\begin{equation*}
\Vert x^{k+1}-x^*\Vert^2\le\Vert x^{k}-x^*\Vert^2-\Vert y^k-x^k\Vert^2+2\gamma_k\langle\bar\epsilon^k_2,x^*-z^k\rangle.
\end{equation*}
\end{lemma}
	
We now aim at controlling the error term $\gamma_k\langle\bar\epsilon^k_2,x-z^k\rangle$. This term is not a martingale difference, since $z^k$ depends on $\xi^k$. We shall need the following lemma.

\begin{lemma}\label{lemma:gamma}
Suppose that \textsf{\emph{Algorithm \ref{algorithm:DSSA:hyperplane}}} does not stop at iteration $k+1$. Then
\begin{equation}
0<\gamma_k<\frac{\alpha_k\beta_k}{\lambda}\le\frac{\widehat\alpha\beta_k}{\lambda}.
\end{equation}
\end{lemma}
\begin{proof}
We only need to prove the second inequality. 
The line search \eqref{algo:armijo:rule2} and the fact that $x^k-z^k=\alpha_k(x^k-\Pi(g^k))$ imply that
\begin{equation}\label{lemma:gamma:eq1}
\langle \widehat F(\xi^k,z^k),x^k-z^k\rangle\ge\frac{\lambda}{\alpha_k\beta_k}\Vert
x^k-z^k\Vert^2.
\end{equation}
From \eqref{lemma:gamma:eq1} and the definition of $\gamma_k$ we get
\begin{eqnarray}\label{lemma:gamma:eq2}
\gamma_k =\frac{\langle \widehat F(\xi^k,z^k),x^k-z^k\rangle}{\Vert\widehat F(\xi^k,z^k)\Vert^2}
> \frac{\lambda}{\alpha_k\beta_k}\frac{\Vert x^k-z^k\Vert^2}{\Vert \widehat F(\xi^k,z^k)\Vert^2},
\end{eqnarray}
while the definition of $\gamma_k$ gives
\begin{eqnarray}\label{lemma:gamma:eq3}
\gamma_k =\frac{\langle \widehat F(\xi^k,z^k),x^k-z^k\rangle}{\Vert \widehat F(\xi^k,z^k)\Vert^2}
\le \frac{\Vert\widehat F(\xi^k,z^k)\Vert\Vert x^k-z^k\Vert}{\Vert \widehat F(\xi^k,z^k)\Vert^2}
= \frac{\Vert x^k-z^k\Vert}{\Vert \widehat F(\xi^k,z^k)\Vert},
\end{eqnarray}
using the Cauchy-Schwartz inequality. Inequalities \eqref{lemma:gamma:eq2}-\eqref{lemma:gamma:eq3} imply
the claim.
\end{proof}
	
\begin{lemma}[Error decay]\label{lemma:error:decay:hyper}
Consider Assumptions \ref{assumption:holder:continuity}, \ref{assump:existence} and \ref{assump:iid:sampling:hyperplane} and \eqref{algo:noise:hyper:sample2}. Suppose that \emph{\textsf{Algorithm \ref{algorithm:DSSA:hyperplane}}} does not stop at iteration $k+1$. Then, for all $x^*\in X^*$,
\begin{eqnarray*}
\Lpddnorm{\gamma_k\langle\bar\epsilon^k_2,x^*-z^k\rangle\big|\alg_k}\lesssim\frac{1+\Vert x^k-x^*\Vert^2}{\sqrt{N_k}}.
\end{eqnarray*}
%where the constant depends on $\tilde\beta\mathsf{c}_1\sigma_{2p}(x^*)$ and $\tilde\beta\overline{L}_{2p}$ (see Theorem \ref{thm:variance:error:with:line:search} and Remark \ref{rem:constants:thm:correlated:error}).
\end{lemma}
\begin{proof}
We denote $\widetilde z^k:=\Pi(g^k)$, so that
\begin{eqnarray}\label{lemma:error:decay:hyper:eq1}
x^*-z^k=\alpha_k(x^*-\widetilde z^k)+(1-\alpha_k)(x^*-x^k),
\end{eqnarray}
using the fact that $x^*=\alpha_kx^*+(1-\alpha_k)x^*$.  In view of
\eqref{lemma:error:decay:hyper:eq1},
we have
\begin{eqnarray}\label{lemma:error:decay:hyper:eq2}
\gamma_k\langle\bar\epsilon^k_2,x^*-z^k\rangle
&=&\gamma_k\alpha_k\langle\bar\epsilon^k_2,x^*-\widetilde z^k\rangle +
\gamma_k(1-\alpha_k)\langle\bar\epsilon^k_2,x^*-x^k\rangle\nonumber\\
&\le &\frac{\tilde\beta}{\lambda}\Vert\bar\epsilon^k_2\Vert\left(\Vert x^*-\widetilde z^k\Vert+\Vert x^*-x^k\Vert\right),
\end{eqnarray}
using the Cauchy-Schwarz inequality, Lemma \ref{lemma:gamma}, and the facts that 
$0<\alpha_k\le\hat\alpha\le1$ and $0<\beta_k\le\tilde\beta$.

Since $x^*\in X^*$, by Lemma \ref{lemma:proj}(iv), we use the fact that $x^*=\Pi[x^*-\beta_k T(x^*)]$ and the definitions of $\tilde z^k$, $g^k$ and $\bar\epsilon^k_1$ in order to obtain  
\begin{eqnarray}\label{lemma:error:decay:hyper:eq3}
\Vert\tilde z^k-x^*\Vert &=& \Vert \Pi[x^k-\beta_k (T(x^k)+\bar\epsilon^k_1)]-\Pi[x^*-\beta_k T(x^*)]\Vert\nonumber\\ 
&\le &\Vert x^k-x^* +\beta_k(T(x^*)-T(x^k))-\beta_k\bar\epsilon^k_1\Vert\nonumber\\
&\le &\Vert x^k-x^*\Vert+\tilde\beta L\Vert x^k-x^*\Vert^\delta+\tilde\beta\Vert\bar\epsilon^k_1\Vert,
\end{eqnarray}
using Lemma \ref{lemma:proj}(iii) in the first inequality, and the fact that $0<\beta_k\le\tilde\beta$ together with Lemma \ref{lemma:holder:continuity:mean:std:dev} in the last inequality. 

Using \eqref{lemma:error:decay:hyper:eq2}-\eqref{lemma:error:decay:hyper:eq3} and  the fact that
$\Vert x^k-x^*\Vert^\delta\le 1+\Vert x^k-x^*\Vert$, we take $\Lpddnorm{\cdot|\alg_k}$ and get 
\begin{eqnarray}
\Lpddnorm{\gamma_k\langle\bar\epsilon^k_2,x^*-z^k\rangle\big|\alg_k}\le
\left[\tilde\beta L+(2+\tilde\beta L)\Vert x^k-x^*\Vert\right]\frac{\tilde\beta}{\lambda}\Lpddnorm{\Vert\bar\epsilon^k_2\Vert\big|\alg_k}
+\frac{\tilde\beta^2}{\lambda}\Lpddnorm{\Vert\bar\epsilon^k_1\Vert\Vert\bar\epsilon^k_2\Vert\big|\alg_k},\nonumber\\
\label{lemma:error:decay:hyper:eq4}
\end{eqnarray}
using the fact that $x^k\in\alg_k$. By Lemma \ref{lemma:decay:empirical:error} with $q=p$ and the facts that $x^k\in\alg_k$ and $\xi^k\perp\perp\alg_k$, we get
\begin{equation}\label{lemma:error:decay:hyper:eq5}
\Lpnorm{\Vert\bar\epsilon^k_1\Vert\big|\alg_k}\le C_{p}\frac{\sigma_p(x^*)+L_p+L_p\Vert x^k-x^*\Vert}{\sqrt{N_k}},
\end{equation}
where we used the fact that $\Vert x^k-x^*\Vert^\delta\le1+\Vert x^k-x^*\Vert$. By Theorem \ref{thm:variance:error:with:line:search}, \eqref{algo:noise:hyper:sample2} and the facts that $z^k=\overline{z}_{\beta_k}(\xi^k;\alpha_k,x^k)$, $x^k\in\alg_k$ and $\alpha_k\in(0,1]$, we get  
\begin{equation}\label{lemma:error:decay:hyper:eq6}
\Lpddnorm{\Vert\bar\epsilon^k_2\Vert\big|\alg_k}\le\Lpnorm{\Vert\bar\epsilon^k_2\Vert\big|\alg_k}\lesssim \frac{\sigma_{2p}(x^*)+\Vert x^k-x^*\Vert}{\sqrt{N_k}},
\end{equation}
where we used the fact that $\delta\vee\Vert x^k-x^*\Vert^\delta\le1+\Vert x^k-x^*\Vert$. Invoking H\"older's inequality, we also get
\begin{equation}\label{lemma:error:decay:hyper:eq7}
\Lpddnorm{\Vert\bar\epsilon^k_1\Vert\Vert\bar\epsilon^k_2\Vert\big|\alg_k}\le
\Lpnorm{\Vert\epsilon^k_1\Vert\big|\alg_k}\cdot\Lpnorm{\Vert\epsilon^k_2\Vert\big|\alg_k}.
\end{equation}
Relations \eqref{lemma:error:decay:hyper:eq4}-\eqref{lemma:error:decay:hyper:eq7} prove the claim.
\end{proof}

\begin{proposition}[Stochastic quasi-Fej\'er property]\label{prop:fejer:hyper}
Consider Assumptions \ref{assumption:holder:continuity}, \ref{assump:existence}-\ref{assump:monotonicity} and \ref{assump:iid:sampling:hyperplane}. Assume that \emph{\textsf{Algorithm \ref{algorithm:DSSA:hyperplane}}} generates an infinite sequence $\{x^k\}$. Then
\begin{itemize}
\item[(i)] For all $x^*\in X^*$, there exists $c(x^*)\ge1$  such that, for all $k\in\mathbb{N}$,
\begin{eqnarray*}
\esp\big[\Vert x^{k+1}-x^*\Vert^2\big|\alg_k\big]\le \Vert x^{k}-x^*\Vert^2-\esp\big[\Vert y^k-x^k\Vert^2\big|\alg_k\big]+c(x^*)\frac{1+\Vert x^k-x^*\Vert^2}{\sqrt{N_k}}.
\end{eqnarray*}
\item[(ii)] A.s.
$\{\Vert x^k-x^*\Vert\}$ and $\{\dist(x^k,X^*)\}$ converge 
for all $x^*\in X^*$. 
In particular, $\{x^k\}$ is a.s.-bounded.
\item[(iii)]  A.s. if a cluster point of $\{x^k\}$ belongs to $X^*$ then $\lim_{k\rightarrow\infty}\dist(x^k,X^*)=0$. 
\end{itemize}
\end{proposition}
\begin{proof}
i) It is an immediate consequence of Lemmas \ref{lemma:recursion:hyper}, \ref{lemma:error:decay:hyper} and the fact that $x^k\in\alg_k$, after taking $\esp[\cdot|\alg_k]$ in Lemma \ref{lemma:recursion:hyper}.

ii) Set
$
\mathsf{c}_k(x^*):=\frac{c(x^*)}{\sqrt{N_k}}.
$
From (i),  for all $k\in\mathbb{N}_0$, 
\begin{equation}\label{prop:fejer:hyper:eq1}
\esp\big[\Vert x^{k+1}-x^*\Vert^2\big|\alg_k\big]\le\left[1+\mathsf{c}_k(x^*)\right]\Vert x^k-x^*\Vert^2+\mathsf{c}_k(x^*).
\end{equation}
By Assumption \ref{assump:iid:sampling:hyperplane}, we have $\sum_k\mathsf{c}_k(x^*)<\infty$. 
Hence, from \eqref{prop:fejer:hyper:eq1} and Theorem \ref{thm:rob} we conclude that a.s. $\{\Vert x^k-x^*\Vert\}$ 
converges and, in particular, $\{x^k\}$ is bounded. 

Set $\bar x^k:=\Pi_{X^*}(x^k)$. Relation \eqref{prop:fejer:hyper:eq1} and the fact that $x^k\in\alg_k$ imply 
\begin{equation}\label{eebb}
\esp\big[\dist(x^{k+1},X^*)^2\big|\alg_k\big]\le\left[1+\mathsf{c}_k(\bar x^k)\right]\dist(x^{k},X^*)^2+\mathsf{c}_k(\bar x^k).
\end{equation}
The boundedness of $\{\bar x^k\}$ and Assumption \ref{assump:iid:sampling:hyperplane} imply that a.s. $\sum_k\mathsf{c}_k(\bar x^k)<\infty$. Hence, Theorem \ref{thm:rob} and \eqref{eebb} imply that $\{\dist(x^k,X^*)\}$ a.s.-converges.

iii) Suppose that a.s. there exists $\bar x\in X^*$ and  a subsequence $\{k_\ell\}$ such that 
$\lim_{\ell\rightarrow\infty}\Vert x^{k_\ell}-\bar x\Vert=0$. Clearly, 
$\dist(x^{k_\ell},X^*)\le\Vert x^{k_\ell}-\bar x\Vert$ a.s., and therefore it follows that 
$\lim_{\ell\rightarrow\infty}\dist(x^{k_\ell},X^*)=0$. 
By (ii), $\{\dist(x^{k},X^*)\}$ a.s.-converges and hence $\lim_{k\rightarrow\infty}\dist(x^{k},X^*)=0$. 
\end{proof}
		
We now prove asymptotic convergence of \textsf{Algorithm \ref{algorithm:DSSA:hyperplane}}.
\begin{theorem}[Asymptotic convergence]\label{thm:convergence:hyper}
Under Assumptions \ref{assumption:holder:continuity}, \ref{assump:existence}-\ref{assump:monotonicity} and \ref{assump:iid:sampling:hyperplane}, either \emph{\textsf{Algorithm \ref{algorithm:DSSA:hyperplane}}} stops at iteration $k+1$, in which case $x^k$ is a solution of \emph{VI}$(T,X)$, or it generates an infinite sequence $\{x^k\}$ that a.s. is bounded and such that $\lim_{k\rightarrow\infty}\dist(x^k,X^*)=0$. In particular, a.s. every cluster point of $\{x^k\}$ belongs to $X^*$.
\end{theorem}
\begin{proof}
If \textsf{Algorithm \ref{algorithm:DSSA:hyperplane}} stops at iteration $k$, then 
$
x^k=\Pi[x^k-\beta_k\widehat F(\xi^k,x^k)].
$
From this fact and Lemma \ref{lemma:proj}(iv) we have
\begin{equation}\label{thm:convergence:hyper:eq1}
\langle \widehat F(\xi^k,x^k),x-x^k\rangle\ge0,\quad\quad\forall x\in X.
\end{equation}
From Assumption \ref{assump:iid:sampling:hyperplane}, \eqref{equation:expected:valued:objective} and the facts that $x^k\in\alg_k$ and $\xi^k\perp\perp\alg_k$, we 
get $\esp[\widehat F(\xi^k,x^k)|\alg_k]$ $=T(x^k)$. Using this equality and the fact that $x^k\in\alg_k$, 
we take $\esp[\cdot|\alg_k]$ in \eqref{thm:convergence:hyper:eq1} and obtain
$\langle T(x^k),x-x^k\rangle\ge0$, for all $x\in X$. Hence $x^k\in X^*$.

We now suppose that the sequence $\{x^k\}$ is infinite. By Proposition \ref{prop:fejer:hyper}(iii), 
it is sufficient to show that a.s. the bounded sequence $\{x^k\}$ has a cluster point in $X^*$. Choose any $x^*\in X^*$. As in Proposition \ref{prop:fejer:hyper}, set
$
\mathsf{c}_k(x^*):=\frac{c(x^*)}{\sqrt{N_k}}.
$
Using the property that $\esp[\esp[\cdot|\alg_k]]=\esp[\cdot]$, we take the expectation in Proposition \ref{prop:fejer:hyper}(i), 
and get, for all $k\in\mathbb{N}_0$,
\begin{equation}\label{eecc}
\esp\big[\Vert x^{k+1}-x^*\Vert^2\big|\alg_k\big]\le\left[1+\mathsf{c}_k(x^*)\right]\esp\left[\Vert x^{k}-x^*\Vert^2\right]-
\esp\left[\Vert y^k-x^k\Vert^2\right]+\mathsf{c}_k(x^*).
\end{equation}
From the fact that $\sum_k\mathsf{c}_k(x^*)<\infty$ (Assumption \ref{assump:iid:sampling:hyperplane}), \eqref{eecc} and Theorem \ref{thm:rob} we conclude that
\begin{equation}\label{thm:convergence:hyper:eq2}
\sum_{k=0}^\infty\esp\left[\Vert y^k-x^k\Vert^2\right]<\infty,
\end{equation}
and that $\left\{\esp\left[\Vert x^k-x^*\Vert^2\right]\right\}$ converges. 
In particular, $\left\{\esp\left[\Vert x^k-x^*\Vert^2\right]\right\}$ is a bounded sequence.

By the definition of \textsf{Algorithm \ref{algorithm:DSSA:hyperplane}}, we have that
$
\Vert y^k-x^k\Vert^2=\langle T(z^{k})+\bar\epsilon^{k}_2,x^{k}-z^{k}\rangle^2
\Vert T(z^{k})+\bar\epsilon^{k}_2\Vert^{-2}.
$
Hence, from \eqref{thm:convergence:hyper:eq2} we get
\begin{equation}\label{thm:convergence:hyper:eq3} 
\lim_{k\rightarrow\infty}\esp\Bigg[\frac{\langle T(z^{k})+\bar\epsilon^{k}_2,x^{k}-z^{k}\rangle^2}{\Vert T(z^{k})+\bar\epsilon^{k}_2\Vert^2}\Bigg]=0.
\end{equation}

From the definitions of $\{\bar\epsilon^k_1,\bar\epsilon^k_2,\bar\epsilon^k_3\}$ in \eqref{algo:noise:hyper:sample1}-\eqref{algo:noise:hyper:sample3}, 
Lemma \ref{lemma:decay:empirical:error} with $q=p=2$, Theorem \ref{thm:variance:error:with:line:search}(i) and the facts that $z^k=\overline z_{\beta_k}(\xi^k;\alpha_k,x^k)$ and $\widehat z^k=\overline z_{\beta_k}(\xi^k;\theta^{-1}\alpha_k,x^k)$, the property that $\esp[\esp[\cdot|\alg_k]]=\esp[\cdot]$ and the boundedness of $\left\{\esp\left[\Vert x^k-x^*\Vert^2\right]\right\}$, we get 
$$
\esp\left[\Vert\bar\epsilon^k_s\Vert^2\right]\lesssim\frac{\sup_{k\in\mathbb{N}_0}\esp\left[\Vert x^k-x^*\Vert^2\right]+1}{N_k},
$$
for $s\in\{1,2,3\}$ and all $k\in\mathbb{N}_0$. Since $\lim_{k\rightarrow\infty}N_k^{-1}=0$ (Assumption \ref{assump:iid:sampling:hyperplane}), we have in particular that, for $s\in\{1,2,3\}$,
\begin{equation}\label{thm:convergence:hyper:eq4}
\lim_{k\rightarrow\infty}\esp[\Vert\bar\epsilon^k_s\Vert^2]=0.
\end{equation}
Since $\mathcal{L}^2$-convergence implies a.s.-convergence along a subsequence, 
from \eqref{thm:convergence:hyper:eq3}-\eqref{thm:convergence:hyper:eq4}, we may take a (deterministic) 
subsequence $\{k_\ell\}_{\ell=1}^\infty$ such that a.s. for $s\in\{1,2,3\}$,
\begin{eqnarray}
\lim_{\ell\rightarrow\infty}\frac{\alpha_{k_\ell}\langle T(z^{k_\ell})+\bar\epsilon^{k_\ell}_2,x^{k_\ell}-\Pi(g^{k_\ell})\rangle}{\Vert T(z^{k_\ell})+\bar\epsilon^{k_\ell}_2\Vert}=0,\label{thm:convergence:hyper:eq5}\\
\lim_{\ell\rightarrow\infty}\bar\epsilon^{k_\ell}_s=0,\label{thm:convergence:hyper:eq6}
\end{eqnarray}
using the fact that $x^k-z^k=\alpha_k[x^k-\Pi(g^k)]$. Since $\beta_k\in[\hat\beta,\tilde\beta]$ with $\hat\beta>0$, 
we may refine $\{k_\ell\}$ if necessary so that, for some $\beta>0$,
\begin{equation}\label{thm:convergence:hyper:eq7}
\lim_{\ell\rightarrow\infty}\beta_{k_\ell}=\beta.
\end{equation}

From Proposition \ref{prop:fejer:hyper}(ii), the a.s.-boundedness of the sequence $\{x^{k_\ell}\}$ 
implies that, on a set $\Omega_1$ of total probability, there exists a (random) subsequence  
$\mathfrak{N}\subset\{k_{\ell}\}_{\ell=1}^\infty$ such that
\begin{equation}\label{thm:convergence:hyper:eq8}
\lim_{k\in\mathfrak{N}}x^{k}=x^*,\\
\end{equation}	
for some (random) $x^*\in\re^d$. Using the fact that $g^k=x^k-\beta_k[T(x^k)+\bar\epsilon^k_1]$, 
\eqref{thm:convergence:hyper:eq6}-\eqref{thm:convergence:hyper:eq8} and the continuity of $T$ and $\Pi$, 
for the event $\Omega_1$, we have
\begin{equation}\label{thm:convergence:hyper:eq9}
g^*:=\lim_{k\in\mathfrak{N}}g^k=x^*-\beta T(x^*).
\end{equation}
Also, for the event $\Omega_1$, from the definition of $z^k$ in \eqref{algo:hyperplane1}, the fact that 
$\alpha_k\in(0,1]$, \eqref{thm:convergence:hyper:eq6} and \eqref{thm:convergence:hyper:eq8}-\eqref{thm:convergence:hyper:eq9},
we get that $\{T(z^k)+\bar\epsilon^k_2\}_{k\in\mathfrak{N}}$ is bounded so that, since \eqref{thm:convergence:hyper:eq5}, we obtain
\begin{equation}\label{thm:convergence:hyper:eq10}
\lim_{k\in\mathfrak{N}}\alpha_{k}\langle T(z^{k})+\bar\epsilon^{k}_2,x^{k}-\Pi(g^{k})\rangle=0.
\end{equation}

We now consider two cases for the event $\Omega_1$.

\textbf{\textsf{Case (i)}}: $\lim_{k\in\mathfrak{N}}\alpha_{k}\neq 0$. In this case, we may refine $\mathfrak{N}$ 
if necessary, and find some (random) $\bar\alpha >0$ such that $\alpha_k\ge\bar\alpha$ 
for all $k\in\mathfrak{N}$. It follows from \eqref{thm:convergence:hyper:eq10} that on $\Omega_1$, 

\begin{equation}\label{thm:convergence:hyper:eq11}
\lim_{k\in\mathfrak{N}}\langle T(z^{k})+\bar\epsilon^{k}_2,x^{k}-\Pi(g^{k})\rangle
= 0.
\end{equation}
From \eqref{algo:armijo:rule2}-\eqref{algo:hyperplane1}, we get 
\begin{equation}\label{thm:convergence:hyper:eq12}
\langle T(z^{k})+\bar\epsilon^{k}_2,x^{k}-\Pi(g^{k})\rangle\ge
	\frac{\lambda}{\beta_{k}}\Vert x^{k}-\Pi(g^{k})\Vert^2
\ge\frac{\lambda}{\tilde\beta}\Vert x^{k}-\Pi(g^{k})\Vert^2
\end{equation}
for all $k$.
Relations \eqref{thm:convergence:hyper:eq11}-\eqref{thm:convergence:hyper:eq12} imply that, on $\Omega_1$,
\begin{equation}\label{thm:convergence:hyper:eq13}
0=\lim_{k\in\mathfrak{N}}\Vert x^{k}-\Pi(g^{k})\Vert.
\end{equation}
From \eqref{thm:convergence:hyper:eq8}-\eqref{thm:convergence:hyper:eq9}, we take limits in \eqref{thm:convergence:hyper:eq13} and obtain, by continuity of $\Pi$,
$$
0=\Vert x^*-\Pi[x^*-\beta T(x^*)]\Vert.
$$
Therefore, $x^* = \Pi[x^*-\beta T(x^*)]$,
so that $x^*\in X^*$ by Lemma \ref{lemma:proj}(iv). 

\textbf{\textsf{Case (ii)}}: $\lim_{k\in\mathfrak{N}}\alpha_{k}=0$. In this case we have
\begin{equation}\label{thm:convergence:hyper:eq14}
\lim_{{k\in\mathfrak{N}}}\theta^{-1}\alpha_k=0.
\end{equation}
Since $\widehat z^k:=\theta^{-1}\alpha_k\Pi(g^k)+(1-\theta^{-1}\alpha_k)x^k$ and $\{g^k\}_{{k\in\mathfrak{N}}}$ is bounded, we get from \eqref{thm:convergence:hyper:eq8} and \eqref{thm:convergence:hyper:eq14} that
\begin{equation}\label{thm:convergence:hyper:eq15}
\lim_{k\in\mathfrak{N}}\widehat z^{k}=x^*.
\end{equation}
Observe that, by the definition of the line search rule \eqref{algo:armijo:rule2} and \eqref{algo:noise:hyper:sample3}, we have 
\begin{equation}\label{thm:convergence:hyper:eq16}
\langle T(\widehat z^k)+\bar\epsilon^k_3,x^k-\Pi(g^k)\rangle <
\frac{\lambda}{\beta_k}\Vert x^k-\Pi(g^k)\Vert^2,
\end{equation}
for all $k\in\mathbb{N}_0$.
We take limit in \eqref{thm:convergence:hyper:eq16} along $\mathfrak{N}$, and we get, 
using the continuity of $T$ and $\Pi$ and relations \eqref{thm:convergence:hyper:eq6}-\eqref{thm:convergence:hyper:eq9} and \eqref{thm:convergence:hyper:eq15} that 
\begin{equation}\label{thm:convergence:hyper:eq17}
\langle T(x^*),x^*- \Pi(g^*)\rangle\le
\frac{\lambda}{\beta}\Vert x^*-\Pi(g^*)\Vert^2.
\end{equation}
Since the sequence $\{x^k\}$ is feasible and $X$ is closed, the limit point $x^*$ belongs to $X$. 
Thus, from \eqref{thm:convergence:hyper:eq17} and Lemma \ref{lemma:proj}(v), we get that, on $\Omega_1$,
\begin{equation}
\label{thm:convergence:hyper:eq18}
\lambda\Vert x^*-\Pi(g^*)\Vert^2\ge\beta\langle T(x^*),x^* -\Pi(g^*)\rangle
=\langle x^*-g^*,x^*-\Pi(g^*)\rangle
\ge\Vert x^*-\Pi(g^*)\Vert^2.
\end{equation}
Since $\lambda\in (0,1)$,   
\eqref{thm:convergence:hyper:eq18}
implies that
$\Vert x^*-\Pi(g^*)\Vert =0$. Hence, in view of \eqref{thm:convergence:hyper:eq9}, we have
$x^* = \Pi(x^* - \beta T(x^*))$.
By Lemma \ref{lemma:proj}(iv), we conclude that $x^*\in X^*$.

We have proved that on the event $\Omega_1$ of total probability, both in case (i) and in case (ii), 
$\{x^k\}$ has a cluster point which solves VI($T$,$X$). The claim follows from Proposition \ref{prop:fejer:hyper}(iii).
\end{proof}

\section{Discussion on the complexity constants of \textsf{Algorithm \ref{algorithm:DSSA:extragradient}}}\label{section:conclusion}
Suppose the oracle is \emph{exact}. In that case, \textsf{Algorithm \ref{algorithm:DSSA:extragradient}} would have essentially the same rate estimates, up to universal constants and a factor of $\mathcal{O}(\ln L)$ in the oracle complexity, either if a line search scheme is used or a CSP is used with a known Lipschitz constant (LC). The reason is that the Lipschitz continuity is only related to the \emph{smoothness class} of the operator. The situation is different when the oracle is \emph{stochastic}: the Lipschitz continuity also quantifies the \emph{spread of the oracle's error variance}.\footnote{This is true either for the martingale difference errors $\{\epsilon^k_i\}_{i=1,2}$ or the correlated error $\epsilon^k_3$ in \eqref{equation:oracle:errors:DSSA:extragradient}. The Lipschitz continuity in the analysis of $\epsilon^k_3$ is crucial in our chaining and self-normalization arguments of Lemmas \ref{lemma:lnorm:process} and \ref{lemma:error:decay:empirical:process}.} Consequently, the lack of knowledge of the LC is much more demanding in the stochastic case. It is instructive to compare the complexity constants when the LC is known or not. In the following, we recall the rate of convergence of Theorem \ref{thm:rate:convergence} and the constants defined in Assumption \ref{assumption:holder:continuity}, Theorem \ref{thm:variance:error:with:line:search}, Lemma \ref{lemma:decay:empirical:error}, Remarks \ref{rem:constants:thm:correlated:error} and \ref{rem:constants:A:armijo} and Proposition \ref{prop:A:armijo} with $p=2$.

Suppose first the LC is known. This was already considered in \cite{iusem:jofre:oliveira:thompson2017} under a more general condition than Assumption \ref{assumption:holder:continuity}. However, it leads to weaker complexity constants as argued in the following.\footnote{See Assumption 3.8 in \cite{iusem:jofre:oliveira:thompson2017}. Differently than Lemma \ref{lemma:holder:continuity:mean:std:dev}, it allows the multiplicative noise to depend on the reference point $x^*\in X^*$.} It is possible to show that if the stronger but fairly general condition of Lemma \ref{lemma:holder:continuity:mean:std:dev} holds and $\hat\alpha=\mathcal{O}(\frac{1}{L_2})$, then the rate statement of Theorem \ref{thm:rate:convergence} and the estimates \eqref{equation:J}-\eqref{thm:rate:convergence:k0} are valid when we replace $\sigma_4(x^*)$ by $\sigma_2(x^*)$, $\overline{L}_4$ by $L_2$ and\footnote{Up to universal constants, $\mathsf{C}_2$ and $\mathsf{\overline{C}}_2$ are unchanged.} the coefficient $(1-6\lambda^2)[(\lambda\theta)\wedge\hat\alpha]$ by a term of order $1-\mathcal{O}(1)(\hat\alpha L_2)^2$. Since $\hat\alpha L_2\lesssim1$ we also have $\mathsf{C}_2\lesssim1$ and $\mathsf{\overline {C}}_2\lesssim1$. Assuming $L_2$ is known, we obtain a property not satisfied by the estimates in \cite{iusem:jofre:oliveira:thompson2017}: $k_0$ in \eqref{thm:rate:convergence:k0} is \emph{independent of the oracle's error variances $\{\sigma_2(x)^2\}_{x\in X}$ over $X$} and there exist $b$, $N$ and $\mu$ and policy $\hat\alpha=\mathcal{O}(\frac{1}{L_2})$ such that $k_0:=0$. It is then possible to obtain the rate
\begin{eqnarray}
\min_{i=0,\ldots,k}\esp\left[r(x^i)^2\right]
\lesssim\frac{L_2^2\Vert x^0-x^*\Vert^2+\sigma_2(x^*)^2}{k},\label{equation:rate:known:LC}
\end{eqnarray}
which depends only on the \emph{local} variance $\sigma_2(x^*)^2$ and the \emph{initial} iterate $x^0$. This can be seen as a \emph{variance localization property}. We note that the above rate is sharper than those obtained in \cite{iusem:jofre:oliveira:thompson2017} (See Section 3.4.1 in \cite{iusem:jofre:oliveira:thompson2017}.)\footnote{In \cite{iusem:jofre:oliveira:thompson2017}, given $x^*\in X^*$, the rate is of the order of $\sigma(x^*)^4\cdot\max_{0\le i\le k_0(x^*)}\esp[\Vert x^i-x^*\Vert^2]$, where $k_0(x^*)\in\mathbb{N}_0$ depends on $\sigma(x^*)$. See Assumption 3.8 in \cite{iusem:jofre:oliveira:thompson2017} for the definition of $\sigma(x^*)$.}

Consider now the more challenging regime when the LC is unknown. As expected, the constants in the rate of Theorem \ref{thm:rate:convergence} are less sharp then the ones in \eqref{equation:rate:known:LC}. First, \eqref{equation:rate:known:LC} is not explicitly dependent on the dimension $d$. In terms of dimension, the rate in Theorem \ref{thm:rate:convergence} is of $\mathcal{O}(\frac{d}{N})$ and, thus, it is valid in the large sample regime $N:=\mathcal{O}(d)$. This is a manifestation of our need to treat correlated errors when using a line search scheme. Such scheme is an inner statistical estimator for the LC. Second, if we set $\mathsf{M}:=(\hat\alpha\Lqrtnorm{\mathsf{L}(\xi)})^2$, then the constants in the rate of Theorem \ref{thm:rate:convergence} satisfy $\mathsf{C}_2\lesssim \frac{\mathsf{M}}{N}$, $\mathsf{\overline{C}}_2\lesssim\textsf{M}$ and $\frac{(\hat\alpha\widetilde L_2)^2\mathsf J}{N}\lesssim\mathsf{M}^2\mathsf{J}$, for a general\footnote{The given order of dependence on $\mathsf{M}$ for an unbounded $X$ is an artifact of our proof techniques. We believe a sharper dependence can be obtained via more sophisticated concentration inequalities (instead of moment inequalities).} $X$ and $\mathsf{C}_2\lesssim1$, $\mathsf{\overline{C}}_2\lesssim1$ and $\frac{(\hat\alpha\widetilde L_2)^2\mathsf{J}}{N}\lesssim \mathsf{M}\diam(X)^2$, for a compact $X$. Observe that a line search scheme can only estimate a \emph{lower bound} for $\Lqrtnorm{\mathsf{L}(\xi)}$. For a large $\hat\alpha$, the lack of an \emph{upper bound} leads to a rate with larger constants when compared to \eqref{equation:rate:known:LC}. This is a manifestation of our absence of information of the LC. Note that robust methods are expected to have nonoptimal constants since the endogenous parameters are unknown \cite{nem:jud:lan:shapiro2009}. Third, note that \eqref{equation:rate:known:LC} only depends on the \emph{initial} iterate $x^0$. This is possible since $k_0$ in \eqref{thm:rate:convergence:k0} can be calibrated using the knowledge of the LC. For an unknown LC and for an unbounded $X$, $k_0$ depends on $\mathsf{M}$ but it still \emph{independent of the oracle's error moments $\{\sigma_4(x)\}_{x\in X}$ over $X$}. Differently than \eqref{equation:rate:known:LC}, for a large $\hat\alpha$ (implying a larger value for $\mathsf{M}$), the rate in Theorem \ref{thm:rate:convergence} will depend on $D_{k_0}^2(x^*):=\max_{k=0,\ldots k_0}\esp[\Vert x^k-x^*\Vert^2]$ for a possibly large $k_0$. Although not as sharp as \eqref{equation:rate:known:LC}, the resulted rate estimate for a large $k_0$ is not a limiting issue. It is still in accordance to, and in fact generalize, previous estimates which rely on compactness of $X$ (see e.g. \cite{nem:jud:lan:shapiro2009}): for a compact $X$, we have $\max_{k=0,\ldots k_0}\esp[\Vert x^k-x^*\Vert^2]\le\diam(X)^2$.

\section*{Appendix}

\begin{proof}[Proof of Lemma \ref{lemma:holder:continuity:mean:std:dev}]
By Jensen's inequality and Assumption \ref{assumption:holder:continuity} we get
\begin{eqnarray*}
\Vert T(x)-T(x_*)\Vert\le\esp\left[\Vert F(\xi,x)-F(\xi,x_*)\Vert\right]
\le\esp[\mathsf{L}(\xi)]\Vert x-y\Vert^\delta.
\end{eqnarray*}
Using this fact and definition \eqref{equation:oracle:error}, we get
\begin{eqnarray*}
\Lqnorm{\Vert\epsilon(\xi,x)\Vert}&\le &\Lqnorm{\Vert F(\xi,x)-F(\xi,x_*)\Vert}+\Lqnorm{\Vert F(\xi,x_*)-T(x_*)\Vert}+\Lqnorm{\Vert T(x)-T(x_*)\Vert}\nonumber\\
&\le &\Lqnorm{\mathsf{L}(\xi)\Vert x-x_*\Vert^\delta}+\Lqnorm{\Vert\epsilon(\xi,x_*)\Vert}+L\Vert x-x_*\Vert^\delta\nonumber\\
&=&\Lqnorm{\Vert\epsilon(\xi,x_*)\Vert}+\left(\Lqnorm{\mathsf{L}(\xi)}+L\right)\Vert x-y\Vert^\delta,
\end{eqnarray*}
where we used the triangle inequality for $\Vert\cdot\Vert$ and Minkowski's inequality for $\Lqnorm{\cdot}$. The claim is proved from the above fact,  \eqref{equation:oracle:error:variance} and $L_q=\Lqnorm{\mathsf{L}(\xi)}+L$.
\end{proof}

\begin{proof}[Proof of Lemma \ref{lemma:recursion:armijo}]
By \eqref{algo:extragradient:armijo1}-\eqref{algo:extragradient:armijo2}, we invoke twice Lemma \ref{lemma:proj}(i) with $v:=\alpha_k\widehat F(\xi^k,x^k)$, $x:=x^k$ and $z:=z^k$ and with $v:=\alpha_k\widehat F(\eta^k,z^k)$, $x:=x^k$ and $z:=x^{k+1}$, obtaining, for all $x\in X$, 
\begin{eqnarray}
2\langle\alpha_k\widehat F(\xi^k,x^k),z^k-x\rangle &\le &\Vert x^k-x\Vert^2-\Vert z^k-x\Vert^2-\Vert z^k-x^k\Vert^2,\label{equation:lemma:recursion:second:eq1}\\
2\langle\alpha_k\widehat F(\eta^k,z^k),x^{k+1}-x\rangle &\le &\Vert x^k-x\Vert^2-\Vert x^{k+1}-x\Vert^2-\Vert x^{k+1}-x^k\Vert^2.\label{equation:lemma:recursion:second:eq2}
\end{eqnarray}

We now set $x:=x^{k+1}$ in \eqref{equation:lemma:recursion:second:eq1} and sum the obtained relation with \eqref{equation:lemma:recursion:second:eq2} eliminating $\Vert x^k-x^{k+1}\Vert^2$. We thus get, for all $x\in X$,
\begin{eqnarray*}
\mathsf{I}&:=&2\langle\alpha_k\widehat F(\xi^k,x^k),z^k-x^{k+1}\rangle+2\langle\alpha_k\widehat F(\eta^k,z^k),x^{k+1}-x\rangle\\
&\le &\Vert x^k-x\Vert^2-\Vert x^{k+1}-x\Vert^2-\Vert z^k-x^{k+1}\Vert^2-\Vert z^k-x^k\Vert^2.
\end{eqnarray*}
Using definitions \eqref{equation:oracle:error}, \eqref{equation:empirical:mean:operator:&:error} and \eqref{equation:oracle:errors:DSSA:extragradient}, we have
\begin{eqnarray*}
\mathsf{I}&=&2\alpha_k\langle\widehat F(\xi^k,x^k)-\widehat F(\eta^k,z^k),z^k-x^{k+1}\rangle+2\langle\alpha_k\widehat F(\eta^k,z^k),z^k-x\rangle\\
&=&2\alpha_k\langle\widehat F(\xi^k,x^k)-\widehat F(\eta^k,z^k),z^k-x^{k+1}\rangle+2\alpha_k\langle T(z^k),z^k-x\rangle+2\alpha_k\langle\epsilon^k_2,z^k-x\rangle,
\end{eqnarray*}
The two previous relations imply that, for all $\in X$, 
\begin{eqnarray}
2\alpha_k\langle T(z^k),z^k-x\rangle &\le &2\alpha_k\langle\widehat F(\eta^k,z^k)-\widehat F(\xi^k,x^k),z^k-x^{k+1}\rangle+2\alpha_k\langle\epsilon^k_2,x-z^k\rangle\nonumber\\
&&+\Vert x^k-x\Vert^2-\Vert x^{k+1}-x\Vert^2-\Vert z^k-x^{k+1}\Vert^2-\Vert z^k-x^k\Vert^2\nonumber\\
&\le &2\alpha_k\Vert\widehat F(\eta^k,z^k)-\widehat F(\xi^k,x^k)\Vert\Vert z^k-x^{k+1}\Vert
+2\alpha_k\langle\epsilon^k_2,x-z^k\rangle\nonumber\\
&&+\Vert x^k-x\Vert^2-\Vert x^{k+1}-x\Vert^2-\Vert z^k-x^{k+1}\Vert^2-\Vert z^k-x^k\Vert^2\nonumber\\
&\le &2\alpha_k^2\Vert\widehat F(\eta^k,z^k)-\widehat F(\xi^k,x^k)\Vert^2+2\alpha_k\langle\epsilon^k_2,x-z^k\rangle\nonumber\\
&&+\Vert x^k-x\Vert^2-\Vert x^{k+1}-x\Vert^2-\Vert z^k-x^k\Vert^2,\label{equation:lemma:recursion:second:eq3}
\end{eqnarray}
where we used Cauchy-Schwartz in second inequality and Lemma \ref{lemma:proj}(iii) with \eqref{algo:extragradient:armijo1}-\eqref{algo:extragradient:armijo2} in the third inequality. 

Concerning the first term in the rightmost expression in \eqref{equation:lemma:recursion:second:eq3}, we have
\begin{eqnarray}
\alpha_k^2\Vert\widehat F(\eta^k,z^k)-\widehat F(\xi^k,x^k)\Vert^2 &\le &
3\alpha_k^2\Vert\widehat F(\xi^k,z^k)-\widehat F(\xi^k,x^k)\Vert^2\nonumber\\
&+&3\alpha_k^2\Vert\widehat F(\eta^k,z^k)-T(z^k)\Vert^2+3\alpha_k^2\Vert\widehat F(\xi^k,z^k)-T(z^k)\Vert^2\nonumber\\
&\le &3\lambda^2\Vert z^k-x^k\Vert^2+3\hat\alpha^2\Vert\epsilon^k_2\Vert^2+3\hat\alpha^2\Vert\epsilon^k_3\Vert^2,\label{equation:lemma:recursion:second:eq4}
\end{eqnarray}
using triangle inequality and the fact that $(\sum_{i=1}^3a_i)^2\le3\sum_{i=1}a_i^2$ in the first inequality and the line search \eqref{algo:armijo:rule} and definitions in \eqref{equation:oracle:error}, \eqref{equation:empirical:mean:operator:&:error} and \eqref{equation:oracle:errors:DSSA:extragradient} in the last inequality.

From $z^k=\Pi[x^k-\alpha_k(T(x^k)+\epsilon^k_1)]$ and Lemma \ref{lemma:residual:decrease} with $\alpha_k\in(0,1]$, we also have
\begin{eqnarray}	
\alpha_k^2 r(x^k)^2&\le & r_{\alpha_k}(x^k)^2\nonumber\\
&=&\Vert x^k-\Pi[x^k-\alpha_kT(x^k)]\Vert^2\nonumber\\
&\le &2\Vert x^k-z^k\Vert^2+2\Vert\Pi[x^k-\alpha_k(T(x^k)+\epsilon^k_1)]-\Pi[x^k-\alpha_k T(x^k)]\Vert^2\nonumber\\
&\le &2\Vert x^k-z^k\Vert^2+2\hat\alpha^2\Vert\epsilon^k_1\Vert^2\label{equation:lemma:recursion:second:eq5},
\end{eqnarray}
where we used Lemma \ref{lemma:proj}(iii) in the second inequality. The claim is proved using relations \eqref{equation:lemma:recursion:second:eq3}-\eqref{equation:lemma:recursion:second:eq5} with $x:=x^*$, for a given $x^*\in X^*$, definitions \eqref{def:A:armijo}-\eqref{def:M:armijo} and the facts that $0<1-6\lambda^2<1$ (see \textsf{Algorithm \ref{algorithm:DSSA:extragradient}}) and $\langle T(z^k),z^k-x^*\rangle\ge0$, which follows from $\langle T(x^*),z^k-x^*\rangle\ge0$ (since $x^*\in X^*$) and Assumption \ref{assump:monotonicity}.
\end{proof}

\begin{proof}[Proof of Proposition \ref{prop:A:armijo}]
First, we obtain a bound on $\Vert z^k-x^*\Vert$ similar to \eqref{lemma:error:decay:emp:eq1} in the proof of Theorem \ref{thm:variance:error:with:line:search} and then take $\Lpnorm{\cdot|\alg_k}$. Indeed, using the facts that $z^k=z(\xi^k;\alpha_k,x^k)$, $\epsilon^k_1=\widehat\epsilon(\xi^k,x^k)$ and $x^k\in\alg_k$, we obtain
\begin{equation}\label{prop:A:armijo:eq1}
\Lpnorm{\Vert z^k-x^*\Vert|\alg_k}\le(1+L\hat\alpha)\Vert x^k-x^*\Vert+\hat\alpha\Lpnorm{\Vert\epsilon^k_1\Vert|\alg_k}.
\end{equation}
Lemma \ref{lemma:decay:empirical:error} with $q=p$, \eqref{equation:oracle:errors:DSSA:extragradient} and the facts that $x^k\in\alg_k$ and $\xi^k\perp\perp\alg_k$ imply that
\begin{equation}\label{prop:A:armijo:eq2}
\Lpnorm{\Vert\epsilon^k_1\Vert|\alg_k}\le C_p\frac{\sigma_p(x^*)+L_p\Vert x^k-x^*\Vert}{\sqrt{N_k}}.
\end{equation}
Lemma \ref{lemma:decay:empirical:error} with $q=p$, \eqref{equation:oracle:errors:DSSA:extragradient} and the facts that $z^k\in\widehat{\alg}_k$, $\eta^k\perp\perp\widehat{\alg}_k$ and $\Lpnorm{\Lpnorm{\cdot|\widehat{\alg}_k}|\alg_k}=\Lpnorm{\cdot|\alg_k}$ imply that
\begin{equation}\label{prop:A:armijo:eq3}
\Lpnorm{\Vert\epsilon^k_2\Vert|\alg_k}=\Lpnorm{\Lpnorm{\Vert\epsilon^k_2\Vert\big|\widehat{\alg}_k}\Big|\alg_k}\le C_p\frac{\sigma_p(x^*)+L_p\Lpnorm{\Vert z^k-x^*\Vert|\alg_k}}{\sqrt{N_k}}.
\end{equation}

Finally, Theorem \ref{thm:variance:error:with:line:search}(i), \eqref{equation:oracle:errors:DSSA:extragradient}, Assumption \ref{assump:iid:sampling}, $0<\alpha_k\le\hat\alpha\le1$ and the facts that $z^k=z(\xi^k;\alpha_k,x^k)$, $x^k\in\alg_k$ and $\xi^k\perp\perp\alg_k$ imply that 
\begin{equation}\label{prop:A:armijo:eq4}
\Lpnorm{\Vert\epsilon^k_3\Vert|\alg_k}=\Lpnorm{\left\Vert\widehat\epsilon\left(\xi^k,z(\xi^k;\alpha_k,x^k)\right)\right\Vert\big|\alg_k}\le \frac{\mathsf{c}_1\sigma_{2p}(x^*)+\overline{L}_{2p}\Vert x^k-x^*\Vert}{\sqrt{N_k}}.
\end{equation}

The required claim is proved by putting together relations \eqref{def:A:armijo},  \eqref{prop:A:armijo:eq1}-\eqref{prop:A:armijo:eq4} and using the facts that $\Lpddnorm{a^2|\alg_k}=\Lpnorm{a|\alg_k}^2$, $(a+b)^2\le2a^2+2b^2$, $\overline{L}_{2p}>L_pC_p$, $\mathsf{c}_1>C_p$ (as defined in Assumption \ref{assumption:holder:continuity}, Theorem \ref{thm:variance:error:with:line:search}, Lemma \ref{lemma:decay:empirical:error} and Remark \ref{rem:constants:thm:correlated:error}) and $\sigma_{2p}(x^*)\ge\sigma_p(x^*)$.

The proof for the case $X$ is compact is analogous but replacing \eqref{prop:A:armijo:eq1} by the facts that $\Vert x^k-x^*\Vert\le\diam(X)$ and $\Vert z^k-x^*\Vert\le\diam (X)$ and replacing \eqref{prop:A:armijo:eq4} by the bound of Theorem \ref{thm:variance:error:with:line:search}(ii). 
\end{proof}
\begin{remark}[Constants of Proposition \ref{prop:A:armijo}]\label{rem:constants:A:armijo}
Recall definitions in Assumption \ref{assumption:holder:continuity}, \textsf{Algorithm \ref{algorithm:DSSA:extragradient}}, Theorem \ref{thm:variance:error:with:line:search}, Lemma \ref{lemma:decay:empirical:error} and Remark \ref{rem:constants:thm:correlated:error}. Let $\mathsf{G}_{p}:=\sup_k\frac{C_pL_p\hat\alpha}{\sqrt{N}_k}$. The constants in Proposition \ref{prop:A:armijo} are given, for a general $X$, by\footnote{For simplicity we do not explore the decay with $N_k^{-1}$ in the mentioned constants.} 
\begin{eqnarray*}
\mathsf{C}_{p}:=2\mathsf{c}_1^2\left[6\left(1+\mathsf{G}_{p}\right)^2+7-6\lambda^2\right],\quad\quad
\mathsf{\overline C}_{p}:=2\left[6\left(1+L\hat\alpha+\mathsf{G}_{p}\right)^2+7-6\lambda^2\right].
\end{eqnarray*}
For a compact $X$, the constants are $\mathsf{C}_p:=(26-12\lambda^2)C_p^2$ and $\mathsf{\overline{C}}_p:=26-12\lambda^2$.
\end{remark}

\begin{proof}[Proof of Lemma \ref{lemma:recursion:hyper}]
By Lemma \ref{lemma:armijo:hyper:def}(ii), we have that $\gamma_k>0$. Thus
\begin{eqnarray}
\Vert x^{k+1}-x^*\Vert^2 & = & \Vert \Pi(y^k)-x^*\Vert^2\nonumber\\
&\le & \Vert y^k-x^*\Vert^2-\Vert y^k-\Pi(y^k)\Vert^2\nonumber\\
&\le & \Vert y^k-x^*\Vert^2\nonumber\\
&=& \Vert (x^k-x^*)-\gamma_k(T(z^k)+\bar\epsilon^k_2)\Vert^2\nonumber\\
&=& \Vert x^k-x^*\Vert^2+\gamma_k^2\Vert T(z^k)+\bar\epsilon^k_2\Vert^2
-2\gamma_k\langle T(z^k)+\bar\epsilon^k_2,x^k-x^*\rangle,\label{lemma:recursion:hyper:eq1}
\end{eqnarray}
using Lemma \ref{lemma:proj}(ii) in the first inequality.  
Concerning the last term in the rightmost expression of \eqref{lemma:recursion:hyper:eq1}, we have
\begin{eqnarray}
-2\gamma_k\langle T(z^k)+\bar\epsilon^k_2,x^k-x^*\rangle 
&=& -2\gamma_k\langle T(z^k)+\bar\epsilon^k_2,x^k-z^k\rangle+\nonumber\\
&&2\gamma_k\langle T(z^k),x^*-z^k\rangle+2\gamma_k\langle\bar\epsilon^k_2,x^*-z^k\rangle
\nonumber\\
&=& -2\gamma_k(\gamma_k\Vert T(z^k)+\bar\epsilon^k_2\Vert^2)\nonumber\\
&&+2\gamma_k\langle T(z^k),x^*-z^k\rangle+2\gamma_k\langle\bar\epsilon^k_2,x^*-z^k\rangle\nonumber\\
&\le & -2\gamma_k^2\Vert T(z^k)+\bar\epsilon^k_2\Vert^2+2\gamma_k\langle\bar\epsilon^k_2,x^*-z^k\rangle,\label{lemma:recursion:hyper:eq2}
\end{eqnarray}
using
the definition of $\gamma_k$ 
in the second equality,   
and the facts that
$\gamma_k>0$ and $\langle T(z^k),x^*-z^k\rangle\le0$ (which follows from 
the pseudo-monotonicity of $T$, and the facts $x^*\in X^*$, $z^k\in X$) 
in the inequality. 
Combining \eqref{lemma:recursion:hyper:eq1}-\eqref{lemma:recursion:hyper:eq2} we get
\begin{eqnarray}
\Vert x^{k+1}-x^*\Vert^2 
&\le & \Vert x^k-x^*\Vert^2+\gamma_k^2\Vert T(z^k)+\bar\epsilon^k_2\Vert^2
-2\gamma_k^2\Vert T(z^k)+\bar\epsilon^k_2\Vert^2+
2\gamma_k\langle\bar\epsilon^k_2,x^*-z^k\rangle\nonumber\\
&=& \Vert x^k-x^*\Vert^2-\gamma_k^2\Vert T(z^k)+\bar\epsilon^k_2\Vert^2+ 2\gamma_k\langle\bar\epsilon^k_2,x^*-z^k\rangle\nonumber\\
&=& \Vert x^k-x^*\Vert^2-\Vert y^k-x^k\Vert^2+ 2\gamma_k\langle\bar\epsilon^k_2,x^*-z^k\rangle,\label{lemma:recursion:hyper:eq3}
\end{eqnarray}
using the fact that $\Vert y^k-x^k\Vert=\gamma_k\Vert T(z^k)+\epsilon^k_2\Vert$ (which follows from the 
definition of $\gamma_k$), in the last equality.
\end{proof}

\end{document}

%% file: Extragradient_Armijo_SIAM_shared.tex
\usepackage{amsmath}
\usepackage{latexsym}
\usepackage{amssymb}
\usepackage{lipsum}
\usepackage{amsfonts}
\usepackage{graphicx}
\usepackage{epstopdf}
\usepackage{algorithmic}
\usepackage{lmodern}
\newlength\longest
\usepackage{hyperref}
\ifpdf
  \DeclareGraphicsExtensions{.eps,.pdf,.png,.jpg}
\else
  \DeclareGraphicsExtensions{.eps}
\fi

\newsiamremark{remark}{Remark}
\newsiamremark{assumption}{Assumption}
\newsiamremark{example}{Example}

\newcommand{\esp}{\mathbb{E}}
\newcommand{\prob}{\mathbb{P}}
\newcommand{\probn}{\mathbf{P}}
\newcommand{\var}{\mathbb{V}}
\newcommand{\alg}{\mathcal{F}}

\newcommand{\unit}{\mathsf{1}}

\newcommand{\re}{\mathbb{R}}

\DeclareMathOperator*{\epi}{epi}

\newcommand{\vertiii}[1]{{\left\vert\kern-0.4ex #1 
\kern-0.4ex\right\vert}}
\newcommand{\Lpnorm}[1]{\vertiii{\,#1\,}_{p}}
\newcommand{\Ldpnorm}[1]{\vertiii{\,#1\,}_{2p}}
\newcommand{\Lpddnorm}[1]{\vertiii{\,#1\,}_{\frac{p}{2}}}
\newcommand{\Lqnorm}[1]{\vertiii{\,#1\,}_{q}}

\newcommand{\Lnorm}[1]{\vertiii{\,#1\,}_{2}}
\newcommand{\Lqrtnorm}[1]{\vertiii{\,#1\,}_{4}}
\DeclareMathOperator*{\dist}{d}
\DeclareMathOperator*{\diam}{\mathcal{D}}
\DeclareMathOperator*{\argmin}{argmin}

\numberwithin{theorem}{section}

\newcommand{\TheTitle}{Variance-based extragradient methods with line search for stochastic variational inequalities} 
\newcommand{\TheAuthors}{A. Iusem, A. Jofr\'e, R. I. Oliveira, and P. Thompson}

\headers{Variance-based extragradient methods with line search}{\TheAuthors}

\title{{\TheTitle}\thanks{Submitted to the editors DATE.
}}

\author{
Alfredo N. Iusem\thanks{Instituto de Matem\'atica Pura e Aplicada (IMPA), 
Rio de Janeiro, RJ, Brazil.
(\email{iusp@impa.br}).}
\and
Alejandro Jofr\'e\thanks{Centro de Modelamiento Matem\'atico (CMM \& DIM), 
Santiago, Chile.
(\email{ajofre@dim.uchile.cl}).}
\and
Roberto I. Oliveira\thanks{Instituto de Matem\'atica Pura e Aplicada (IMPA), 
Rio de Janeiro, RJ, Brazil.
(\email{rimfo@impa.br}). Roberto I. Oliveira's work was supported by a {\em Bolsa de Produtividade em Pesquisa} from CNPq, Brazil. His work in this article is part of the activities of FAPESP Center for Neuromathematics (grant \#2013/07699-0, FAPESP - S. Paulo Research Foundation).}
\and
Philip Thompson\thanks{Instituto de Matem\'atica Pura e Aplicada (IMPA), 
Rio de Janeiro, RJ, Brazil.
(\email{philip@impa.br}). Philip Thompson's work was supported by a CNPq Doctoral scholarship while he was a PhD student at IMPA with visit appointments at CMM. His work was conducted at IMPA and CMM.}
}